\documentclass[12pt,a4paper,final]{amsart}
\usepackage[marginratio=1:1]{geometry}

\usepackage[T1]{fontenc}
\usepackage{ae,aecompl}
\usepackage[utf8]{inputenc}
\usepackage{hyperref}
\usepackage{enumitem}
\usepackage{graphicx}
\usepackage[backend=biber,
	maxnames=99,
	minalphanames=4,
	maxalphanames=4,
	url=false,
	isbn=true,
	backref=true,
	citestyle=alphabetic,
	bibstyle=alphabetic,
	autocite=inline,
	sorting=nty,]{biblatex}

\usepackage{xcolor}
\hypersetup{
	colorlinks,
	linkcolor={red!80!black},
	citecolor={blue!50!black},
	urlcolor={blue!80!black}
}

\usepackage{amsfonts}
\usepackage{amssymb}
\usepackage{tikz-cd}
\usepackage{float}
\usepackage{multicol}
\makeatletter
\providecommand*{\twoheadrightarrowfill@}{%
	\arrowfill@\relbar\relbar\twoheadrightarrow
}
\providecommand*{\twoheadleftarrowfill@}{%
	\arrowfill@\twoheadleftarrow\relbar\relbar
}
\providecommand*{\xtwoheadrightarrow}[2][]{%
	\ext@arrow 0579\twoheadrightarrowfill@{#1}{#2}%
}
\providecommand*{\xtwoheadleftarrow}[2][]{%
	\ext@arrow 5097\twoheadleftarrowfill@{#1}{#2}%
}
\makeatother

\newcommand{\C}{{\mathfrak C}}
\newcommand{\M}{{\mathcal M}}
\newcommand{\N}{{\mathbb{N}}}
\newcommand{\R}{{\mathbb{R}}}
\newcommand{\restr}{\mathord{\upharpoonright}}
\newcommand{\EZ}{\mathrel{ { {\mathbb E}_0 } } }
\newcommand{\Er}{\mathrel{E}}
\newcommand{\lang}{{\mathcal L}}
\newcommand{\catg}{{\mathcal C}}
\newcommand{\powerset}{{\mathcal P}}

\newcommand{\biglor}{\bigvee}

\newcommand{\sbmid}{\mid}
\newcommand{\fcolon}{\colon}

\DeclareMathOperator{\tp}{{tp}}
\DeclareMathOperator{\Th}{{Th}}
\DeclareMathOperator{\gal}{{Gal}}
\DeclareMathOperator{\cl}{{cl}}
\DeclareMathOperator{\Id}{{Id}}
\DeclareMathOperator{\id}{{id}}
\DeclareMathOperator{\aut}{{Aut}}
\DeclareMathOperator{\autf}{{Autf}}
\DeclareMathOperator{\CLO}{{CLO}}
\DeclareMathOperator{\dom}{{dom}}
\DeclareMathOperator{\Souslin}{{\mathcal A}}

\newtheorem{thm}{Theorem}[section]
\newtheorem{conj}[thm]{Conjecture}
\newtheorem{ques}[thm]{Question}

\newtheorem{lem}[thm]{Lemma}
\newtheorem{fct}[thm]{Fact}
\newtheorem{cor}[thm]{Corollary}
\newtheorem{prop}[thm]{Proposition}

\theoremstyle{remark}
\newtheorem{rem}[thm]{Remark}
\theoremstyle{definition}
\newtheorem{dfn}[thm]{Definition}

\newtheorem*{clm*}{Claim}
\newtheorem{ex}[thm]{Example}
\newcounter{claimcounter}[thm]
\newenvironment{clm}{\stepcounter{claimcounter}{\noindent {\textbf{Claim}} \theclaimcounter:}}{}
\newenvironment{clmproof}[1][\proofname]{\proof[#1]}{\endproof}

\title{Topological dynamics and the complexity of strong types}
\author{Krzysztof Krupi\'nski}
\email[K.\ Krupi\'{n}ski]{kkrup@math.uni.wroc.pl}
\address[K.\ Krupi\'nski]{
	Instytut Matematyczny, Uniwersytet Wrocławski\\
	pl. Grunwaldzki 2/4\\
	50-384 Wrocław, Poland
}
\thanks{The first author is supported by the Narodowe Centrum Nauki grants no. 2012/07/B/ST1/03513, 2015/19/B/ST1/01151, and 2016/22/E/ST1/00450}

\author{Anand Pillay}
\email[A.\ Pillay]{apillay@nd.edu}
\address[A.\ Pillay]{Department of Mathematics, University of Notre Dame\\
	281 Hurley Hall\\
	Notre Dame, IN 46556, USA}
\thanks{The second author is supported by NSF grant DMS-1360702.}

\author{Tomasz Rzepecki}
\email[T.\ Rzepecki]{tomasz.rzepecki@math.uni.wroc.pl}
\address[T.\ Rzepecki]{
	Instytut Matematyczny, Uniwersytet Wrocławski\\
	pl. Grunwaldzki 2/4\\
	50-384 Wrocław, Poland
}
\thanks{The third author is supported by the Narodowe Centrum Nauki grant no. 2015/17/N/ST1/02322}

\keywords{Topological dynamics, Galois groups, strong types, Borel cardinality}

\subjclass[2010]{03C45, 54H20, 03E15, 54H11}

\date{}

\begin{document}
	
	\begin{abstract}
		We develop topological dynamics for the group of automorphisms of a monster model of any given theory. In particular, we find strong relationships between objects from topological dynamics (such as the generalized Bohr compactification introduced by Glasner) and various Galois groups of the theory in question, obtaining essentially new information about them, e.g.\ we present the closure of the identity in the Lascar Galois group of the theory as the quotient of a compact, Hausdorff group by a dense subgroup.
		
		We apply this to describe the complexity of bounded, invariant equivalence relations, obtaining comprehensive results, subsuming and extending the existing results and answering some open questions from earlier papers. We show that, in a countable theory, any such relation restricted to the set of realizations of a complete type over $\emptyset$ is type-definable if and only if it is smooth. Then we show a counterpart of this result for theories in an arbitrary (not necessarily countable) language, obtaining also new information involving relative definability of the relation in question. As a final conclusion we get the following trichotomy. Let $\C$ be a monster model of a countable theory, $p \in S(\emptyset)$, and $E$ be a bounded, (invariant) Borel (or, more generally, analytic) equivalence relation on $p(\C)$. Then, exactly one of the following holds:
		\begin{enumerate}
			\item
			$E$ is relatively definable (on $p(\C)$), smooth, and has finitely many classes,
			\item
			$E$ is not relatively definable, but it is type-definable, smooth, and has $2^{\aleph_0}$ classes,
			\item
			$E$ is not type definable and not smooth, and has $2^{\aleph_0}$ classes.
		\end{enumerate}
		All the results which we obtain for bounded, invariant equivalence relations carry over to the case of bounded index, invariant subgroups of definable groups.
	\end{abstract}

	\maketitle

	\section{Introduction}
	
	Generally speaking, this paper concerns applications of topological dynamics and the ``descriptive set theory'' of compact topological groups to model theory.
	
	The idea of using methods and tools of topological dynamics in the study of groups definable in first order structures originates in \cite{Ne1}.
	Since then further important developments in this direction have be made (see e.g.\ \cite{Ne2, GiPePi, KrPi, ChSi}). 
	The motivation for these considerations is the fact that using the ``language'' of topological dynamics, one can describe new interesting phenomena concerning various model-theoretic objects which lead to non-trivial results and questions in a very general context (sometimes without any assumption on the theory in question, sometimes under some general assumptions such as NIP).
	
	With a given definable group $G$, one can associate various connected components of it (computed in a big model, called a monster model). The quotients by these connected components are invariants of the group $G$ (in the sense that they do not depend on the choice of the monster model) and one of the important tasks is to understand these quotients as mathematical objects. Topological dynamics turns out to be an appropriate tool to do that. Already Newelski noticed some connections between notions from topological dynamics (mainly Ellis groups) and quotients by these components. This was investigated more deeply in \cite{KrPi}, which led to important new results on such quotients.
	
	In Section~\ref{section: top dyn for aut(C)} of the current paper, we adapt ideas and some proofs from \cite{KrPi} to the following context. We consider any complete theory $T$ and its monster model $\C$. We develop topological dynamics for the group $\aut(\C)$ (in place of the definable group $G$ considered in the above paragraph). Instead of quotients by connected components, we are now considering certain Galois groups of $T$, namely $\gal_L(T)$, $\gal_{KP}(T)$ and $\gal_0(T)$ 
	(the first group is called the Lascar Galois group, the second one -- 	the Kim-Pillay Galois group, and the third one is the kernel of the canonical epimorphism from $\gal_L(T)$ to $\gal_{KP}(T)$). These groups are very important invariants of the given theory. While $\gal_{KP}(T)$ is naturally a compact, Hausdorff group, $\gal_L(T)$ and $\gal_0(T)$ are more mysterious objects, and our results shed new light on them; in particular, we show that $\gal_L(T)$ is naturally the quotient of a compact, Hausdorff group by some normal subgroup, while $\gal_0(T)$ is such a quotient but by a dense, normal subgroup. All of this follows from our considerations relating topological dynamics of the group $\aut(\C)$ and the above Galois groups.
	
	Our original motivation for the above considerations was to say something meaningful about Galois groups of first order theories. Later, it turned out that as a non-trivial outcome of these considerations, we obtained very general results on the complexity of bounded, invariant equivalence relations which refine type (which are sometimes called strong types, or rather their classes are called strong types). Certain concrete strong types play a fundamental role in model theory, mainly: Shelah strong types (classes of the relation which is the intersection of all $\emptyset$-definable equivalence relations with finitely many classes), Kim-Pillay strong types (classes of the finest bounded, $\emptyset$-type-definable equivalence relation denoted by $E_{KP}$), and Lascar strong types (classes of the finest bounded, invariant equivalence relation denoted by $E_L$). While the quotients by bounded, type-definable equivalence relations are naturally compact, Hausdorff spaces (with the so-called logic topology), the quotients by bounded, invariant equivalence relations are not naturally equipped with such a nice topology
	(the logic topology on them is compact but not necessarily Hausdorff, and may even be trivial). Thus, a natural question is how to measure the complexity of bounded, invariant equivalence relations and how to view quotients by them as mathematical objects. One can, of course, just count the number of elements of these quotients, but more meaningful is to look at Borel cardinalities (in the sense of descriptive set theory) of such relations (the precise sense of this is explained in Section~\ref{section: preliminaries}). Important results in this direction have already been established for Lascar strong types in \cite{KPS} and \cite{KMS}, and later they were generalized in \cite{KaMi} and \cite{KrRz} to a certain wider class of bounded, $F_\sigma$ equivalence relations. A fundamental paper in this area, focusing on the number of elements in the quotient spaces, is \cite{Ne}.
	
	From the main results of \cite{KaMi} and \cite{KrRz} it follows that, working in a countable theory, smoothness (in the sense of descriptive set theory) of a bounded, $F_\sigma$ equivalence relation restricted to the set of realizations of a single complete type over $\emptyset$ and satisfying 
	an additional technical assumption (which we call orbitality) is equivalent to its type-definability. It was asked whether one can drop this extra assumption and also weaken the assumption that the relation is $F_\sigma$ to the one that it is only Borel. In Section~\ref{section: smoothness}, we prove a very general theorem which answers these questions. In a simplified form, it says that a bounded, invariant equivalent relation defined on the set of realizations of a single complete type over $\emptyset$ in a countable theory is smooth if and only if it is type-definable; in other words, such a relation is either type-definable, or non-smooth. It is worth emphasizing that this kind of a result was not accessible by the methods of \cite{KMS}, \cite{KaMi} or \cite{KrRz}, as they were based on a distance function coming from the fact that the relation in question was $F_\sigma$ in those papers. In Section~\ref{section: arbitrary language}, we prove a variant of this result for theories in
	an arbitrary (i.e.\ not necessarily countable) language; this time, however, we do not talk about smoothness, focusing only on the cardinality of quotient spaces, but with extra information concerning relative definability. All of this yields the trichotomy formulated at the end of the abstract, which is 
	a comprehensive result relating smoothness, type-definability, relative definability and the number of classes of bounded, Borel equivalence relations in a countable theory. This trichotomy appears in Section~\ref{section: trichotomy} in a more general form.
	
	It should be stressed that -- using the ``affine sort'' technique -- all the results we have obtained easily carry over to the case of subgroups of definable groups, mirroring what was done in \cite{KrRz} and \cite{KaMi}, essentially extending some results of these papers. 
	These new corollaries will be stated along with the main theorems.

	We finish the introduction with a description of the structure of this paper. First of all we should say that the main results are contained in Sections~\ref{section: top dyn for aut(C)},~\ref{section: smoothness},~\ref{section: arbitrary language} and~\ref{section: trichotomy}.
	
	In Section~\ref{section: preliminaries}, we define the fundamental notions and recall the key facts, and also make some basic observations.
	
	In Section~\ref{section: top dyn for aut(C)}, we develop some topological dynamics of the group $\aut(\C)$ of automorphisms of the monster model, focusing on relationships with Galois groups of the theory in question. As an outcome, we get new information on these Galois groups as well as on the ``spaces'' of strong types, which is then essentially used in Sections~\ref{section: smoothness} and~\ref{section: arbitrary language}. The main results of Section~\ref{section: top dyn for aut(C)} are Theorems~\ref{thm:main theorem 1},~\ref{thm:main theorem 2} and~\ref{thm:main theorem 3}.
	
	The appendix is an extension of Section~\ref{section: top dyn for aut(C)}. In particular, it explains why in Section~\ref{section: top dyn for aut(C)} (and thus in the whole paper) we have to work with the Ellis semigroup of the appropriate type space and not just with this type space itself.
	
	In Section~\ref{section: topological lemmas}, first we prove a general lemma concerning topological dynamics, and then we apply it to prove a technical lemma which is used later in the proofs of the main results of Sections~\ref{section: smoothness} and~\ref{section: arbitrary language}.
	In the second half, we prove a few other observations needed in Section~\ref{section: arbitrary language}.
	
	In Section~\ref{section: smoothness}, we prove our main result on smoothness and type-definability of bounded, invariant equivalence relations in the countable language case. This is Theorem~\ref{thm:main_Borel} which is formulated in a very general form and then followed by a collection of immediate corollaries, which are restrictions to more concrete situations and give answers to some questions from \cite{KaMi} and \cite{KrRz} discussed in the final part of Subsection~\ref{subsection: bounded relations}.
	
	Section~\ref{section: arbitrary language} deals with bounded, invariant equivalence relations in a language of arbitrary cardinality. The main result here is Theorem~\ref{thm:nwg}.
	We also explain some of the consequences and limitations of this theorem, and suggest and motivate Conjecture~\ref{conj:nwg2}, which would be a strengthening of part (I) of Theorem~\ref{thm:nwg}.
	
	Section~\ref{section: trichotomy} summarizes the main results of Sections~\ref{section: smoothness} and~\ref{section: arbitrary language} in the form of the aforementioned trichotomy theorem, along with a variant for definable groups.
	
	It is worth mentioning that after this paper was submitted, the third author made some further progress \cite{Rze17}. In the current paper, the equivalence of smoothness and type-definability for bounded, invariant equivalence relations defined on the set of realizations of a single complete type over $\emptyset$ is proved. One can still ask what happens if the relation in question is defined on a bigger set. It is easy to see that the assumption that the relation refines type is then needed.  Example 4.4 of \cite{KrRz} shows that even under this assumption, in general, smoothness does not imply type-definability. In \cite{Rze17}, the third author introduced a new class of {\em weakly orbital equivalence relations} (which contains invariant relations defined on a single complete type over $\emptyset$ as well as orbital relations, such as $E_L$, considered on the whole monster model), and proved that for such relations smoothness implies type-definability. This result generalizes  Theorem~\ref{thm:main_Borel}, but one should emphasize that the proof of this generalization uses Theorem~\ref{thm:main_Borel} and does not yield a new proof of Theorem~\ref{thm:main_Borel}.

	\section{Preliminaries}\label{section: preliminaries}

	\subsection{Topological dynamics}\label{subsection: topological dynamics}

	As a general reference for the knowledge on topological dynamics needed in this paper we would recommend \cite{Gl} and \cite{Au}. In this paper, by ``compact'' we mean what some may call ``quasicompact'', namely we do not include the Hausdorff property in the definition, and in fact we will explicitly state the separation properties satisfied by spaces in question.
	
	A {\em $G$-flow} is a pair $(G,X)$, where $G$ is a topological group acting continuously on a compact, Hausdorff space $X$.
	
	\begin{dfn}
		The {\em Ellis semigroup} of the flow $(G,X)$, denoted by $EL(X)$, is the closure of the collection of functions $\{\pi_g \sbmid g \in G\}$ (where $\pi_g: X \to X$ is given by $\pi_g(x)=gx$) in the space $X^X$ equipped with the product topology, with composition as the semigroup operation.
	\end{dfn}
	
	Since composition of functions in $X^X$ is continuous in the left coordinate,
	the semigroup operation on $EL(X)$ is also continuous in the left coordinate. Moreover, both $X^X$ and $EL(X)$ are $G$-flows, and minimal subflows of $EL(X)$ are exactly minimal left ideals with respect to the semigroup structure on $EL(X)$. We have the following fundamental fact proved by Ellis (e.g. see Corollary 2.10 and Propositions 3.5 and 3.6 of \cite{El}, or Proposition 2.3 of \cite{Gl}).
	
	\begin{fct}\label{Ellis theorem}
		Let $S$ be a semigroup equipped with a compact, Hausdorff topology so that the semigroup operation is continuous in the left coordinate.
		Let $\M $ be a minimal left ideal in $S$, and let $J(\M )$ be the set of all idempotents in $\M $. Then:
		\begin{enumerate}[label=\roman*), nosep]
			\item
			For any $p \in \M $, $Sp=\M p=\M $.
			\item
			$\M $ is the disjoint union of sets $u\M $ with $u$ ranging over $J(\M )$.
			\item
			For each $u \in J(\M )$, $u\M $ is a group with the
			identity element $u$, where the group operation is the restriction of the semigroup operation on $S$.
			\item
			All the groups $u\M $ (for $u \in J(\M )$) are isomorphic, even when we vary the minimal ideal $\M $.
		\end{enumerate}
	\end{fct}
	
	Applying this to $S:=EL(X)$, the isomorphism type of the groups $u\M $ (or just any of these groups) from the above fact is called the {\em Ellis group} of the flow $X$.
	
	A {\em $G$-ambit} is a $G$-flow $(G,X,x_0)$ with a distinguished point $x_0 \in X$ such that the orbit $Gx_0$ is dense. A {\em universal $G$-ambit} is an initial object in the category of all $G$-ambits, where morphisms are homomorphisms of $G$-ambits (i.e. continuous maps between pointed spaces, preserving the action of $G$). 
%It is a well-known and easy fact that a universal $G$-ambit always exists. 
	It is clear that a universal $G$-ambit always exists. Indeed, take a set $\{(G,X_i,x_i) : i \in I\}$ of representatives of isomorphism ``classes'' of all $G$-ambits, put $X:= \prod_i X_i$, $x:=(x_i)_i$, and let ${\mathcal U}$ be the closure of the orbit of $x$ with respect to the coordinatewise action of $G$ on $X$. Then $(G,{\mathcal U},x)$ is universal.  For example, in the case when $G$ is discrete, it is just $\beta G$ (the Stone-\v Cech compactification of $G$); in the category of externally definable $G$-ambits, it is the appropriate space of externally definable types (see \cite[Fact 1.10]{KrPi} for details). The universal $G$-ambit can be equipped with
	the structure of a left continuous semigroup which is isomorphic to its own Ellis semigroup, so, in fact, there is no need to work with the original definition of the Ellis semigroup for universal ambits. However, we will be considering the action of $\aut(\C)$ (where $\C$ is a monster model) on a certain space of global types on which, as we will see in the appendix, rather often there is no natural left continuous semigroup operation. Hence, we will have to really work with the original definition of the Ellis semigroup of our $\aut(\C)$-ambit.

	A very important notion for this paper is 
	the { \em $\tau$-topology} on an Ellis group. In \cite{Au,Gl}, it is defined on Ellis groups
	of $\beta G$ (for a discrete group $G$), but it can also be defined on Ellis groups of any flow $(G,X)$. To introduce the $\tau$-topology, we first need to define the so-called circle operation on subsets of $EL(X)$. Similarly to \cite{KrPi} (see the discussion at the beginning of Section 2 in \cite{KrPi}), although we do not have a continuous on the left ``action'' of the semigroup $EL(X)$ on $2^{EL(X)}$ (i.e.\ on the space of non-empty, closed subsets of $EL(X)$) extending the natural action of $G$, we can take the statement in point (1) of \cite[Chapter IX, Lemma 1.1]{Gl} as the definition of $\circ$.
	
	\begin{dfn}
		For $A \subseteq EL(X)$ and $p\in EL(X)$, $p \circ A$ is defined as the set of all points $\eta \in EL(X)$ for which there exist nets $(\eta_i)$ in $A$ and $(g_i)$ in $G$ such that $\lim g_i=p$ and $\lim g_i\eta_i=\eta$.
	\end{dfn}
	
	As was observed in \cite[Section 2]{KrPi},
	it is easy to check that $p\circ A$ is closed, $pA \subseteq p \circ A$ and $p\circ (q \circ A) \subseteq (pq)\circ A$ (but we do not know whether $p\circ (q \circ A) = (pq)\circ A$; in \cite{Gl}, it follows from the existence of the action of $\beta G$ on $2^{\beta G}$).
	
	Now, choose any minimal left ideal $\M $ in $EL(X)$ and an idempotent $u \in \M $.
	
	\begin{dfn}\label{Def: tau topology}
		For $A \subseteq u\M $, define $\cl_\tau (A) = (u \circ A) \cap u\M$.
	\end{dfn}
	
	Now, the proofs of 1.2-1.12 (except 1.12(2)) from \cite[Chapter IX]{Gl} go through (with some slight modifications) in our context.
	In particular, $\cl_\tau$ is a closure operator on subsets of $u\M $, and it induces the so-called {\em $\tau$-topology} on $u\M $ which is weaker than the topology inherited from $EL(X)$; the $\tau$-topology is compact and $T_1$, and multiplication is continuous in each coordinate separately. Also, the topological isomorphism type of $u\M$ depends on the choice of neither $\M$ nor $u \in J(\M)$: for the fact that it is independent of the choice of $u \in J(\M)$ see Lemma 1.4 in  \cite[Chapter IX]{Gl}; however, we could not find in the literature a proof that it is also independent from the choice of $\M$, so we briefly explain how to do that. 

Consider two minimal left ideals $\M$ and $\M'$, and an idempotent $u \in \M$. By Proposition 2.5 of \cite{Gl} or Proposition 3.6 of \cite{El}, there is an idempotent $u' \in \M'$ such that $uu' = u'$ and $u'u =u$. Then one easily checks that $f \fcolon u\M \to u'\M'$ given by $f(x)= xu'$ is an (abstract)  isomorphism with the inverse given by $f^{-1}(y)=yu$. So it is enough to show that $f$ is a closed map. Consider a $\tau$-closed subset $A$ of $u\M$. Then 
%$A=uu'A \subseteq u(u' \circ A) = u(u' \circ uA) \subseteq u(u' \circ (u \circ A)) \subseteq u(u'u \circ A) = u(u \circ A)=A$ (where the last equality is restatement of the assumption that $A$ is closed).
$A=u'u'A \subseteq  u'(u' \circ A) = uu'(u' \circ A) \subseteq u(u'\circ (u' \circ A)) \subseteq  u(u' \circ A) = u(u' \circ uA) \subseteq u(u' \circ (u \circ A)) \subseteq u(u'u \circ A) = u(u \circ A)=A$ (the last equality is equivalent to $\tau$-closedness of $A$). Hence, $u'(u' \circ A) =A$. One can easily check that $(u' \circ A)u' = u' \circ Au'$, so we conclude that $u'(u' \circ Au') =Au' = f[A]$, which means that $f[A]$ is $\tau$-closed in $u'\M'$.
	
	\begin{dfn}
		$H(u\M )$ is the intersection of the sets $\cl_\tau(V)$ with $V$ ranging over all $\tau$-neighborhoods of $u$ in the group $u\M $.
	\end{dfn}
	\begin{fct}[Theorem 1.9 in Chapter IX of \cite{Gl}]
		$H(u\M )$ is a $\tau$-closed, normal subgroup of $u\M $. The quotient group $u\M /H(u\M )$ equipped with the quotient topology induced by the $\tau$-topology is a compact, Hausdorff group (and this quotient topology will also be called the {\em $\tau$-topology}). For any $\tau$-closed subgroup $K$ of $u\M $, $u\M /K$ is a Hausdorff space if and only if $K \supseteq H(u\M )$.
	\end{fct}

	In \cite[Chapter IX]{Gl}, it is proved that in the case of the discrete group $G$, working in $\beta G$, the topological group $u\M /H(u\M )$ coincides with the so-called generalized Bohr compactification of $G$. In \cite{KrPi}, a similar result is proved in the category of externally definable objects. In fact, the proof of \cite[Theorem 2.5]{KrPi} can be adapted to show that, working in $EL(X)$, $u\M /H(u\M )$ is also the generalized Bohr compactification of $G$, but computed in the category of $G$-flows $(G,Y)$ such that for any $y_0 \in Y$ there is a homomorphism of $G$-flows from $EL(X)$ to $Y$ sending $\Id$ to $y_0$. But since this observation is not in the main stream of the current paper, we will not talk about the details.
	
	The key fact for us is that $u\M /H(u\M )$ is a compact, Hausdorff group.

	\subsection{Descriptive set theory}\label{subsection: descriptive set theory}
	
	Let $E$ and $F$ be equivalence relations on Polish spaces $X$ and $Y$, respectively. We say that {\em $E$ is Borel reducible to $F$} if there exists a Borel reduction of $E$ into $F$, i.e.\ a Borel function $f \fcolon X \to Y$ such that for all $x,y \in X$
	\[
		x\Er y \iff f(x)\mathrel{F} f(y).
	\]
	If $E$ is Borel reducible to $F$, we write $E\leq_{B}F$.
	
	We say that $E$ and $F$ are {\em Borel equivalent} or {\em Borel bi-reducible} or that they have the same {\em Borel cardinality}, symbolically $E \sim_B F$, if $E \leq_B F$ and $F \leq_B E$.
	
	$E$ is said to be {\em smooth} if it is Borel reducible to $\Delta_{2^{\N}}$, i.e.\ to equality on the Cantor set. Note that each smooth equivalence relation is automatically Borel (as the preimage of $\Delta_{2^{\N}}$ by a Borel function).
	
	The following two dichotomies are fundamental.
	
	\begin{fct}[Silver dichotomy]
		\label{fct:silver}
		For every Borel equivalence relation $E$ on a Polish space either $E \leq_B \Delta_{\N}$, or $\Delta_{2^{\N}} \leq_B E$.
	\end{fct}
	
	By $\EZ$ we denote the equivalence relation of eventual equality on $2^{\N}$.
	
	\begin{fct}[Harrington-Kechris-Louveau dichotomy]\label{fct:Harrington-Kechris-Louveau dichotomy}
		For every Borel equivalence relation $E$ on a Polish space either $E \leq_B \Delta_{2^{\N}}$ (i.e. $E$ is smooth), or ${\EZ} \leq_B E$.
	\end{fct}
	
	The definition of Borel cardinalities makes sense for non-Borel equivalence relations.
%and we will be using it later, particularly in Section~\ref{section: smoothness}. 
	However, one has to be careful here. While for Borel equivalence relations on Polish spaces non-smoothness implies possessing $2^{\aleph_0}$ classes
	(e.g.\ by Fact~\ref{fct:Harrington-Kechris-Louveau dichotomy}),
	there are non-Borel equivalence relations which are non-smooth and with only 2 classes (e.g. a partition of a Polish space into two non-Borel subsets).
	
	Recall that for an equivalence relation $E$ on a set $X$, a subset $Y$ of $X$ is said to be {\em $E$-saturated} if it is a union of some classes of $E$.  In this paper, we will say that a family $\{B_i \sbmid i \in \omega\}$ of subsets of $X$ {\em separates classes} of $E$ if for every $x \in X$, $[x]_E= \bigcap \{ B_i \sbmid x \in B_i\}$. Note that this implies that all $B_i$'s are $E$-saturated. Thus, a family  $\{B_i \sbmid i \in \omega\}$ of subsets of $X$ separates classes of $E$ if and only if each $B_i$ is $E$-saturated and each class of $E$ is the intersection of those sets $B_i$ which contain it. 
	The following is folklore.
	
	\begin{fct}\label{fct: separating family}
		Let $X$ be an equivalence relation on a Polish space $X$. Then, $E$ is smooth if and only if there is a countable family $\{B_i\sbmid i \in \omega\}$ of Borel ($E$-saturated) subsets of $X$ separating classes of $E$. 
%(meaning that each class of $E$ is the intersection of those sets $B_i$ which contain it).
	\end{fct}
	
	\begin{proof}
		Let $f$ be a Borel reduction of $E$ to $\Delta_{2^{\N}}$. Let $\{C_i\sbmid i \in \omega\}$ be a countable open basis of the space $2^{\N}$. Then $\{f^{-1}[C_i]\sbmid i \in \omega \}$ is a countable family consisting of Borel ($E$-saturated) subsets of $X$ separating classes of $E$.

		For the converse, consider a family $\{ B_i \sbmid i \in \N\}$ satisfying all the requirements. Define $f \colon X \to 2^{\N}$ by $f(x) = \chi_{\{i \in \N \sbmid x \in B_i\}}$ (i.e. the characteristic function of ${\{i \in \N \sbmid x \in B_i\}}$). It is easy to see that this is a Borel reduction of $E$ to $\Delta_{2^{\N}}$.
	\end{proof}

	\subsection{Model theory}\label{subsection: model theory}
	
	Let $T$ be a first order theory. We will usually work in a {\em monster model} $\C$ of $T$, which by definition is a {\em $\kappa$-saturated} (i.e.\ each type over an arbitrary set of parameters from $\C$ of size less than $\kappa$ is realized in $\C$) and {\em strongly $\kappa$-homogeneous} (i.e.\ any elementary map between subsets of $\C$ of cardinality less than $\kappa$ extends to an isomorphism of $\C$) model of $T$ for a ``sufficiently large''	strong limit cardinal $\kappa$. Then $\kappa$ is called the {\em degree of saturation} of $\C$. Recall that a monster model in this sense always exists \cite[Theorem 10.2.1]{Ho}. Whenever we talk about types or type-definable sets, we mean
	that they are defined over {\em small} (i.e.\ of cardinality less than $\kappa$) sets of parameters from $\C$; an exception are global types which by definition are complete types over $\C$.
	When we consider a product of sorts of $\C$, we assume that it is a product of a small (i.e.\ less than $\kappa$) number of sorts.
	Sometimes we will also work in a bigger monster model $\C' \succ \C$ whose degree of saturation $\kappa'$ is always assumed to be ``much'' bigger than the cardinality of $\C$.
	
	An {\em invariant} set is a subset of a product of sorts of $\C$ which is invariant under $\aut(\C)$; an {\em $A$-invariant} set is a subset invariant under $\aut(\C/A)$ (such a set is clearly a union of sets of realizations of some number of complete types over $A$). 
	
	We would like stress that in this paper ``type-definable'' means ``type-definable with parameters'' whereas ``invariant'' means ``invariant over $\emptyset$'' (unless otherwise specified).
	
	We say that $D$ is a {\em relatively definable subset} of a subset $C$ of a product of sorts if $D$ is an intersection of $C$ with a definable set.
	
	By $a,b,\dots$ we will denote (possibly infinite) tuples of elements from some sorts of $\C$; to emphasize that these are tuples, sometimes we will write $\bar a, \bar b,\dots$.
	
	Recall that $\bar a \equiv \bar b$ means that $\bar a$ and $\bar b$ have the same type over $\emptyset$. For a tuple $\bar a$ from $\C$ and a set of parameters $A$, by $S_{\bar a}(A)$ we denote the space of all types $\tp(\bar b/A)$ with $\bar b \equiv \bar a$. For an $A$-invariant subset $X$ of a product of sorts of $\C$, we define
	$X_A:=\{ \tp(\bar x/A)\sbmid \bar x \in X\}$.
	
	An invariant equivalence relation on a product of (an arbitrary small number $\lambda$ (i.e.\ $\lambda <\kappa$) of sorts of $\C$) is said to be {\em bounded} if it has less than $\kappa$ many classes (equivalently, at most $2^{|T|+\lambda}$ classes, which follows from the fact that the relation of having the same type over any given model refines any bounded invariant equivalence relation (see below)); we use the same definition for relations defined on invariant or type-definable subsets of products of sorts. If a bounded, invariant equivalence relation refines the relation of having the same type over $\emptyset$ (in short, refines type), we call its classes {\em strong types}. Recall that:
	
	\begin{itemize}
		\item $E_L$ is the finest bounded, invariant equivalence relation on a given product of sorts, and its classes are called {\em Lascar strong types},
		\item $E_{KP}$ is the finest bounded, $\emptyset$-type-definable equivalence relation on a given product of sorts, and its classes are called {\em Kim-Pillay strong types}.
	\end{itemize}
	
	Clearly $E_L$ refines $E_{KP}$. $E_L$ can be described as the transitive closure of the relation $\Theta(a,b)$ saying that $(a,b)$ begins an infinite indiscernible sequence, and also as the transitive closure of the relation saying that the elements have the same type over some small submodel of $\C$ (e.g. see \cite[Proposition 5.4]{KiPi} and \cite[Fact 1.13]{CLPZ01}). Recall that $\Theta(a,b)$ and the relation saying that the elements have the same type over some small submodel of $\C$ are both $\emptyset$-type-definable. The Lascar distance $d_L(a,b)$ is defined as the minimal number $n$ for which there are $a_0=a,a_1,\dots,a_n=b$ such that $\Theta(a_i,a_{i+1})$ holds for all $i$, if such a number $n$ exists, and otherwise it is $\infty$.
	
	\begin{dfn}
		Let $E$ be a bounded, invariant equivalence relation on a product $P$ of some sorts of $\C$. We define the {\em logic topology} on $P/E$ by saying that a subset $D \subseteq P/E$ is closed if its preimage in $P$ is type-definable.
	\end{dfn}
	
	It is well known that $P/E$ is compact, and if $E$ is type-definable, then $P/E$ is also Hausdorff \cite[Lemma 3.3]{LaPi}. The same remains true if we restrict $E$ to a type-definable subset of $P$. The next remark will be useful later.
	
	\begin{rem}\label{rem:type-definability_of_relations}
		If $E$ is an invariant equivalence relation defined on 
		a single complete type $[a]_{\equiv}$ over $\emptyset$, then $E$ has a type-definable [resp. relatively definable] class if and only if $E$ is type-definable [resp. relatively definable].
	\end{rem}

	\begin{proof}
		We prove the type-definable version; the relatively definable version is similar.
		The implication $(\Leftarrow)$ is obvious. For the other implication,
		without loss of generality $[a]_E$ is type-definable. Since $[a]_E$ is $a$-invariant, we get that it is type-definable over $a$, i.e.\ $[a]_E=\pi(\C,a)$ for some partial type $\pi(x,y)$ over $\emptyset$. Then, for any $b \equiv a$ we have $[b]_E=\pi(\C,b)$. Thus, $\pi(x,y)$ defines $E$.
	\end{proof}

	The following easy proposition seems to be new.
	
	\begin{prop}\label{prop: closure of alpha/E}
		If $E$ is a bounded, invariant equivalence relation defined on a single complete type $p$ over $\emptyset$, then for any $a \in p(\C)$ and $b/E \in \cl(a/E)$ one has $\cl(b/E)=\cl(a/E)$ (i.e.\ the logic topology on $\cl(a /E)$ is trivial). This implies that the closures of singletons in $p(\C)/E$ form a partition of $p(\C)/E$, 
%and the preimages of these closures are classes of the finest bounded, $\emptyset$-type-definable equivalence relation on $p(\C)$ coarsening $E$. 
		and the preimage of the equivalence relation on $p(\C)/E$ defined by $\cl(x)=\cl(y)$ is the finest bounded, $\emptyset$-type-definable equivalence relation on $p(\C)$ coarsening $E$.
	\end{prop}
	
	\begin{proof}
		By Zorn's Lemma and compactness of $p(\C)/E$, we can find a minimal nonempty, closed subset $D$ of $p(\C)/E$. Then, for any $c/E \in D$, $\cl(c/E)=D$.
		Now, for any $a \in p(\C)$ there is an automorphism $f \in \aut(\C)$ mapping $a$ to some $c$ such that $c/E \in D$, and so for any $b/E \in \cl(a/E)$ one has $\cl(b/E)=\cl(a/E)$. This clearly implies that the closures of singletons form a partition of $p(\C)/E$, and the final statement follows from the definition of the logic topology and Remark~\ref{rem:type-definability_of_relations}.
	\end{proof}
	
	Since it is known that $E_{KP}$ restricted to any complete type over $\emptyset$ is the finest bounded, $\emptyset$-type-definable equivalence relation on the set of realizations of this type \cite[Lemma 4.18]{LaPi}, the above proposition gives us the next corollary, whose last part answers a question asked by Domenico Zambella in conversation with the first author.
	
	\begin{cor}\label{cor: KP class is the closure of L class}
		For any $a$, $[a]_{E_{KP}}/E_L=\cl(a/E_L)$, and the logic topology on $[a]_{E_{KP}}/E_L$ is trivial. In particular, $[a]_{E_{KP}}$ is the smallest $E_L$-saturated, type-definable subset containing $[a]_{E_L}$.
	\end{cor}
	
	Now, we recall fundamental issues about Galois groups of first order theories. Good references for this knowledge are \cite{LaPi}, \cite{Zi}, and \cite{GiNe}.
	
	\begin{dfn}$\,$
		\begin{enumerate}[label=\roman{*}),nosep]
			\item
			{\em The group of Lascar strong automorphisms}, which is denoted by $\autf_L(\C)$, is the subgroup of $\aut(\C)$ which is generated by all automorphisms fixing small submodels of $\C$ pointwise, i.e.\ $\autf_L(\C)=\langle \sigma \sbmid \sigma \in \aut(\C/M)\;\, \mbox{for a small}\;\, M\prec \C\rangle$.
			\item
			{\em The Lascar Galois group of $T$}, which is denoted by $\gal_L(T)$, is the quotient group $\aut(\C)/\autf_L(\C)$ (which makes sense, as $\autf_L(\C)$ is a normal subgroup of $\C$).
		\end{enumerate}
	\end{dfn}

	Now, we are going to define a certain natural topology on $\gal_L(T)$. For details, the reader may consult Sections 4 and 5 of \cite{Zi}.
%For more details consult \cite{LaPi} and \cite{GiNe}.
	Let $\mu\fcolon \aut(\C) \to \gal_L(\C)$ be the quotient map. Choose a small model $M$, and let $\bar m$ be the tuple of all its elements. Let $\mu_1\fcolon \aut(\C) \to S_{\bar m}(M)$ be defined by $\mu_1(\sigma)=\tp(\sigma(\bar m)/M)$, and $\mu_2\fcolon S_{\bar m}(M) \to \gal_L(T)$ by $\mu_2(\tp(\sigma(\bar m)/M))=\sigma /\autf_L(\C)$. Then $\mu_2$ is a well-defined surjection, and $\mu=\mu_2 \circ \mu_1$. Thus, $\gal_L(T)$ becomes the quotient of the space $S_{\bar m}(M)$ by the relation of lying in the same fiber of $\mu_2$, and so we can define a topology on $\gal_L(T)$ as the quotient topology. In this way, $\gal_L(T)$ becomes a compact (but not necessarily Hausdorff) topological group. This topology does not depend on the choice of the model $M$. The topological group $\gal_L(T)$ does not depend (up to a topological isomorphism) on the choice of the monster model $\C$ in which it is computed (for $\gal_L(T)$ treated as an abstract group a proof can be found in Section 2 of \cite{Zi}, and in order to see that the isomorphism obtained there is a homeomorphism, use the definition of the topology on $\gal_L(T)$). 
	
	\begin{fct}\label{fct: characterization of topology on Gal_L(T)}
		The following conditions are equivalent for $C \subseteq \gal_L(T)$.
		\begin{enumerate}[label=\roman{*}),nosep,]
			\item
			$C$ closed.
			\item
			For every (possibly infinite) tuple $\bar a$ of elements of $\C$, the set $\{\sigma(\bar a)\sbmid \sigma \in \aut(\C)\;\, \mbox{and}\;\, \mu(\sigma) \in C\}$ is type-definable [over some [every] small submodel of $\C$].
			\item
			There are a tuple $\bar a$ and a partial type $\pi(\bar x)$ (with parameters) such that $\mu^{-1}[C]=\{ \sigma \in \aut(\C)\sbmid \sigma(\bar a) \models \pi(\bar x)\}$.
			\item
			For some tuple $\bar m$ enumerating a small submodel of $\C$, the set $\{\sigma(\bar m)\sbmid \sigma \in \aut(\C)\;\, \mbox{and}\;\, \mu(\sigma) \in C\}$ is type-definable [over some [any] small submodel of $\C$].
		\end{enumerate}
	\end{fct}

	\begin{proof}
		A part of this fact is contained in \cite[Lemma 4.10]{LaPi}. The rest is left as an exercise.
	\end{proof}
	
%	It is easy to check that the topological group $\gal_L(T)$ does not depend on the choice of the monster model $\C$ in which it is computed (up to a topological isomorphism).
	
	\begin{dfn}$\,$
		\begin{enumerate}[label=\roman{*}),nosep]
			\item
			$\gal_0(T)$ is defined as the closure of the identity in $\gal_L(T)$.
			\item
			$\gal_{KP}(T):=\gal_L(T)/\gal_0(T)$ equipped with the quotient topology is called the {\em Kim-Pillay Galois group of $T$}.
		\end{enumerate}
	\end{dfn}
	
	By general topology, $\gal_{KP}(T)$ is always a compact, Hausdorff group. On the other hand, the topology on $\gal_0(T)$ inherited from $\gal_L(T)$ is trivial, and one of the problems we address
	is how to treat $\gal_0(T)$ and $\gal_L(T)$ as mathematical objects and how to measure their complexity. Section~\ref{section: top dyn for aut(C)} will give us an answer to this question.
	
	Finally, recall that $E_L$ (on a given product of sorts) turns out to be the orbit equivalence relation of $\autf_L(\C)$, and $E_{KP}$ is the orbit equivalence relation of $\autf_{KP}(\C):=\mu^{-1}[\gal_0(T)]$.
	
	We finish with an easy lemma which will be used in the proof of Theorem~\ref{thm:main_Borel}, and whose last point is easily seen to be equivalent to Corollary~\ref{cor: KP class is the closure of L class}.
	
	\begin{lem}\label{lem:lem_closed}
		Suppose $Y$ is a type-definable set which is $E_L$-saturated. Then:
		\begin{enumerate}[label=\roman{*}),,nosep]
			\item
			$\autf_L(\C)$ acts naturally on $Y$.
			\item
			The subgroup $S$ of $\gal_L(T)$ consisting of all $\sigma/\autf_L(\C)$ such that $\sigma[Y]=Y$ (i.e.\ the setwise stabilizer of $Y/E_L$ under the natural action of $\gal_L(T)$) is a closed subgroup of $\gal_L(T)$. In particular, $\autf_{KP}(\C)/\autf_L(\C) =\gal_0(T) \leq S$.
			\item
			$Y$ is a union of $E_{KP}$-classes.
		\end{enumerate}
	\end{lem}
	
	\begin{proof}
		(i) follows immediately from the assumption that $Y$ is $E_L$-saturated.
		
		(ii) The fact that $S$ is closed can be	deduced from Fact~\ref{fct: characterization of topology on Gal_L(T)} and from the fact that this is a topological (not necessarily Hausdorff) group. To see this, note that $S=P \cap P^{-1}$, where $P:= \bigcap_{a \in Y} \{ \sigma/\autf_L(\C)\sbmid \sigma(a) \in Y\}$ is closed in $\gal_L(T)$ by Fact~\ref{fct: characterization of topology on Gal_L(T)}(iii). The second part of (ii) follows from the first one and the fact that $\autf_{KP}(\C)/\autf_L(\C) = \gal_0(T)=\cl (\id/\autf_L(\C))$.
		
		(iii) is immediate from (ii) and the fact that $E_{KP}$ is the orbit equivalence relation of $\autf_{KP}(\C)$. Alternatively, one can use Corollary ~\ref{cor: KP class is the closure of L class}. Namely, since $Y$ is type-definable and $E_L$-saturated, $Y/E_L$ is closed, so by Corollary ~\ref{cor: KP class is the closure of L class}, we get that for every $a \in Y$, $[a]_{E_{KP}}/E_L = \cl(a/E_L) \subseteq Y/E_L$, i.e. $Y$ is $E_{KP}$-saturated.
	\end{proof}

	\subsection{Bounded invariant equivalence relations and Borel cardinalities}\label{subsection: bounded relations}
	
	As was already mentioned, one of the general questions is how to measure the complexity of bounded, invariant equivalence relations. A possible answer is: via Borel cardinalities. However, any such a relation is defined on the monster model which is not any reasonable (Polish) topological space. Therefore, one has to interpret the relation in question in the space of types over a model. This was formalized in \cite{KPS} for Lascar strong types and generalized to arbitrary relations in \cite{KrRz}.
	
	First, we recall basic definitions and facts from \cite{KrRz}. Then we will discuss the most important known theorems and some questions which we answer in this paper.
	
	We work in a monster model $\C$ of some theory $T$.
	Recall that if $X$ is an $A$-invariant set, we associate with $X$ the subset
	\[
		X_A:=\{\tp(a/A)\sbmid a\in X\}
	\]
	of $S(A)$.
	
	In contrast to \cite{KrRz}, here by a type-definable set we mean a type-definable set over parameters.

	\begin{dfn}
		Suppose $X$ is a subset of some product of sorts $P$. Then we say that $P$ is the \emph{support} of $X$, and we say that $X$ is \emph{countably supported} if $P$ is a product of countably many sorts.
	\end{dfn}
	
	\begin{dfn}[Borel invariant set, Borel class of an invariant set]\label{definition: Borel subsets of a model}
		For any invariant set $X$, we say that $X$ is \emph{Borel} if the corresponding subset $X_\emptyset$ of $S(\emptyset)$ is, and in this case, by the {\em Borel class} of $X$ we mean the Borel class of $X_\emptyset$ (e.g.\ we say that $X$ is $F_\sigma$ if $X_\emptyset$ is $F_\sigma$, and we might say that $X$ is clopen if $X_\emptyset$ is clopen (i.e.\ if $X$ is definable)).
		
		Similarly if $X$ is $A$-invariant, we say that it is {\em Borel over $A$} if the corresponding subset $X_A$ of $S(A)$ is (and Borel class is understood analogously).
		
		We say that a set is {\em pseudo-$*$} if it is $*$ over some small set of parameters, e.g it is {\em pseudo-closed} if it is closed over some small set (equivalently, if it is type-definable (with parameters from a small set)).
	\end{dfn}
	
	\begin{dfn}
		Suppose $E$ is a bounded, invariant equivalence relation on an invariant set $X$ in a product $P$ of sorts, and $M$ is a model.
		
		Then we define $E^M\subseteq (X_M)^2\subseteq (P_M)^2$ as the relation
		\[
			p \Er^M q \iff \textrm{there are $a\models p$ and $b\models q$ such that }a \Er b.
		\]
		(Since $E$-classes are $M$-invariant, this is equivalent to saying that for all $a\models p,\, b\models q$ we have $a \Er b$, which implies that $E^M$ is an equivalence relation.)
	\end{dfn}
	The next proposition shows the Borel classes of $E^M$ and $E$ are the same in the countable case.
	\begin{fct}[Proposition 2.9 in \cite{KrRz}]\label{fct:eqcmp}
		Consider a model $M$, and some bounded, invariant equivalence relation $E$ on an invariant subset $X$ of a product of sorts $P$.
		
		Consider the natural restriction map $\pi\fcolon (P^2)_M\to (P_M)^2$ (i.e.\ $\pi(\tp(a,b/M))=(\tp(a/M),\tp(b/M))$). Then we have the following facts:
		\begin{itemize}
			\item Each $E$-class is $M$-invariant, in particular, for any $a,b\in X$
			\[
				a \Er b\iff \tp(a,b/M)\in E_M \iff \tp(a/M)\Er^M \tp(b/M)
			\]
			and $\pi^{-1}[E^M]=E_M$.
			\item
			If one of $E^M$ (as a subset of of $(P_M)^2$), $E_M$ (as a subset of $(P^2)_M$), or $E$ (considered as a subset of $(P^2)_{\emptyset}$) is closed or $F_\sigma$, then all of them are closed or $F_\sigma$ (respectively). In the countable case (when the support of $E$, the language and $M$ are all countable), we have more generally that the Borel classes of $E^M, E_M, E$ are all the same.
			\item
			Similarly -- for $M$-invariant $Y\subseteq X$ -- the relation $E^M\restr_{Y_M}$ is closed or $F_\sigma$ [or Borel in the countable case] if and only if $E_M\cap (Y^2)_M$ is.
		\end{itemize}
	\end{fct}
	
%Suppose for a moment that the language is countable.
	Although analyticity was not considered in \cite{KrRz}, one can easily check that the above definitions and observations have their counterparts for analyticity.
	Namely, using Definition~\ref{definition: analytic sets} of analytic sets in arbitrary spaces (which coincides with the definition of analytic sets in Polish spaces), we say that an invariant subset $X$ of some product $P$ of sorts is {\em analytic} if $X_\emptyset$ is an analytic subset of $P_{\emptyset}$.
%	
%	Suppose now that $E$ is additionally bounded and $M$ is a countable model. Since analyticity is preserved under taking images and preimages by continuous functions between Polish spaces, one easily gets that if one of $E^M$, $E_M$, or $E$ is analytic, then all of them are.
	Now, let $E$ be a bounded, invariant equivalence relation defined on an invariant subset of a product $P$ of sorts, and $M$ be a model.
	By Remark~\ref{remark: preservation of analyticity by images and preimages}, analyticity is preserved under taking images and preimages by continuous functions between compact, Hausdorff spaces. Moreover, the function $\pi\fcolon (P^2)_M\to (P_M)^2$ from the last fact and the restriction function $r \fcolon (P^2)_M \to (P^2)_\emptyset$ are both continuous and satisfy: $\pi[E_M] = E^M$, $\pi^{-1}[E^M]=E_M$, $r[E_M]=E_\emptyset$, and $r^{-1}[E_\emptyset]=E_M$. All of this implies that  if one of $E^M$, $E_M$, or $E$ is analytic, then all of them are.	
%In fact, similar statements also hold in the uncountable case for a suitable notion of analyticity (as in the discussion following Fact~\ref{fct:Miller} and in Proposition~\ref{prop: Souslin and parameters}), which can be shown by arguments similar to the proof of Proposition~\ref{prop: Souslin and parameters}(ii).
	
	Below, we will sometimes restrict a bounded, invariant equivalence relation $E$ defined on $X$ to an $E$-saturated set $Y \subseteq X$. Note that in such a situation, $Y$ is invariant over any model $M$ (which follows from the fact that $E$ is coarser than the relation of having the same type over $M$, and so classes of $E$ are invariant over $M$).

	\begin{fct}[Proposition 2.12 in \cite{KrRz}]\label{fct:cartdf}
		Assume that the language is countable. For any $E$ which is a bounded, invariant equivalence relation on some $\emptyset$-type-definable and countably supported set $X$, and for any $Y\subseteq X$ which is pseudo-closed (i.e.\ type-definable) and $E$-saturated, the Borel cardinality of the restriction of $E^M$ to $Y_M$ does not depend on the choice of the countable model $M$. In particular, if $X=Y$, the Borel cardinality of $E^M$ does not depend on the choice of the countable model $M$.
	\end{fct}
	
	This justifies the following definition.
	
	\begin{dfn}
		If $E$ is as in the previous proposition, then by the \emph{Borel cardinality of $E$} we mean the Borel cardinality of $E^M$ for a countable model $M$. Likewise, we say that $E$ is {\em smooth} if $E^M$ is smooth for a countable model $M$.
		
		Similarly, if $Y$ is pseudo-closed and $E$-saturated, the Borel cardinality of $E\restr_Y$ is the Borel cardinality of $E^M\restr_{Y_M}$ for a countable model $M$.
	\end{dfn}
	
%	We can extend the definition of Borel cardinalities to all (not necessarily Borel) equivalence relations on Polish spaces, and then Proposition~\ref{fct:cartdf} and the above definition can be extended to the context of arbitrary bounded, invariant equivalence relations. In particular, it makes sense to talk about smoothness of bounded, invariant (not necessarily Borel) equivalence relations.
	
	Type-definable equivalence relations are trivially smooth, because the associated relations on type spaces are closed and so smooth (in fact, any Borel equivalence relation $E$ on a Polish space $Y$ such that all $E$-classes are $G_\delta$-subsets of $Y$ is smooth \cite[Corollary 1.32]{KMS}). 
	\begin{fct}[Fact 2.14 in \cite{KrRz}]\label{fct: type-definability imply smoothness}
		\label{fct:tdsmt}
		A bounded, type-definable equivalence relation in a countable theory is smooth.
	\end{fct}

	Before we recall the main known theorems on non-smoothness of Lascar equivalence and, more generally, of some bounded, $F_\sigma$ equivalence relations, we need to recall first some definitions, particularly the definition of a normal form and the associated distance function.
	
	\begin{dfn}[Normal form]
		\label{dfn:nfrm}
		If $(\Phi_n(x,y))_{n \in \N}$ is a sequence of (partial) types over $\emptyset$ on a $\emptyset$-type-definable set $X$ such that $\Phi_0(x,y)=((x=y)\land x\in X)$ and which is increasing (i.e.\ for all $n$, $\Phi_n(x,y)\vdash \Phi_{n+1}(x,y)$), then we say that \emph{$\biglor_{n\in \N} \Phi_n(x,y)$ is a normal form} for an invariant equivalence relation $E$ on $X$ if we have for any $a,b\in X$ the equivalence $a \Er b\iff \C \models \biglor_{n\in \N}\Phi_n(a,b)$, and if the binary function $d=d_\Phi\fcolon X^2\to \N\cup \{\infty\}$ defined as
		\[
			d(a,b)= \min \{n\in \N \sbmid \C\models \Phi_n(a,b)\}
		\]
		(where $\min \emptyset=\infty$) is an invariant metric with possibly infinite values -- that is, it satisfies the axioms of coincidence, symmetry and triangle inequality. In this case, we say that \emph{$d$ induces $E$ on $X$}.
	\end{dfn}
	
	\begin{ex}
		The prototypical example of a normal form is $\biglor_n d_L(x,y)\leq n$, inducing $E_L$, and $d_L$ is the associated metric (where $E_L$ is the relation of having the same Lascar strong type and $d_L$ is the Lascar distance).
	\end{ex}
	
	It turns out that any $F_\sigma$ equivalence relation has a normal form (see \cite[Proposition 2.21]{KrRz}).
	
	A fundamental theorem of Newelski is the following.
	
	\begin{fct}[{Corollary 1.12 in \cite{Ne}}]\label{fct:twN}
		Assume $E$ is an equivalence relation with normal form $\biglor_{n\in\N}\Phi_n$. Assume $p\in S(\emptyset)$ and $Y\subseteq p(\C)$ is type-definable and $E$-saturated. Then $E$ is equivalent on $Y$ to some $\Phi_n(x,y)$ (and therefore $E$ is type-definable on $Y$), or $\lvert Y/E\rvert\geq 2^{\aleph_0}$.
	\end{fct}
	
	By Remark~\ref{rem:type-definability_of_relations}, one immediately gets
	
	\begin{cor}
		\label{cor:nwcr}
		Suppose $E$ is an invariant, $F_\sigma$ equivalence relation. Then, if for some complete type $p$ over $\emptyset$ the restriction $E\restr_{p(\C)}$ is not type-definable, it has at least $2^{\aleph_0}$ classes within any type-definable and $E$-saturated set $Y\subseteq p(\C)$.
	\end{cor}
	
	In particular,
	
	\begin{cor}\label{cor: Newelski's theorem}
		For any tuple $\bar a$, either $E_L \restr_{[\bar a]_{E_{KP}}}$ 
		has only one class, or it has at least $2^{\aleph_0}$ classes.
	\end{cor}

	If the language is countable, the above corollary says that either $E_L \restr_{[\bar a]_{E_{KP}}}$ has only one class, or $\Delta_{2^\omega}$ Borel reduces to it. Having in mind the Silver dichotomy and the Harrington-Kechris-Louveau dichotomy, it was conjectured in \cite{KPS} that the second part can be strengthened to the statement that $E_L \restr_{[\bar a]_{E_{KP}}}$ is non-smooth (i.e.\ $\EZ$ Borel reduces to it). This was proved in \cite{KMS}. More precisely:
	
	\begin{fct}[Main Theorem A in \cite{KMS}]\label{fct: KPS theorem}
		Assume that $T$ is a complete theory in a countable language, and consider $E_L$ on a product of countably many sorts. Suppose $Y$ is an $E_L$-saturated, pseudo-$G_\delta$ subset of the domain of $E_L$. Then either each $E_L$ class on $Y$ is $d_L$-bounded (from which it easily follows that $E_L$ coincides with $E_{KP}$ on $Y$, so it is type-definable on $Y$), or $E\restr_Y$ is non-smooth.
	\end{fct}
	
	In \cite{KaMi} and \cite{KrRz}, the last fact was generalized to a certain wider class of bounded $F_\sigma$ relations. In order to formulate this generalization, we need to recall one more definition from \cite{KrRz}.
	
	\begin{dfn}[Orbital equivalence relation, orbital on types equivalence relation]
		Suppose $E$ is an invariant equivalence relation on a set $X$.
		\begin{itemize}[nosep]
			\item
			We say that $E$ is \emph{orbital} if there is a group $\Gamma\leq \aut(\C)$ such that 
			$E$ is the orbit equivalence relation of $\Gamma$.
			\item
			We say that $E$ is \emph{orbital on types} if it refines type and the restriction of $E$ to any complete type over $\emptyset$ is orbital.
		\end{itemize}
	\end{dfn}
	
	\begin{fct}[Theorem 3.4 in \cite{KrRz}] \label{fct:mainA}
		We are working in the monster model $\C$ of a complete, countable theory. Suppose we have:
		\begin{itemize}[nosep]
			\item
			a $\emptyset$-type-definable, countably supported set $X$,
			\item
			an $F_\sigma$, bounded equivalence relation $E$ on $X$, which is orbital on types,
			\item
			a pseudo-closed and $E$-saturated set $Y\subseteq X$,
			\item
			an $E$-class $C\subseteq Y$ with infinite diameter with respect to some normal form of $E$.
		\end{itemize}
		Then $E\restr_Y$ is non-smooth.
	\end{fct}
	
	In fact, in \cite[Theorem 3.17]{KaMi}, the authors allow $X$ to be type-definable over some parameters (and $E$ is the intersection of an invariant set with $X \times X$) and assume only that $Y$ is pseudo-$G_\delta$, but we work with the stronger assumption that $X$ is $\emptyset$-type-definable and $Y$ is pseudo-closed (i.e.\ type-definable), as we find it the most interesting situation.
	Note also that in \cite[Theorem 3.17]{KaMi}, there is a slightly weaker assumption than orbitality on types, but one can easily see that both formulations of the theorem are equivalent (but assuming in \cite[Theorem 3.17]{KaMi} additionally that $X$ is $\emptyset$-type-definable and $Y$ is pseudo-closed).
	
	In \cite[Problem 3.21]{KaMi}, the authors asked if one can drop the assumption concerning orbitality in the above theorem. From our Theorem~\ref{thm:main_Borel}, it will follow that the answer is yes (assuming instead that $E$ refines $\equiv$; otherwise the answer is no,
	by \cite[Example 4.4]{KrRz}).
	In fact, our theorem is a much stronger generalization of the above theorem: not only
	do we remove the orbitality assumption but also, more importantly, the assumption that the relation is $F_\sigma$ (removing from the statement the part concerning the diameter and replacing it by an appropriate assumption of non-type-definability
	-- note that the two are equivalent for $F_\sigma$ equivalence relations, by Fact~\ref{fct:twN}).
	
	In \cite[Theorem 4.9]{KrRz}, it was deduced from Fact~\ref{fct:mainA} that if $E$ is an $F_\sigma$, bounded, orbital on types equivalence relation defined on a single complete type over $\emptyset$ or refining $E_{KP}$, then smoothness of $E$ is equivalent to type-definability of $E$. On the other hand, it was shown that if one drops the assumption that $E$ is defined on a single complete type over $\emptyset$ or refines $E_{KP}$, then smoothness need not imply type-definability. The following question was formulated there (Question 4.11 in Section 4.3).
	\begin{ques}\label{ques:quest1_from_KrRz}
		Suppose that $E$ is a Borel, bounded equivalence relation which is defined on a single complete type over $\emptyset$ or which refines $E_{KP}$.
		Is it true that smoothness of $E$ implies that $E$ is type-definable?
	\end{ques}
	From our Theorem~\ref{thm:main_Borel}, we will immediately get the positive answer to this question. 
%even after removing the assumption that the relation is Borel.
	
	All our results on [non-]smoothness of bounded, invariant equivalence relations (which are not necessarily $F_\sigma$) were not accessible by the methods of \cite{KMS,KaMi,KrRz} mainly due to the lack of 
	a distance function associated with normal forms of $F_\sigma$ relations.

	\subsection{Definable groups and their subgroups}\label{subsection: definable groups}
	Definable groups are not 
	the central notion in this paper, however, the results we obtain can be readily adapted to their context, as we will see in Corollaries~\ref{cor:group_borel},~\ref{cor:group_nwg} and~\ref{cor:group_trichotomy}.

	To formulate those corollaries, we need to recall some basic facts.	
	
	\begin{dfn}
		Suppose $G$ is a $\emptyset$-type-definable group and $H\leq G$ is invariant. We define $E_H$ as the relation on $G$ of lying in the same right coset of $H$.
	\end{dfn}
	
	In \cite{KrRz}, the following result has been proved. See also \cite[Corollary 3.36]{KaMi} for a more general statement.
	
	\begin{fct}[Corollary 3.9 in \cite{KrRz}]\label{fct:mainG_KR}
		Assume the language is countable.
		Suppose that $G$ is a $\emptyset$-definable group (and therefore finitely supported) and $H\unlhd G$ is an invariant, normal subgroup of bounded index, which is $F_\sigma$ (equivalently, generated by a countable family of type-definable sets). Suppose in addition that $K\geq H$ is a pseudo-closed (i.e.\ type-definable) subgroup of $G$. Then $E_H\restr_{K}$ is smooth if and only if $H$ is type-definable.
	\end{fct}
	
	To obtain it, the following construction is used.
	
	\begin{fct}[see {\autocite[Section 3, in particular Propositions 3.3 and 3.4]{GiNe}}]
		\label{fct:affine_sort} If $G$ is a $\emptyset$-definable group, and we adjoin to $\C$ a left principal homogeneous space $X$ of $G$ (as a new sort; we might think of it as an ``affine copy of $G$''), along with a binary function symbol for the left action of $G$ on $X$, we have the isomorphism
		\[
			\aut((\C,X,\cdot))\cong G\rtimes \aut(\C),
		\]
		where:
		\begin{enumerate}
			\item
			the semidirect product is induced by the natural action of $\aut(\C)$ on $G$,
			\item
			on $\C$, the action of $\aut(\C)$ is natural, and that of $G$ is trivial,
			\item
			on $X$ we define the action by fixing some $x_0$ and putting $\sigma_g(h\cdot x_0)=(hg^{-1})x_0$ and $\sigma(h\cdot x_0)=\sigma(h)\cdot x_0$ (for $g\in G$ and $\sigma\in \aut(\C)$).
		\end{enumerate}
	\end{fct}
	
	In this context, we induce another equivalence relation (which is an equivalence relation on the set of realizations of a single type).
	
	\begin{dfn}
		Let $H$ be an invariant subgroup of $G$. Then $E_{H,X}$ is the relation on $X$ of being in the same $H$-orbit.
	\end{dfn}
	
	Then the following fact, paired with Fact~\ref{fct:mainA}, yields Fact~\ref{fct:mainG_KR}.
	\begin{fct}[Lemma 2.35 and Proposition 2.42 from \cite{KrRz}]
		\label{fct:grres}
		Let $H\leq G$ be an invariant subgroup of bounded index and let $K$ be a pseudo-closed (i.e.\ type-definable) subgroup such that $H\leq K\leq G$.

		Let $M\preceq \C$ be any small model. Then, if we put $N=(M,G(M)\cdot x_0)\preceq (\C,X,\cdot)$, the map $g\mapsto g\cdot x_0$ induces a homeomorphism $G_M\to X_N$ which takes $E_{H}^M$ to $E_{H,X}^N$ and $K_M$ to $(K\cdot x_0)_N$.
		
		Furthermore:
		\begin{itemize}
			\item
			$E_{H,X}$ is type-definable or $F_\sigma$ if and only if $E_H$ is, if and only if $H$ is (respectively), 
			\item
			if the language and $M$ are both countable, while $H$ is Borel, so are $E_H$ and $E_{H,X}$, and the Borel cardinalities of $E_H\restr_{K}$ and $E_{H,X}\restr_{K\cdot x_0}$ coincide.
		\end{itemize}
	\end{fct}
	
	In fact, the assumption of Borelness is not needed for the last part of the last item concerning Borel cardinalities.
	
	\begin{rem}
		\label{rem:grresplus}
		The preceding fact can easily be extended to obtain the following additional information.
		\begin{itemize}
			\item
			$K\cdot x_0$ is type-definable (because $K$ is).
			\item 
			One of $H$, $E_H\restr_K$, and $E_{H,X}\restr_{K\cdot x_0}$ is type-definable if and only if all of them are.
			\item
			One of $H$, $E_H\restr_K$, and $E_{H,X}\restr_{K\cdot x_0}$ is relatively definable if and only if all of them are (in $K$, $K^2$ and $(K\cdot x_0)^2$, respectively).
		\end{itemize}
		
	\end{rem}

%It should be noted that, similarly to Facts~\ref{fct:eqcmp} and~\ref{fct:cartdf}, here we can also have $H$ only analytic, in which case $E_H$ and $E_{H,X}$ will be analytic (with an appropriate generalization of the definition in case when the language is uncountable, as in the discussion following Fact~\ref{fct:Miller} and in Proposition~\ref{prop: Souslin and parameters}), which can be shown by an argument similar to the proof of Proposition~\ref{prop: Souslin and parameters}(ii).
	
	Similarly to the discussion following Fact~\ref{fct:eqcmp}, we have a counterpart of the last item of Fact~\ref{fct:grres} for analyticity. Recall that an invariant set $Y$ is said to be {\em analytic} if $Y_\emptyset$ is analytic in the sense of Definition~\ref{definition: analytic sets}.
	
	\begin{rem}\label{remark: analyticity for groups}
		%Let $H$ be an invariant, bounded index subgroup of $G$ and $M$ be a model. 
		Take the general situation from Fact~\ref{fct:grres}. If one of the sets $H$, $E_H$, $E_{H,X}$, $E_H^M$, or $E_{H,X}^N$ is analytic, then all of them are. 
	\end{rem}

	\begin{proof}
		This follows from Remark~\ref{remark: preservation of analyticity by images and preimages} and the existence of appropriate continuous functions. In order to see that $H$ being analytic is equivalent to $E_H$ being analytic, consider the continuous function $f \fcolon (G \times G)_\emptyset \to H_\emptyset$ given by $f(\tp(a,b))= \tp(ba^{-1})$, and note that $f[(E_H)_\emptyset] = H_\emptyset$ and $f^{-1}[ H_\emptyset] = (E_H)_\emptyset$. The equivalences for pairs of relations $E_H$, $E_H^M$  and $E_{H,X}$, $E_{H,X}^N$ follow from the more general remark proved in the paragraph following Fact ~\ref{fct:eqcmp}. Finally, that $E_H^M$ is analytic if and only if $E_{H,X}^N$ is analytic follows from the existence of the continuous function  $\pi \fcolon G_M \to X_N$ considered in Fact~\ref{fct:grres}.
	\end{proof}

	\subsection{Topology}
	
	Let $X$ be a topological space. Recall that a subset $B$ of $X$ has the {\em Baire property} (BP) in $X$
	if it is the symmetric difference of an open and meager subset of $X$. We say that $B$ is {\em strictly Baire} if $B \cap C$ has the BP in $C$ for every closed subset $C$ of $X$ (or, equivalently, for every $C\subseteq X$; for this and other facts about strictly Baire sets, see \cite[§11 VI.]{Ku}). We say that $X$ is {\em totally non-meager} if no non-empty closed subset of $X$ is meager in itself. Of course, each compact, Hausdorff space is totally non-meager.
	
	One of the important ingredients of the proof of Theorem~\ref{thm:main_Borel} will be the following theorem from \cite{Mi}.
%whose proof given in \cite{Mi} is short and easy. 
	This theorem was pointed out to the first author by Maciej Malicki.
	
	\begin{fct}[Theorem 1 in \cite{Mi}]\label{fct:Miller}
		Assume $G$ is a totally non-meager topological group. Suppose $H$ is a subgroup of $G$ and $\{E_i \sbmid i \in \omega\}$ is a collection of right $H$-invariant (i.e.\ $E_iH=E_i$), strictly Baire sets which separates left $H$-cosets (i.e.\ for each $g \in G$, $gH= \bigcap \{ E_i \sbmid g \in E_i\}$). Then $H$ is closed in $G$.
	\end{fct}
	
	We will also use the Souslin operation $\Souslin$. Recall that a {\em Souslin scheme} is a family $(P_s)_{s \in \omega^{<\omega}}$ of subsets of a given set. The {\em Souslin operation} $\Souslin$ applied to such a scheme produces the set
	\[
		\Souslin_s P_s:=\bigcup_{s \in \omega^\omega} \bigcap_n P_{s\restr_n}.
	\]
	Given any collection $\Gamma$ of subsets of a set $X$, $\Souslin(\Gamma)$ denotes the collection of sets $\Souslin_s P_s$, where all sets $P_s$ are in $\Gamma$.
	
	It is well-known that in a Hausdorff topological space $X$, the collection of all subsets with BP is a $\sigma$-algebra which is closed under the Souslin operation \cite[Theorem 25.3]{Arh}. In particular, all sets in $\Souslin(\CLO(X))$ have BP, where $\CLO(X)$ is the collection of all closed subsets of $X$. It follows that, in fact, all sets in $\Souslin(\CLO(X))$ are even strictly Baire.
	
	We say that a Souslin scheme $(P_s)_{s \in \omega^{<\omega}}$ is {\em regular} if $s \subseteq t$ implies $P_s \supseteq P_t$. It is easy to check that if
	$(P_s)_{s \in \omega^{<\omega}}$ is a Souslin scheme and $Q_s:= \bigcap_{s \subseteq t} P_s$, then $(Q_s)_{s \in \omega^{<\omega}}$ is regular and $\Souslin_s P_s =\Souslin_s Q_s$.
	
	By \cite[Theorem 25.7]{Ke}, we know that in a Polish space, all Borel (even analytic) subsets are of the form $\Souslin_s F_s$ for a regular Souslin scheme $(F_s)_{s \in \omega^{<\omega}}$ consisting of closed subsets. 
	In fact, all analytic subsets of a Polish space $X$ form exactly the family $\Souslin(\CLO(X))$, and this description can be taken as a possible extension of the definition of analytic sets to arbitrary spaces, which we have in mind in the paragraph following Fact~\ref{fct:eqcmp} and in Remark~\ref{remark: analyticity for groups}.

	\begin{dfn}\label{definition: analytic sets}
		Let $X$ be a topological space. The members of $\Souslin(\CLO(X))$ will be called {\em analytic} subsets of $X$.
	\end{dfn}

%	We finish with an easy topological observation which we will use a couple of times.
	
	\begin{rem}\label{rem: image of intersection}
		Assume that $X$ is a compact (not necessarily Hausdorff) space and that $Y$ is a $T_1$-space. Let $f\fcolon X \to Y$ be a continuous map. Suppose $(F_n)_{n\in \omega}$ is descending sequence of closed subsets of $X$. Then $f[\bigcap_n F_n]=\bigcap_n f[F_n]$.
	\end{rem}
	
	\begin{proof}
		The inclusion $(\subseteq)$ is always true. For the opposite inclusion, consider any $y \in \bigcap_n f[F_n]$. Then $f^{-1}(y) \cap F_n \ne \emptyset$ for all $n$. Since $(F_n)_{n\in \omega}$ is descending, we get that the family $\{f^{-1}(y) \cap F_n\sbmid n \in \omega\}$ has the finite intersection property. On the other hand, since $\{y\}$ is closed in $Y$ (as $Y$ is $T_1$) and $f$ is continuous, we have that each set $f^{-1}(y) \cap F_n$ is closed. So compactness of $X$ implies that $f^{-1}(y) \cap \bigcap_n F_n=\bigcap_n f^{-1}(y) \cap F_n \ne \emptyset$. Thus $y \in f[\bigcap_n F_n]$.
	\end{proof}
	
	\begin{rem}\label{remark: preservation of analyticity by images and preimages}
		Let $f\fcolon X \to Y$ be a continuous map between topological spaces. Then:
		\begin{enumerate}
			\item The preimage by $f$ of any analytic subset of $Y$ is an analytic subset of $X$.
			\item Assume that $X$ is compact (not necessarily Hausdorff) and that $Y$ is Hausdorff. Then the image by $f$ of any analytic subsets of $X$ is an analytic subset of $Y$.
		\end{enumerate}
	\end{rem}
	
	\begin{proof}
		(1) is clear by continuity of $f$ and general properties of preimages.
		
		To show (2), consider any analytic subset $A$ of $X$. Then $A=\bigcup_{s \in \omega^\omega} \bigcap_n F_{s\restr_n}$ for some regular Souslin scheme $(F_s)_{s \in \omega^{<\omega}}$ of closed subsets of $X$. By compactness of $X$ and the assumptions that $Y$ is Hausdorff and $f$ is continuous, we see that each set $f[F_s]$ is closed. By Remark~\ref{rem: image of intersection}, 
		$$f[X] = \bigcup_{s \in \omega^\omega} \bigcap_n f[F_{s\restr_n}].$$
		Hence, $f[X]$ is analytic. 
	\end{proof}

	Let us recall Pettis theorem (for a proof see e.g.\ \cite[Theorem 9.9]{Ke}).
	
	\begin{fct}\label{fct: Pettis}
		Let $G$ be a topological group. If $A\subseteq G$ has BP and is non-meager, the set $A^{-1}A:=\{a^{-1}b\sbmid a,b \in A\}$ contains an open neighborhood of the identity.
	\end{fct}

	\section{Topological dynamics for \texorpdfstring{${\operatorname{Aut}(\mathfrak{C})}$}{Aut(C)}}\label{section: top dyn for aut(C)}

	In this section, we will prove our main results relating 
	the topological dynamics of $\aut(\C)$ with Galois groups and spaces of strong types, namely Theorems~\ref{thm:main theorem 1},~\ref{thm:main theorem 2} and~\ref{thm:main theorem 3}.
	
	In this section, $\C$ denotes a monster model of a complete, first order theory $T$, and $\bar c$ -- a tuple consisting of ALL elements of $\C$; $\C' \succ \C$ is a bigger monster model. Whenever we compute Galois groups, we do it inside $\C'$.
	Nonetheless, in this section, as well as the later ones, we will use automorphisms of both $\C$ and $\C'$, sometimes in the same context. To distinguish between the two, we will denote the latter by $\sigma$ or $\tau$ with primes (i.e.\ $\sigma',\tau'$).
	
	Recall that
	\[
		S_{\bar c}(\C): = \{ \tp(\bar a/\C) \sbmid \bar a \equiv \bar c\}.
	\]
	
	The group $\aut(\C)$ acts naturally on the space $S_{\bar c}(\C)$. It is easy to check that $(\aut(\C),S_{\bar c}(\C), \tp(\bar c/\C))$ is an $\aut(\C)$-ambit, where $\aut(\C)$ is equipped with the pointwise convergence topology. Moreover, the assignment $f \mapsto \tp(f(\bar c)/\C)$ yields a homeomorphic embedding of $\aut(\C)$ in $S_{\bar c}(\C)$.
	
	We will be working in the Ellis semigroup $EL:=EL(S_{\bar c}(\C))$ of the above ambit. One could ask whether on $S_{\bar c}(\C)$ there is a left continuous semigroup operation extending the natural action of $\aut(\C)$ on $S_{\bar c}(\C)$, because then $S_{\bar c}(\C)$ would be isomorphic to $EL$ and so the situation would be simplified (as for $\beta G$ for a discrete group $G$). As we will see in the appendix, such a semigroup operation exists if and only if $T$ is stable, which shows that in order to stay in full generality, we really have to work with $EL$.
	
	Recall that $EL$ is the closure in $S_{\bar c}(\C)^{S_{\bar c}(\C)}$ of $\aut(\C)$ (where the elements of $\aut(\C)$ are naturally treated as elements of $S_{\bar c}(\C)^{S_{\bar c}(\C)}$), and the semigroup operation, denoted by $*$, is just the composition of functions. Let $\Id\fcolon S_{\bar c}(\C) \to S_{\bar c}(\C)$ be the identity function.
	
	We enumerate $S_{\bar c}(\C)$ as $(\tp({\bar c}_k/\C) \sbmid k<\lambda)$ for some cardinal $\lambda$ and some tuples ${\bar c}_k\equiv \bar c$, where ${\bar c}_0=\bar c$. Then the elements of $S_{\bar c}(\C)^{S_{\bar c}(\C)}$ can be naturally viewed as sequences of types indexed by $\lambda$. For $k<\lambda$, denote by $\pi_k$ the projection from $EL$ to the $k$-th coordinate. In particular,
	\phantomsection
	\label{ind:pi0} $\pi_0(\Id)= \tp(\bar c/\C)$, and, more generally, $\pi_k(\Id)= \tp (\bar{c}_k/\C)$ for $k<\lambda$.

	Clearly, $S_{\bar c}(\C)^{S_{\bar c}(\C)}$ is an $\aut(\C)$-flow (with the coordinatewise action of $\aut(\C)$, denoted by $\cdot$). Then $EL=\cl (\aut(\C) \cdot \Id)$. So, $(\aut(\C),EL, \Id)$ is an $\aut(\C)$-ambit. Moreover, the natural embedding of $\aut(\C)$ in $EL$ is an isomorphism with its image (equal to $\aut(\C) \cdot \Id$) in the category of topological groups. Thus, we will be freely considering $\aut(\C)$ as a topological subgroup of $EL$.
	
	\begin{rem}\label{rem: pi0 is onto}
		$\pi_0$ is surjective.
	\end{rem}
	
	\begin{proof}
		It follows from the fact that the image of $\pi_0$ is closed in $S_{\bar c}(\C)$ (as $\pi_0\fcolon EL \to S_{\bar c}(\C)$ is continuous, $EL$ is compact, and $S_{\bar c}(\C)$ is Hausdorff) and the fact that the image of $\pi_0$ contains the orbit of $\aut(\C)$ on $\tp(\bar c/\C)$ which is dense in $S_{\bar c}(\C)$.
	\end{proof}
	
%	We have the following easy remark (which immediately follows from the left continuity of $*$).
	
%	\begin{rem}\label{rem: very basic}
%		For any $x ,y\in EL$ and any net $(f_i)_i \subseteq \aut(\C)$ such that $\lim_i f_i = x$, we have $x *y=\lim_i (f_i \cdot y)$ (where $y$ is treated as a sequence of types).
%	\end{rem}
	
	\begin{prop}\label{prop: very basic}
		For any $x \in EL$ there is $\sigma' \in \aut(\C')$ such that for all $k$, $\pi_k(x)= \tp(\sigma'({\bar c}_k)/\C)$.
	\end{prop}
	
	\begin{proof}
		This follows from compactness and the fact that $\aut(\C)$ is dense in $EL$. Indeed, by the strong $\kappa'$-homogeneity of $\C'$, we need to show that there are $\bar c_k' \in \C'$, $k < \lambda$, such that $({\bar c}_k'\sbmid k < \lambda) \equiv ({\bar c}_k\sbmid k < \lambda)$ and ${\bar c}_k' \models \pi_k(x)$ for all $k < \lambda$. This is a type-definable condition on $({\bar c}_k'\sbmid k < \lambda)$, so, by compactness (or rather $\kappa'$-saturation of $\C'$), it is enough to realize each finite fragment of this type. But this can be done by the density of $\aut(\C) \cdot \Id$ in $EL$.
	\end{proof}

	Below, we give a commutative diagram of maps which will be defined in the rest of this section and which play a fundamental role in this paper.
	
	\begin{figure}[H]
		\centering
		\begin{tikzcd}
			u\M \arrow[rdd, hook]\arrow[rd,two heads,"j"]\arrow[rrd,bend left=10,"f", two heads] \arrow[rrrd,bend left=15, two heads, "h_E"] \\
			& u\M/H(u\M) \arrow[r,"\bar f", two heads] \arrow[rr,"\bar h_E",bend left=15, two heads]& \gal_L(T)\arrow[r,"g_E",two heads]&\left[ \bar\alpha\right]_{\equiv}/E\\
			& \M\arrow[r, hook]& EL=EL(S_{\bar c}(\C))\arrow[u,"\hat{f}", two heads]
		\end{tikzcd}
		\caption{Commutative diagram of maps considered below (the tuple $\bar\alpha$ will be fixed after Corollary~\ref{cor:galquot}).}
	\end{figure}
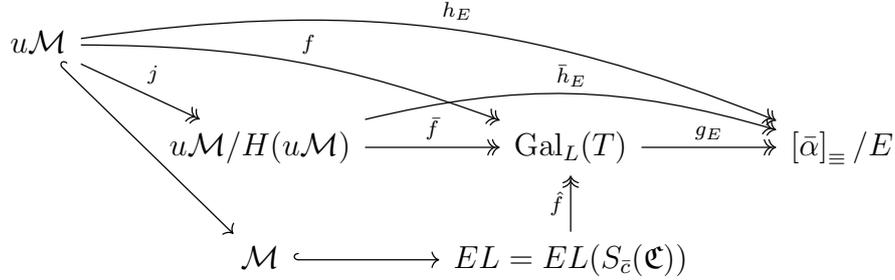
	In the table below, we give short descriptions and references to definitions of arrows in the diagram.
	\begin{center}
		\renewcommand{\arraystretch}{1.3}
		\begin{tabular}{r| l || r|l }
			$j$ & the quotient map 
			&
			$\hat f$& the homomorphism described below
			\\ \hline
			$g_E$& the orbit map of $[\bar\alpha]_E$ (page \pageref{ind:gE})
			&
			$f$ & restriction of $\hat f$ (page \pageref{ind:f})
			\\ \hline
			$h_E$ & $g_E\circ f$ (page \pageref{ind:hE}) 
			&
			$\bar f$& factor of $f$ (page \pageref{ind:fbar}) 
			\\ \hline 
			$\bar h_E$& $g_E\circ \bar f$ (page \pageref{ind:barhE}) 
		\end{tabular}
	\end{center}
	Now, define 
	\phantomsection
	\label{ind:fhat}
	$\hat{f}\fcolon EL \to \gal_L(T)$ by
	\[
		\hat{f}(x)=\sigma' \autf_L(\C'),
	\]
	where $\sigma' \in \aut(\C')$ is such that $\sigma'(\bar c) \models \pi_0(x)$. By a standard argument, we get that $\hat{f}$ is well-defined and onto. Indeed, suppose that $\sigma'_1,\sigma'_2 \in \aut(\C')$ are such that $\sigma'_1(\bar c), \sigma'_2(\bar c) \models \pi_0(x)$.
	Then there is $\tau' \in \aut(\C'/\C)$ such that $\tau'(\sigma'_1(\bar c))=\sigma'_2(\bar c)$, and so ${\sigma'_2}^{-1} \circ \tau' \circ \sigma'_1 \in \aut(\C'/\C)$,
	hence, using normality of $\autf_L(\C')$,
	we get that ${\sigma'_2}^{-1}\sigma'_1 \in \autf_L(\C')$, which shows that $\hat{f}$ is well-defined.
%	To see that it is onto, consider any $\sigma' \in \aut(\C')$, and take any small submodel $M \prec \C$ and its enumeration $\bar m$. By $\kappa$-saturation of $\C$, we can find $\bar m'$ in $\C$ such that $\bar m' \equiv_M \sigma'(\bar m)$, and hence we can choose $\tau \in \aut(\C)$ such that $\tau(\bar m)=\bar m'$ and extend it arbitrarily to $\tau' \in \aut(\C')$. Then one easily checks that $\sigma' \autf_L(\C')= \tau' \autf_L(\C')$ which implies that the element $\sigma' \autf_L(\C')$ of $\gal_L(T)$ belongs to the image of $\hat{f}$.
	To see that it is onto, consider any $\sigma' \in \aut(\C')$. By Remark~\ref{rem: pi0 is onto}, there is $x \in EL$ with $\sigma'(\bar c) \models \pi_0(x)$. Then $\hat{f}(x)=\sigma' \autf_L(\C')$. 
	
	\begin{rem}
		For any $k<\lambda$, $\hat{f}(x)=\sigma' \autf_L(\C')$, where $\sigma' \in \aut(\C')$ is such that $\sigma'({\bar c}_k) \models \pi_k(x)$.
	\end{rem}
	\begin{proof}
		By Proposition~\ref{prop: very basic}, there is $\sigma' \in \aut(\C')$ such that $\sigma'({\bar c}) \models \pi_0(x)$ and $\sigma'({\bar c}_k) \models \pi_k(x)$. This is enough, as the value $\tau' \autf_L(\C')$ does not depend on the choice of $\tau' \in \aut(\C')$ such that $\tau'({\bar c}_k) \models \pi_k(x)$.
	\end{proof}
	
	\begin{prop}\label{prop: semigroup epi}
		$\hat{f}:EL \to \gal_L(T)$ is a semigroup epimorphism.
	\end{prop}
	\begin{proof}
		Take any $x,y \in EL$.
		There is a unique $k$ such that $\pi_0(y)=\pi_k(\Id)$. 
%Choose any net $(f_i)_i \subseteq \aut(\C)$ converging to $x$. Then $\pi_k(x)= \lim_i f_i(\pi_0(y))$. By Remark~\ref{rem: very basic}, we also have $x * y = \lim_i (f_i \cdot y)$, and so $\pi_0(x *y)= \lim_i f_i(\pi_0(y))$. Thus, $\pi_0(x*y)=\pi_k(x)$.
		Then $\pi_0(xy) = (x*y)(\tp(\bar c/\C))= x(y(\tp(\bar c/\C)))=x(\pi_0(y)) = x(\pi_k(\Id)) = \pi_k(x)$.
		
		By Proposition~\ref{prop: very basic}, there is $\sigma' \in \aut(\C')$ such that $\sigma'({\bar c}) \models \pi_0(x)$ and $\sigma'({\bar c}_k) \models \pi_k(x)$. There is also $\tau' \in \aut(\C')$ such that $\tau'(\bar c)={\bar c}_k \models \pi_0(y)$.
		
		By these two paragraphs, we conclude that $(\sigma' \tau')(\bar c) \models \pi_0(x*y)$. Thus, $\hat{f}(x*y)=(\sigma' \tau')\autf_L(\C')= (\sigma' \autf_L(\C')) (\tau' \autf_L(\C'))=\hat{f}(x) \hat{f}(y)$.
	\end{proof}
	
	Although the next remark will not be applied anywhere in this paper, we thought that it should be included here.
	
	\begin{rem}
		$\hat{f}$ is continuous, where $\gal_L(T)$ is equipped with the standard (compact but not necessarily Hausdorff) topology as defined in Subsection~\ref{subsection: model theory}.
	\end{rem}
	
	\begin{proof}
		Let $C \subseteq \gal_L(T)$ be closed. By the definition of the topology on $\gal_L(T)$, we get that $D:=\{ \tp(\sigma'(\bar c)/\C)\sbmid \sigma' \autf_L(\C') \in C\}$ is closed in $S_{\bar c}(\C)$. 
		Since $\hat{f}^{-1}[C]=\{x \in EL\sbmid \pi_0(x) \in D\}$ and $\pi_0$ is continuous, we conclude that $\hat{f}^{-1}[C]$ is closed in $EL$.
	\end{proof}
	
	From now on, let $\M $ be a minimal left ideal in $EL$, and $u$ -- an idempotent in $\M $. So, $u\M $ ($=u*\M $) is the associated Ellis group. Clearly $\M =EL * u$, and so $u*\M =u*EL*u$. Since $\hat{f}$ is a semigroup epimorphism, 
	%we see that $\hat{f}(u)=\id \autf_L(\C')$, and 
	we get the following corollary concerning the function
	\phantomsection\label{ind:f}
	$f\fcolon u\M \to \gal_L(T)$ defined as the restriction of $\hat{f}$ to $u\M $.
	
	\begin{cor}\label{cor: group epi}
		$f\fcolon u\M \to \gal_L(T)$ is a group epimorphism.
	\end{cor}
	
	\begin{proof}
		Only surjectivity requires an explanation. We have
		$f[u*\M] = \hat{f}[u*EL *u]=\hat{f}(u)\hat{f}[EL]\hat{f}(u) = \gal_L(T)$, because $\hat{f}[EL] =\gal_L(T)$.
	\end{proof}

	Now, we will prove a counterpart of Theorem 0.1 from \cite{KrPi}. The proof is an adaptation of the proof from that paper to our new context, which, however, requires more technicalities.
	
	As usual, $\mu\fcolon \aut(\C') \to \gal_L(T)$ will be the quotient map, and $\gal_L(T)$ and $\gal_{KP}(T)$ are equipped with the standard topologies (see Subsection~\ref{subsection: model theory}). For the definition of the $\tau$-topology on $u\M $ and the definition of the subgroup $H(u\M )$ see Subsection~\ref{subsection: topological dynamics}.
	
	\begin{thm}\label{thm:main theorem 1}
		Suppose that $u\M$ is equipped with the $\tau$-topology and $u\M/H(u\M)$ -- with the induced quotient topology. Then:
		\begin{enumerate}[nosep]
			\item $f$ is continuous.
			\item $H(u\M) \leq \ker(f)$.
			\item The formula $pH(u\M) \mapsto f(p)$ yields a well-defined continuous epimorphism 
			\phantomsection\label{ind:fbar}
			$\bar f$ from $u\M/H(u\M)$ to $\gal_L(T)$.
		\end{enumerate}
		In particular, we get the following sequence of continuous epimorphisms:
		\begin{equation}
			u\M\twoheadrightarrow u\M/H(u\M)\xtwoheadrightarrow{\bar f}{}\gal_L(T)\twoheadrightarrow\gal_{KP}(T).
		\end{equation}
	\end{thm}
	\begin{proof}
		(1) Let $\bar D \subseteq \gal_L(T)$ be closed. Then $D:=\mu^{-1}[\bar D] \bar c$ is type-definable by Fact~\ref{fct: characterization of topology on Gal_L(T)}.
		The goal is to show that $f^{-1}[\bar D]$ is a $\tau$-closed subset of $u\M$.
		
		Consider any $p \in \cl_{\tau}(f^{-1}[\bar D])$. By the definition of the $\tau$-topology, there are $g_i \in \aut(\C)$ and $p_i \in f^{-1}[\bar D]$ such that $\lim_i g_i =u$ and $\lim_ig_ip_i =p$.
		
		Let $F_n$ be the collection of all pairs
		$(\bar a,\bar b)$ from $\C'$
		(where $\bar a$ and $\bar b$ are from the same sorts as $\bar c$) for which there are models $M_0,\dots,M_{n-1} \prec \C'$ and a sequence $\bar{d}_0,\dots,\bar{d}_{n}$ such that $\bar a= \bar{d}_0 \equiv_{M_{0}} \bar{d}_1 \equiv_{M_1} \dots \equiv_{M_{n-1}} \bar{d}_{n}=\bar b$. 
		%And let $\widetilde{F}_n=\{ \tp(\bar a,\bar b)/\C)\sbmid (\bar a, \bar b) \in F_n\}$. 
		Then $F_n$ is $\emptyset$-type-definable;
		%and $\widetilde{F}_n$ is closed; 
		so we can identify $F_n$ with a partial type over $\emptyset$ closed under conjunction. We will write $d(\bar a,\bar b) \leq n$ iff $(\bar a,\bar b) \in F_n$, and for $\sigma' \in \aut(\C')$, $d(\sigma')\leq n$ iff $\sigma'$ can be written as the composition of $n$ automorphisms each of which fixes pointwise a submodel.
		
		Since $u \in \ker( f)$, we get that for $\bar \alpha \models \pi_0(u)$, one has $d(\bar \alpha,\bar c) \leq n$ for some $n$. As $\lim_i \tp(g_i(\bar c)/\C))=\pi_0(u)$, we get that for every $\varphi(\bar x,\bar y) \in F_n$ the formula $\varphi(g_i(\bar c),\bar c)$ holds for $i$ big enough.
		
		Take any $\bar{a}_i \models \pi_0(p_i)$. Note that each $\bar{a}_i$ belongs to $D$
		(this follows from the fact that $\bar a_i\in \mu^{-1}[\{f(p_i)\}]\bar c$).
		For each $i$, let $g_i'$ be an extension of $g_i$ to an automorphism of $\C'$. By the last paragraph and the fact that $\lim_ig_ip_i =p$, we get that for every $\varphi(\bar x,\bar y) \in  F_n$ and $\psi(\bar x) \in \pi_0(p)$ one has that $\varphi(g_i'(\bar c),\bar c) \wedge \psi(g_i'(\bar{a}_i))$ holds for $i$ big enough.
		Thus,
		\[
			\forall \varphi(\bar x,\bar y) \in F_n\; \forall \psi(\bar x) \in \pi_0(p)\; \exists i\; \exists \bar a, \bar b\; (\models \varphi(\bar b,\bar c) \wedge \psi(\bar a) \; \mbox{and} \; \bar c \bar{a}_i \equiv \bar b \bar a).
		\]
		So, by compactness, there are $\bar a, \bar b$ and $\bar d \in D$ such that
		\[
			d(\bar b,\bar c) \leq n\; \, \mbox{and}\;\, \bar a \models \pi_0(p)\;\, \mbox{and}\;\, \bar c \bar d \equiv \bar b \bar a.
		\]
		So, there is $\sigma' \in \aut(\C')$ such that $\sigma'(\bar c \bar d)=\bar b \bar a$. Since $d(\bar b,\bar c) \leq n$ and $\bar c$ is a model, we see that $d(\sigma')\leq n+1$; in particular, $\sigma' \in \autf_L(\C')$.
		
		By Remark~\ref{rem: pi0 is onto}, there is $q \in EL$ such that $\pi_0(q)=\tp(\bar d/\C)$.
		Since $\sigma'(\bar d) =\bar a \models \pi_0(p)$, we get $f(p)(=\hat{f}(p))=\hat{f}(q)$. But since $\bar d \in D$, we see that $\hat{f}(q) \in \bar D$. Therefore,
		$p \in f^{-1}[\bar D]$.\\[2mm]
		(2) Let $\widetilde{G}_n$ be the subset of $EL$ consisting of all sequences whose first coordinate equals $\tp(\bar a/\C)$ for some $\bar a$ such that $(\bar a,\bar c) \in F_n$. Since $u \in \ker( f)$, we have that $u \in \widetilde{G}_n$ for some $n$. Let $\pi(\bar x)$ be the partial type over $\C$ 
		which defines $F_{2n+1}(\C',\bar c)$ and is closed under conjunction.
		Consider any $\varphi(\bar x) \in \pi(\bar x)$. 
		Let
		\[
			V_\varphi:=\pi_0^{-1}[[\neg \varphi(\bar x)]] \cap u\M.
		\]
		(Recall that $\pi_0^{-1}[[\neg \varphi(\bar x)]]$ is the clopen subset of $EL$ consisting of sequences whose first coordinate is a type containing $\neg \varphi(\bar x)$.)
		\begin{clm*}$\,$
			\begin{enumerate}[label=(\roman*),nosep]
				\item
				$u \notin \cl_{\tau}(V_\varphi)$.
				\item
				$\cl_\tau (u\M \setminus \cl_\tau(V_\varphi)) \subseteq \cl_\tau(u\M \setminus V_\varphi) \subseteq \widetilde{G}_{3n+2}^\varphi$, where $\widetilde{G}_{3n+2}^\varphi$ is the subset of $EL$ consisting of all sequences whose first coordinate equals $\tp(\bar a /\C)$ for some $\bar a$ for which there is $\bar b$ such that $\models \varphi(\bar b)$ and $(\bar a,\bar b) \in F_{n+1}$.
			\end{enumerate}
		\end{clm*}
		\begin{clmproof}[Proof of claim]
			(i) Suppose for a contradiction that $u \in \cl_{\tau}(V_\varphi)$. So there are $g_i \in \aut(\C)$ and $p_i \in V_\varphi$ such that $\lim_i g_i =u$ and $\lim_ig_ip_i =u$. Arguing as in the proof of (1), we conclude that there are $\bar a, \bar b$ and $\bar d \models \neg \varphi(\bar x)$ such that
			\[
				d(\bar b,\bar c) \leq n\; \, \mbox{and}\;\, \bar a \models \pi_0(u)\;\, \mbox{and}\;\, \bar c \bar d \equiv \bar b \bar a.
			\]
			So, there is $\sigma' \in \aut(\C')$ such that $\sigma'(\bar c \bar d)=\bar b \bar a$, and we see that $d(\sigma')\leq n+1$. Thus, $d(\bar d,\bar a) \leq n+1$. But $\bar a \models \pi_0(u)$, so $d(\bar a, \bar c) \leq n$. Therefore, $d(\bar d, \bar c) \leq 2n+1$, i.e.\ $\bar d \models \pi(\bar x)$,
			which contradicts the assumption that $\bar d \models \neg \varphi(\bar x)$.\\[2mm]
			(ii) We need to check that $\cl_\tau(u\M \setminus V_\varphi) \subseteq \widetilde{G}_{3n+2}^\varphi$. Consider any $p \in \cl_\tau(u\M \setminus V_\varphi)$.
			There are $g_i \in \aut(\C)$ and $p_i \in u\M \setminus V_\varphi$ such that $\lim_i g_i =u$ and $\lim_ig_ip_i =p$. Arguing as in the proof of (1), we conclude that there are $\bar a, \bar b$ and $\bar d \models \varphi(\bar x)$ such that
			\[
				d(\bar b,\bar c) \leq n\; \, \mbox{and}\;\, \bar a \models \pi_0(p)\;\, \mbox{and}\;\, \bar c \bar d \equiv \bar b \bar a.
			\]
			As in (i), we get $d(\bar d,\bar a) \leq n+1$, which together with the fact that $\models \varphi(\bar d)$ and $\bar a \models \pi_0(p)$ gives us that $p \in \widetilde{G}_{3n+2}^\varphi$.
		\end{clmproof}
		Notice that $\bigcap_{\varphi(\bar x) \in \pi(\bar x)}\widetilde{G}_{3n+2}^\varphi = \widetilde{G}_{3n+2}$. 
		Moreover, for each $\varphi \in \pi(x)$, $u\M \setminus \cl_\tau(V_\varphi)$ is, by Claim (i), a
		$\tau$-open neighborhood of $u$ in $u\M$. Hence, by Claim (ii), we see that
		\begin{align*}
			H(u\M) &= \bigcap \left\{\cl_{\tau}(U)\sbmid U \; \mbox{$\tau$-neighborhood of} \; u\right\} \subseteq \bigcap_{\varphi(\bar x) \in \pi(\bar x)}\widetilde{G}_{3n+2}^\varphi \cap u\M= \\& = \widetilde{G}_{3n+2} \cap u\M\subseteq \ker(f),
		\end{align*}
		which finishes the proof of (2).\\[1mm]
		(3) follows from (1) and (2).
	\end{proof}
	
	The next observation is an immediate corollary of the above theorem.
	
	\begin{cor}
		The group $\gal_L(T)$ is abstractly isomorphic to a quotient of a compact, Hausdorff group. More precisely, it is (abstractly) isomorphic to the quotient of $u\M /H(u\M )$ by $\ker (\bar f)$.
	\end{cor}
	
	In Corollary~\ref{cor:galquot}, we will see that the above isomorphism is actually topological (i.e.\ it is a homeomorphism).
	
	The next theorem is interesting in its own right, but it is also essential for applications in further sections.
	It is a counterpart of \cite[Theorem 0.2]{KrPi}, and the proof from \cite{KrPi} goes through except that one of the lemmas there and one of the remarks requires a new proof which is done below.
	
	\begin{thm}\label{thm:main theorem 2}
		The group $\gal_0(T)$ is the quotient of a compact, Hausdorff group by a dense subgroup.
		More precisely, for $Y:=\ker(\bar f)$ let $\cl_\tau(Y)$ be its closure inside $u\M/H(u\M)$. Then $\bar f[\cl_\tau(Y)]=\gal_0(T)$, so $\bar f$ restricted to $\cl_\tau(Y)$ induces an isomorphism between $\cl_\tau(Y)/Y$ and $\gal_0(T)$.
	\end{thm}
	\begin{proof}
		The proof is almost the same as the proof of Theorem 0.2 in \cite{KrPi}. So the reader is referred to that proof, and here we only give a brief outline and explain
		the non-obvious modifications which are needed. 

		The point is that if one replaces $S_{G,M}(N)$ by $EL$, $G/{G^*}^{000}_A$ by $\gal_L(T)$, ${G^*}^{00}_A/{G^*}^{000}_A$ by $\gal_0(T)$, and $S_{{G^*}^{000}_A,M}(N)$ by $\ker(\hat{f})$, then the proofs of all the lemmas and remarks involved in the proof of \cite[Theorem 0.2]{KrPi} go through automatically, except Remark 4.2 and Lemma 4.7 whose proofs require an adaptation to the present context which is done below. But before that we give a brief outline of the proof.
		
		Since $\bar f$ is continuous and $\gal_0(T)$ is closed and contains the identity, we get $ \bar f [\cl_\tau(Y)]\subseteq \gal_0(T)$. It remains to prove the opposite inclusion.		
		Take the notation from the proof of \cite[Theorem 0.2]{KrPi}. In particular, $J$ denotes the set of all idempotents in $\M$, $P_u:=\ker(f)=\ker(\hat{f}) \cap u\M $ and $P_v$ is defined analogously for any $v \in J$; $S:=\cl_\tau(P_u)$. Since the quotient map $j \fcolon u\M \to u\M/H(u\M)$ is continuous and closed, we see that $j[S]=\cl_{\tau}(Y)$. Hence, $\bar f[\cl_{\tau}(Y)] = f[S]$. Since $\gal_0(T)$ is the closure of the identity in $\gal_L(T)$, the whole proof boils down to showing that $f[S]$ is closed in $\gal_L(T)$. The counterpart of Remark 4.3 from \cite{KrPi} reduces the last thing to showing that $\hat{f}^{-1}[f[S]] \cap \M$ is closed in $\M$. The counterpart of Lemma 4.6 from \cite{KrPi} says that $\hat{f}^{-1}[f[S]] \cap \M= \bigcup_{v \in J} v \circ P_u$ (where $\circ$ is the circle operation defined in Subsection~\ref{subsection: topological dynamics}). Finally, the counterpart of Lemma 4.8  from \cite{KrPi} says that the last set is closed, which completes the proof.

		Now, we only give proofs of the counterparts of two steps from \cite{KrPi} which require an adaptation to our context.
		
		\begin{rem}[The counterpart of Remark 4.2 from \cite{KrPi}]\label{rem: new 4.2}
			$\bar D \subseteq \gal_L(T)$ is closed iff $\hat{f}^{-1}[\bar D]$ is closed.
		\end{rem}
		\begin{proof}
			By Fact~\ref{fct: characterization of topology on Gal_L(T)}, we have: $\bar D$ is closed iff $D:=\mu^{-1}[\bar D]{\bar c}$ is type-definable iff $\widetilde{D}:=\{ \tp(\bar a/\C) \sbmid \bar a \in D\}$ is closed in $S_{\bar c}(\C)$. But $\hat{f}^{-1}[\bar D] = \pi_0^{-1}[\widetilde{D}]$, and $\pi_0$ is onto $S_{\bar c}(\C)$ by Remark~\ref{rem: pi0 is onto}, continuous and closed. Hence, $\bar D$ is closed iff $\widetilde{D}$ is closed iff $\hat{f}^{-1}[\bar D]$ is closed.
		\end{proof}
		
		\begin{lem}[The counterpart of Lemma 4.7 from \cite{KrPi}]\label{lem: new 4.7}
			$\cl (J) \subseteq \ker(\hat{f}) \cap \M$. Equivalently, $\cl (J) \subseteq \bigcup_{v \in J} P_v = \bigcup_{v \in J} vP_u$.
		\end{lem}
		\begin{proof}
			The second part follows as in \cite{KrPi}. We will prove the first part.
			
			Consider for a moment an arbitrary $v \in J$. 
%Now, we use a similar argument to the proof of Proposition~\ref{prop: semigroup epi}. 
			There is $k$ such that $\pi_0(v)=\pi_k(\Id)$. So $\bar{c}_k \models \pi_0(v)$. Then, by the fact that $v*v=v$,
			%and Remark~\ref{rem: very basic}, 
			we get $\pi_0(v)= \pi_0(v*v)=v(v(\tp(\bar c/\C)))=v(\pi_0(v))= v(\pi_k(\Id))=\pi_k(v)$. By Proposition~\ref{prop: very basic},
			\[
				\exists \bar a\; \exists \bar b \; (\bar a \models \pi_0(v) \;\, \mbox{and}\;\, \bar b \models \pi_k(v)=\pi_0(v)\;\, \mbox{and}\;\, \bar c \bar{c}_k \equiv \bar a \bar b).
			\]
			
			Let $p \in \cl(J)$. Consider any formula $\varphi(\bar x) \in \pi_0(p)$. Then $J \cap \pi_0^{-1}[[\varphi(\bar x)]] \ne \emptyset$. So, by the above paragraph,
			\[
				\exists \bar d \; \exists \bar a \; \exists \bar b\; ( \bar d \models \varphi(\bar x) \;\, \mbox{and}\;\, \bar a \equiv_\C \bar d \equiv_\C \bar b \;\, \mbox{and}\;\, \bar c \bar d \equiv \bar a \bar b).
			\]
			Thus, by compactness, there are $\bar d$, $\bar a$ and $\bar b$ such that
			\[
				\bar d \models \pi_0(p) \;\, \mbox{and}\;\, \bar a \equiv_\C \bar d \equiv_\C \bar b \;\, \mbox{and}\;\, \bar c \bar d \equiv \bar a \bar b.
			\]
			So, we can choose $\sigma' \in \aut(\C')$ such that $\sigma'(\bar c \bar d)= \bar a \bar b$. Since $\sigma'(\bar d)=\bar b \equiv_\C \bar d$, we see that $\sigma' \in \autf_L(\C')$. On the other hand, since $\sigma' (\bar c)=\bar a \equiv_\C \bar d \models \pi_0(p)$, we see that $\hat{f}(p) = \sigma' \autf_L(\C')$. Therefore, $p \in \ker(\hat{f}) \cap \M$.
		\end{proof}
		The proof of Theorem~\ref{thm:main theorem 2} is completed.
	\end{proof}
	
	\begin{cor}
		\label{cor:galquot}
		The mapping $\bar f\fcolon u\M/H(u\M)\to \gal_L(T)$ is a topological group quotient mapping (i.e.\ it is a surjective homomorphism such that any given subset of $\gal_L(T)$ is closed iff its preimage is closed). Thus, the induced group isomorphism from $(u\M/H(u\M))/\ker(\bar f)$ to $\gal_L(T)$ is a homeomorphism.
	\end{cor}
	
	\begin{proof}
		
		$\bar f$ is a continuous, surjective homomorphism immediately by Theorem~\ref{thm:main theorem 1}.
		
		Let $q\fcolon \gal_L(T)\to \gal_{KP}(T)$ be the natural quotient map. Then $q \circ \bar f$ is closed, because it is continuous, $\gal_{KP}(T)$ is Hausdorff and $u\M/H(u\M)$ is compact.

		Consider any $A\subseteq \gal_L(T)$ such that $A':=\bar f^{-1}[A]$ is closed. Then $A'$ is $Y:=\ker (\bar f)$-invariant (i.e.\ $A'=A'Y$), and therefore also $\cl_\tau(Y)$-invariant. 
%But, by Theorem~\ref{thm:main theorem 2}, $\cl_\tau(Y)=\ker (q\circ \bar f)$, so $A=q^{-1}[q\circ \bar f[A']]$, and therefore it is closed (as the preimage by a continuous map of the image by a closed map of a closed set).
		
		Now, we check that $\cl_\tau(Y)=\ker (q\circ \bar f)$. Since $\ker (q\circ \bar f) = \bar f^{-1}[\gal_0(T)]$, the inclusion $\cl_\tau(Y) \subseteq \ker (q\circ \bar f)$ follows immediately from Theorem~\ref{thm:main theorem 2}. For the opposite inclusion, take any $p \in \bar f^{-1}[\gal_0(T)]$. By Theorem~\ref{thm:main theorem 2}, there is $p' \in \cl_\tau(Y)$ such that $\bar f(p) = \bar f(p')$. So $p \in p'\ker(\bar f)=p'Y\subseteq \cl_\tau(Y)\cl_\tau(Y) = \cl_\tau(Y)$.
		
		By the last two paragraphs, $A=q^{-1}[q\circ \bar f[A']]$, and therefore it is closed (as the preimage by a continuous map of the image by a closed map of a closed set).
	\end{proof}

	Now, we would like to extend the context to the spaces of arbitrary strong types on any tuples (not necessarily enumerating a model). So, let $E$ be any bounded, invariant equivalence relation refining $\equiv$, and let $\bar \alpha \in \dom(E)$. Although all the objects defined below ($g_E$, $h_E$, etc.) depend not only on $E$ but also on $\bar \alpha$, we will skip it in the notation (i.e. we will not write $g_{E,\bar \alpha}, h_{E,\bar \alpha}$, etc.).\\ 
	
	 $\gal_L(T)$ acts transitively on $[\bar \alpha]_{\equiv}/E$ in the obvious way.
	Define 
	\phantomsection
	\label{ind:gE}$g_E \fcolon \gal_L(T) \to [\bar \alpha]_{\equiv}/E$ by taking the value of this action at the element $\bar \alpha/E$. An explicit formula for $g_E$ is
	\[
		g_E(\sigma' \autf_L(\C'))=\sigma'(\bar \alpha)/E.
	\]
	The following remark is a folklore result which follows immediately from Fact~\ref{fct: characterization of topology on Gal_L(T)} and the definition of the logic topology.
	
	\begin{rem}\label{rem: g_E is quotient}
		$g_E$ is a topological quotient mapping (i.e.\ it is a continuous surjection such that any given set in $[\bar \alpha]_\equiv/E$ is closed iff its preimage is closed).
	\end{rem}
	
	Composing the action of $\gal_L(T)$ on $[\bar \alpha]_{\equiv}/E$ with the epimorphism $f\fcolon u\M \to \gal_L(T)$, we get a transitive action of $u\M$ on $[\bar \alpha]_{\equiv}/E$, and similarly, composing it with $\bar f\fcolon u\M/H(u\M) \to \gal_L(T)$, we get a transitive action of $u\M/H(u\M)$ on $[\bar \alpha]_{\equiv}/E$. 
	\phantomsection
	Define \label{ind:hE}$h_E\fcolon u\M \to [\bar\alpha]_\equiv/E$ and \label{ind:barhE}$\bar h_E\fcolon u\M/H(u\M) \to [\bar \alpha]_{\equiv}/E$ by taking the values of these actions at the element $\bar \alpha/E$. Then
	\[
		h_E=g_E \circ f \fcolon u\M \to [\bar\alpha]_\equiv/E\;\; \mbox{and}\;\; \bar h_E=g_E \circ \bar f\fcolon u\M/H(u\M) \to [\bar \alpha]_{\equiv}/E
	\]
are continuous surjections (where $\circ$ stands for the composition of functions).
	Note that $\bar h_E$ is the factorization of $h_E$ through $H(u\M)$. An explicit formula for $\bar h_E$ is 
	\[
		\bar h_E (pH(u\M))= \sigma'(\bar \alpha)/E,
	\]
	where $\sigma' \in \aut(\C')$ is such that $\sigma'(\bar c) \models \pi_0(p)$.
	
	Define $\ker(h_E)$ and $\ker(\bar h_E)$ to be the stabilizers of $\bar \alpha/ E$ with respect to the actions of $u\M$ on $[\bar \alpha]_{\equiv}/E$ and $u\M/H(u\M)$ on $[\bar \alpha]_{\equiv}/E$, respectively. So $\ker(h_E)$ and $\ker(\bar h_E)$ are (not necessarily normal) subgroups of $u\M$ and $u\M/H(u\M)$, respectively. More explicitly, these kernels are given by:
	\[
		\ker(h_E)=\{ p \in u\M\sbmid h_E(p)=\bar \alpha/ E\}
	\]
	and
	\[
		\ker(\bar h_E) = \{ pH(u\M)\sbmid \bar h_E (pH(u\M)) = \bar \alpha/ E\}.
	\]
	
%	In fact, $\gal_L(T)$ acts on $[\bar \alpha]_{\equiv}/E$ in the obvious way. Composing this action with the epimorphism $f$, we get an action of $u\M$ on $[\bar \alpha]_{\equiv}/E$, and similarly $u\M/H(u\M)$ acts on $[\bar \alpha]_{\equiv}/E$. Then $h_E$ and $\bar h_E$ are given by these actions applied to the element $\bar \alpha/E$. In particular, $\ker(h_E)$ and $\ker(\bar h_E)$ are just the stabilizers of $\bar \alpha/E$ under these actions, so they are (not necessarily normal) subgroups of $u\M$ and $u\M/H(u\M)$, respectively.
	
	As an immediate corollary of Corollary~\ref{cor:galquot} and Remark~\ref{rem: g_E is quotient}, we get the following conclusion.
	
	\begin{cor}\label{cor: h_E is quotient}
		$\bar{h}_E$ is a topological quotient mapping.
	\end{cor}
	
	As a conclusion, we obtain new topological information 
	about $\cl(\bar \alpha/E)$, which will be directly used in the proofs of Theorems~\ref{thm:main_Borel} and~\ref{thm:nwg}. Recall that Proposition~\ref{prop: closure of alpha/E} tells us that the logic topology on $\cl(\bar \alpha/E)$ is trivial, so it is useless. 
	
	Before the formulation of the theorem, note that by the definitions of $\bar h_E$ and $\ker(\bar h_E)$ in terms of the action of $u\M/H(u\M)$ on $[\bar \alpha]_\equiv/E$, it follows that for any $Z \subseteq u\M/H(u\M)$ one has $\bar h_E^{-1}[\bar h_E[Z]] = Z \ker(\bar h_E)$. By  $\cl_\tau(\ker(\bar h_E)){/}\ker(\bar h_E)$ we will denote the set of left cosets of $\ker(\bar h_E)$ in $\cl_\tau(\ker(\bar h_E))$.
	
	\begin{thm}\label{thm:main theorem 3}
		Let $E$ be a bounded, invariant equivalence relation refining $\equiv$, and $\bar \alpha \in \dom (E)$. Then $\bar h_E[\cl_\tau(\ker(\bar h_E))]= \cl(\bar \alpha/E)$ (where the closure $\cl(\bar \alpha/E)$ is computed in $[\bar \alpha]_\equiv/E$).
		This means that the map $\bar h_E$ restricted to $\cl_\tau(\ker(\bar h_E))$ induces a bijection between $\cl_\tau(\ker(\bar h_E)){/}\ker(\bar h_E)$ and $\cl(\bar \alpha/E)$. Thus, $\cl(\bar \alpha/E)$ is naturally the quotient of a compact, Hausdorff group by a dense (not necessarily normal) subgroup.
	\end{thm}
	
	\begin{proof}
		The inclusion $(\subseteq)$ follows from the continuity of $\bar h_E$. For the other inclusion we need to show that $\bar h_E[\cl_\tau(\ker (\bar h_E))]$ is closed, but this follows from Corollary~\ref{cor: h_E is quotient} and an easy observation that $\bar h_E^{-1}[\bar h_E[\cl_\tau(\ker (\bar h_E))]]=\cl_\tau(\ker (\bar h_E))$ (see the paragraph preceding Theorem~\ref{thm:main theorem 3}).
		
		The second part follows from the first one together with the definitions of $\bar h_E$ and $\ker(\bar h_E)$ in terms of the action of $u\M/H(u\M)$ on $[\bar \alpha]_\equiv/E$ and the following (general) remark concerning actions: there is a natural bijection between the set of left cosets of the stabilizer of $\bar \alpha/E$ and the orbit of $\bar \alpha/E$ under the action of $\cl_\tau(\ker (\bar h_E))$. \qedhere
	\end{proof}
	
	Our considerations in this section lead to various questions which we leave for the future. For example, one can ask for which theories the objects $\M$, $u \M$ or $u\M/H(u\M)$ do not depend (up to isomorphism) on the choice of the monster model $\C$ for which they are computed, or at least when they are of bounded size. One can also try to find some classes of theories for which the natural epimorphism from $u\M/H(u\M)$ to $\gal_{KP}(T)$ (i.e.\ the composition of $\bar f$ with the quotient map from $\gal_L(T)$ to $\gal_{KP}(T)$) is an isomorphism, which could possibly lead to new examples of non G-compact theories.
	
	\section{Topological lemmas}\label{section: topological lemmas}
	
	\subsection{Technical lemma}
	
	In this subsection, we will prove a certain technical lemma (Lemma~\ref{lem: main technical lemma} below), concerning the situation from Section~\ref{section: top dyn for aut(C)}, which will be used in the proofs of the main theorems in Sections~\ref{section: smoothness} and~\ref{section: arbitrary language}.
	We take the notation from Section~\ref{section: top dyn for aut(C)}.
	
	We start from a general lemma concerning the $\tau$-topologies which will be used in the proof of Lemma~\ref{lem: main technical lemma}, but which is also interesting in its own right and may have further applications. As we said, the notation is taken from Section~\ref{section: top dyn for aut(C)}, but this particular lemma is a general observation concerning topological dynamics and it works for the Ellis semigroup of any flow.
	
	\begin{lem}\label{lem: variant}
		\phantomsection
		\label{ind:zeta}
		Let $\zeta\fcolon \cl (u\M) \to u\M$ be the function defined by $\zeta(x)=ux$ and let $\xi\fcolon u\M \to Z$ be a continuous function, where $Z$ is a regular (e.g.\ compact, Hausdorff) space and $u\M$ is equipped with the $\tau$-topology. Then $\xi \circ \zeta:\cl (u\M) \to Z$ is continuous, where $\cl (u\M)$ is equipped with the topology induced from the Ellis semigroup $EL$.
	\end{lem}
	
	\begin{figure}[H]
		\centering
		\includegraphics[width={\textwidth}]{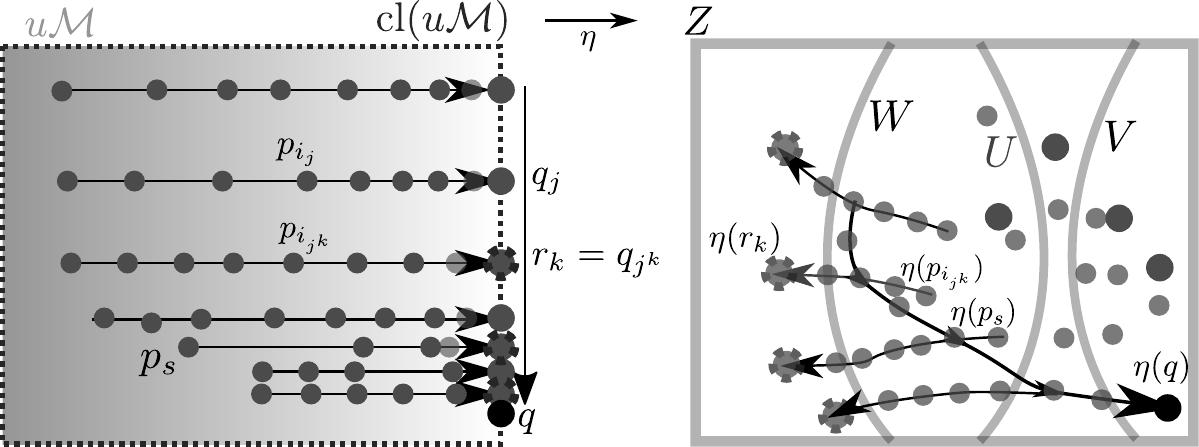}
		{Illustration of the nets described in the following proof.}
	\end{figure}
	
	\begin{proof}
		Denote $\xi \circ \zeta$ by $\eta$.
		By Lemma 1.5 from \cite[Chapter IX]{Gl}, we know that for any net $(p_i)_i$ in $u\M$ and $ p \in \cl(u\M)$ such that $\lim p_i = p$ one has $\tau\mbox{-}\!\lim p_i=up$. So, in such a situation, $\eta(p)=\xi(up)=\lim_i \xi(p_i)=\lim_i \xi(up_i)=\lim_i \eta(p_i)$.
		
		Consider any net $(q_j)_{j\in J}$ in $\cl(u\M)$ converging to $q$ in $\cl(u\M)$. The goal is to show that $\lim_j \eta(q_j)=\eta(q)$. Suppose for a contradiction that there is an open neighborhood $W$ of $\eta(q)$ and a subnet $(r_k)$ of $(q_j)$ such that all points $\eta(r_k)$ belong to $W^c$. Since $Z$ is regular, we can find open subsets $U$ and $V$ such that $W^c \subseteq U$, $\eta(q) \in V$ and $U \cap V=\emptyset$.
		
		For each $j$ we can choose a net $(p_{i^j})_{i^j \in I^j}$ in $u\M$ such that $\lim_{i^j} p_{i^j}=q_j$.
		
		For each $k$, $r_k = q_{j_{k}}$ for some $j_k \in J$, and $\eta(r_k) \in U$. Hence, since by the first paragraph of the proof $\eta(r_k)=\eta(q_{j_k})= \lim_{i^{j_k}} \eta(p_{i^{j_k}})$, we see that for big enough $i^{j_k} \in I^{j_k}$ one has $\eta(p_{i^{j_k}}) \in U$.
		
		On the other hand, let $S:=J \times \prod_{j \in J} I^j$ be equipped with the product order. For $s \in S$, put $p_s:= p_{i^{j_s}}$, where $j_s$ is the first coordinate of $s$ and $i^{j_s}$ is the $j_s$-coordinate of $s$. Since $\lim_{j \in J} q_j=q$ and $\lim_{i^j \in I^j} p_{i^j}=q_j$, we get $\lim_s p_s=q$. So, by the first paragraph of the proof, $\lim_s \eta(p_s)=\eta(q) $, and hence, for $s\in S$ big enough, $\eta(p_s) \in V$.
		
		By the last two paragraphs, we can find $j \in J$ and $i^j \in I^j$ (big enough) so that $\eta(p_{i^j}) \in U \cap V$, a contradiction as the last set is empty.
	\end{proof}
	
	In the proof of Lemma~\ref{lem: variant}, the assumption that $Z$ is regular was essentially used. However, $u\M$ is only $T_1$ with the $\tau$-topology, so we have the following question.
	
	\begin{ques}\label{question: Is zeta continuous}
		Is the function $\zeta\fcolon \cl (u\M) \to u\M$ defined by $\zeta(x)=ux$ continuous (where $\cl (u\M)$ is equipped with the topology induced from the Ellis semigroup $EL$ and $u\M$ is equipped with the $\tau$-topology)?
	\end{ques}
	
	Recall that $E$ is any bounded, invariant equivalence relation refining $\equiv$, and $\bar \alpha \in \dom(E)$. Denote by $P$ the domain of $E$. Then $[\bar \alpha]_\equiv \subseteq P$.
	
	Choose any $M\prec \C$, and put $P_M:=\{ \tp(a/M)\sbmid a\in P\}$. Let
	\[
		\hat{r}\fcolon EL \to P_M
	\]
	be defined by saying that $\hat{r}(x)$ is the restriction of the type $\pi_0(x)$ to $M$ and to the coordinates corresponding to $\bar \alpha$.
	In other words, if $x \in EL$, we take $\sigma' \in \aut(\C')$ such that $\sigma'(\bar c) \models \pi_0(x)$, and then $\hat{r}(x)=\tp(\sigma'(\bar \alpha)/M)$. Note that $\hat{r}$ is continuous.
	
	Let $r \fcolon u\M \to P_M$ and $r_{\cl}\fcolon \cl(u\M) \to P_M$ be the restrictions of $\hat{r}$ to $u\M$ and to $\cl(u\M)$, respectively (where the closure is computed in the topology on $EL$). Note that $r_{\cl}$ is continuous (with the topology on $\cl(u\M)$ inherited from $EL$). Next, $j\fcolon u\M \to u\M/H(u\M)$ denotes the quotient map. Following the notation from Lemma~\ref{lem: variant}, $\zeta\fcolon \cl (u\M) \to u\M$ is the function defined by $\zeta(x)=ux$. We also have the natural function from $P_M$ to $P/E$, mapping a type in $P_M$ to the $E$-class of its realization.

	The commutative diagrams below contain the relevant maps (commutativity follows easily from the definitions of all these maps) and will play an important role also in further sections.
	
	\begin{figure}[H]
		\centering
		\begin{tikzcd}
			EL\arrow[r,"\hat{f}"]\arrow[d,"\hat{r}"]&\gal_L(T)\arrow[d,"g_E"] \\
			P_M\arrow[r] & P/E
		\end{tikzcd}
	\end{figure}	
	
	\begin{figure}[H]
		\centering
		\begin{tikzcd}
			\cl(u\M)\arrow[r,"\zeta"]&u\M\arrow[r,"j"]\arrow[d,"r"]\arrow[rrd,"h_E"] & (u\M)/H(u\M) \arrow[r,"\bar f"]\arrow[dr,"\bar h_E"]& \gal_L(T) \arrow[d,"g_E"] \\
			\cl(u\M)\arrow[r,"r_{\cl}"]&P_M\arrow[rr] & & P/E
		\end{tikzcd}
	\end{figure}
	
	\begin{center}
		\renewcommand{\arraystretch}{1.3}
		\begin{tabular}{r| l || r|l }
			$j$ & the quotient map 
			&
			$\hat f$& the semigroup homomorphism (page \pageref{ind:fhat})
			\\ \hline
			$g_E$& the orbit map of $[\bar\alpha]_E$ (page \pageref{ind:gE})
			&
			$f$ & restriction of $\hat f$ to $u\M$ (page \pageref{ind:f})
			\\ \hline
			$h_E$ & $g_E\circ f$ (page \pageref{ind:hE}) 
			&
			$\bar f$& factor of $f$ (page \pageref{ind:fbar}) 
			\\ \hline 
			$\bar h_E$& $g_E\circ \bar f$ (page \pageref{ind:barhE})
			&
			$\hat r$ & the ``quotient map'' described above
			\\ \hline 
			$\zeta$ & multiplication by $u$ (just defined)
			&
			$r_\textrm{cl}$ & restriction of $\hat r$
			\\ \hline
			& 
			&
			$r$ & restriction of $\hat r$
		\end{tabular}
	\end{center}
	
	Recall that by $\CLO(X)$ we mean the family of closed subsets of a given topological space $X$, and the members of $\Souslin (\CLO(X))$ are called analytic subsets of $X$ (see Definition~\ref{definition: analytic sets}).
		
	\begin{lem}\label{lem: main technical lemma}
		Assume that $B$ is an $E^M$-saturated, analytic subset of $P_M$. 
%which belongs to $\Souslin (\CLO(P_M))$ (i.e.\ it can be obtained by the Souslin operation applied to some closed subsets of $P_M$). 
		Then $j[r^{-1}[B]]$ is an analytic subset of $u\M/H(u\M)$.
	\end{lem}
	
	\begin{proof}
		First, we give a short proof using Lemma~\ref{lem: variant}. Then we will give an alternative proof, using a less general (and more model-theoretic) argument than Lemma~\ref{lem: variant}, but giving more detailed information about $r^{-1}[B]$ (in particular, that  $r^{-1}[B]$ is also an analytic subset of $u\M$).
		
		\begin{clm*}
			$r^{-1}[B]=\zeta[{r_{\cl}}^{-1}[B]]$.
		\end{clm*}
		
		\begin{clmproof}[Proof of Claim]
			The inclusion $(\subseteq)$ is clear, as $r$ is the restriction of $r_{\cl}$ to $u\M$ and $\zeta$ restricted to $u\M$ is the identity function.
			
			To show $(\supseteq)$, consider any $x \in \zeta[{r_{\cl}}^{-1}[B]]$. Take $y \in {r_{\cl}}^{-1}[B]$ such that $\zeta(y)=x$.
			
			Since by Proposition~\ref{prop: semigroup epi}, $\hat{f}\fcolon EL \to \gal_L(T)$ is a semigroup homomorphism, we see that $\hat{f}(y)=\hat{f}(u)\hat{f}(y)=\hat{f}(uy)=\hat{f}(x)$. This implies that $(g_E \circ \hat{f})(y) = (g_E \circ \hat{f})(x)$. Therefore, $r_{\cl}(y) \Er ^M r_{\cl}(x)$.
			
			Since $r_{\cl}(y) \in B$ and $B$ is $E^M$-saturated, we conclude that $r_{\cl}(x) \in B$. As $x \in u\M$ and $r_{\cl}\restr_{u\M}=r$, we get that $x\in r^{-1}[B]$.
		\end{clmproof}
		
		Since $r_{\cl}$ is continuous and $B$ is analytic, Remark~\ref{remark: preservation of analyticity by images and preimages}(1) implies that $r_{\cl}^{-1}[B]$ is an analytic subset of $\cl(u\M)$. Since $\cl(u\M)$ is compact and $u\M/H(u\M)$ is Hausdorff, and, by Lemma~\ref{lem: variant}, $j \circ \zeta \fcolon \cl(u\M) \to u\M/H(u\M)$ is continuous, using Remark~\ref{remark: preservation of analyticity by images and preimages}(2), we conclude that $(j\circ \zeta)[{r_{\cl}}^{-1}[B]]$ is an analytic subset of $u\M/H(u\M)$. On the other hand, by the claim, $j[r^{-1}[B]]=(j \circ \zeta)[{r_{\cl}}^{-1}[B]]$. So the proof is 
		complete.	
	\end{proof}

	\begin{proof}[Alternative proof of Lemma~\ref{lem: main technical lemma}]
		Since $B \in \Souslin(\CLO(P_M))$,
		\[
			B=\Souslin_s F_s=\bigcup_{s \in \omega^\omega} \bigcap_n F_{s\restr_n}
		\]
		for a regular Souslin scheme $(F_s)_{s \in \omega^{<\omega}}$ consisting of closed subsets of $P_M$. Therefore,
		\[
			r^{-1}[B]= \Souslin_s r^{-1}[F_s]= \bigcup_{s \in \omega^\omega} \bigcap_n r^{-1}[F_{s\restr_n}].
		\]
		
		It is clear that $r$ is continuous, but with $u\M$ equipped with the topology induced from the topology on the Ellis semigroup $EL$ (which is stronger than the $\tau$-topology). So each $r^{-1}[F_{s\restr_n}]$ is closed in this topology, but not necessarily in the $\tau$-topology, which is a problem. We resolve it by proving the following
		
		\begin{clm*}
			$r^{-1}[B]= \Souslin_s \cl_\tau(r^{-1}[F_s])= \bigcup_{s \in \omega^\omega} \bigcap_n \cl_\tau(r^{-1}[F_{s\restr_n}])$.
		\end{clm*}
		
		\begin{clmproof}
			Only the inclusion $(\supseteq)$ requires a proof.
			Fix any $s\in \omega^\omega$.
			Pick any element $x \in \bigcap_{n}\cl_{\tau}(r^{-1}[F_{s\restr_n}])\subseteq u\M$. We need to show that $x\in r^{-1}[B]$.
			
			By the choice of $x$, one can find nets $(\sigma_j)$ in $\aut(\C)$ and $(x_j)$ in $u\M$ with the properties:
			
			\begin{enumerate}[label=(\roman{*})]
				\item for every $n$ there is $j_n$ such that for all $j>j_n$, $x_j \in r^{-1}[F_{s\restr_n}]$,
				\item $\lim_j \sigma_j=u$,
				\item $\lim_j \sigma_j x_j=x $.
			\end{enumerate}
			
			By compactness of $\cl (u\M)$ (where the closure is computed in the topology on $EL$), there is a subnet $(y_k)$ of $(x_j)$ such that $\lim_k y_k$ exists in $\cl (u\M)$; denote this limit by $y$. By (i), we get
			
			\begin{equation}\label{eq:3}
				y \in \bigcap_n \cl (r^{-1}[F_{s\restr_n}]).
			\end{equation}
			
			Using (ii), (iii), the fact that $\lim_k y_k=y$ and 
			an argument as in the proof of Theorem~\ref{thm:main theorem 1}, we get that there are $\bar a \models \pi_0(x)$, $\bar d \models \pi_0(y)$, and $\bar b$ such that
			$\bar b \Er _L \bar c\;\; \mbox{and}\;\; \bar c \bar d \equiv \bar b \bar a$. Take $\sigma' \in \aut(\C')$ such that $\sigma'(\bar c \bar d)=\bar b \bar a$. Since $\bar b \Er _L \bar c$ and $\bar c$ is a model, we get that $\sigma' \in \autf_L(\C')$, so $\bar d \Er _L \bar a$. Hence,
			\begin{equation}\label{eq:4}
				\pi_0(x) \Er ^\C_L \pi_0(y).
			\end{equation}
			
			Define the relation $\hat{E}'$ on $\cl(u\M)$ by
			\[
				p\mathrel{\hat{E}'}q \iff r_{\cl}(p)\Er ^M r_{\cl}(q).
			\]
			Since $r_{\cl}$ is continuous and coincides with $r$ on $u\M$,
			\[
				{r_{\cl}}^{-1}[B]= \bigcup_{\eta \in \omega^\omega} \bigcap_n {r_{\cl}}^{-1}[F_{\eta\restr_n}] \supseteq \bigcup_{\eta \in \omega^\omega} \bigcap_n \cl(r^{-1}[F_{\eta\restr_n}]).
			\]
			Hence, by (\ref{eq:3}), we get that $y \in {r_{\cl}}^{-1}[B]$.
			On the other hand, by (\ref{eq:4}), we have that $x \mathrel{\hat{E}'} y$. But, since $B$ is $E^M$-saturated, we also see that ${r_{\cl}}^{-1}[B]$ is $\hat{E}'$-saturated. Therefore, $x \in {r_{\cl}}^{-1}[B]$. Since $x \in u\M$, we conclude that $x \in r^{-1}[B]$,	which completes the proof.
		\end{clmproof}
		
		By the claim, $r^{-1}[B]$ is an analytic subset of $u\M$ (where $u\M$ is equipped with the $\tau$-topology). As $u\M$ is compact, $u\M/H(u\M)$ is Hausdorff and $j$ is continuous, using Remark~\ref{remark: preservation of analyticity by images and preimages}(2), we conclude that $j[r^{-1}[B]]$ is an analytic subset of $u\M/H(u\M)$.
	\end{proof}
	
	We have the following corollary of the claim in the first proof of Lemma~\ref{lem: main technical lemma}.
	
	\begin{rem}\label{remark: to use in 5.12}
		Assume that $B$ is an $E^M$-saturated, $F_\sigma$ subset of $P_M$. 
		Then $j[r^{-1}[B]]$ is an $F_\sigma$ subset of $u\M/H(u\M)$.
	\end{rem}
	\begin{proof}
		By the claim in the first method of the proof of Lemma~\ref{lem: main technical lemma}, $j[r^{-1}[B]]=(j \circ \zeta)[{r_{\cl}}^{-1}[B]]$. On the other hand, since $\cl(u\M)$ is compact and $u\M/H(u\M)$ is Hausdorff, using continuity of the functions $r_{\cl}$ and $j \circ \zeta$, we get that $(j \circ \zeta)[{r_{\cl}}^{-1}[B]]$ is $F_\sigma$.
	\end{proof}

	\subsection{Analytic sets and type spaces}
	The proposition below shows that the Souslin operation $\Souslin$ is quite well-behaved in
	the model-theoretic context. 
%(similarly to $F_\sigma$ and Borel sets considered in Fact~\ref{fct:eqcmp}). 
	We will use it later in the proof of Theorem~\ref{thm:nwg}.
	
	\begin{prop}\label{prop: Souslin and parameters}$\,$
		\begin{enumerate}[label=\roman{*}),nosep]
			\item
			An $A$-invariant set $X$ is in $\Souslin(A\textrm{-type-definable})$ iff $X_A$ is in the class $\Souslin(\CLO(S(A)))$ (in other words, $X_A$ is analytic).
			\item
			For any $A$-invariant and $B$-invariant set $X$, if $X_B$ is in $\Souslin(\CLO(S(B)))$ then the set $X_A$ is in the class $\Souslin(\CLO(S(A)))$.
			\item
			Suppose $X$ is an $A$-invariant set and that $X$ is in $\Souslin(B\textrm{-type-definable})$ (which implies that it is also $B$-invariant). Then $X$ is in the class $\Souslin(A\textrm{-type-definable})$.
		\end{enumerate}
	\end{prop}
	\begin{proof}
		Recall that for an $A$-invariant set $X$, $X_A:=\{\tp(a/A)\sbmid a\in X\}$.
		
		Point (i) is rather clear. Namely, since $A$-type-definable sets are $A$-invariant, if
		\[
			X=\bigcup_{\eta\in \omega^\omega} \bigcap_{n\in \omega} K_{\eta\restr n},
		\]
		where the sets $K_{\eta\restr n}$, $\eta \in \omega^\omega$, $n \in \omega$, are $A$-type-definable, then
		\[
			X_A=\bigcup_{\eta\in \omega^\omega} \bigcap_{n\in \omega} (K_{\eta\restr n})_A.
		\]
		Conversely, if
		\[
			X_A=\bigcup_{\eta\in \omega^\omega} \bigcap_{n\in \omega} [\pi_{\eta\restr n}]
		\]
		for some partial types $\pi_{\eta\restr n}$ over $A$, where $\eta \in \omega^\omega$, $n \in \omega$, then
		\[
			X=\bigcup_{\eta\in \omega^\omega} \bigcap_{n\in \omega} \pi_{\eta\restr n}(\C).
		\]
		
		(ii) We can assume without loss of generality that $A\subseteq B$ or $B \subseteq A$. In the latter case, the conclusion easily follows by (i). Alternatively, it follows from Remark~\ref{remark: preservation of analyticity by images and preimages}(1) after noticing that the preimage of $X_B$ by the restriction map $S(A) \to S(B)$ is exactly $X_A$. In the former case, the conclusion follows from Remark~\ref{remark: preservation of analyticity by images and preimages}(2) after noticing that the image of $X_B$ by the restriction map $S(B) \to S(A)$ is exactly $X_A$.
		
		Point (iii) readily follows from (i) and (ii).
	\end{proof}

	\subsection{Mycielski-style lemma}
	The next -- purely topological -- proposition (and its corollary for locally compact groups) is a generalization of a classical theorem of Mycielski for Polish spaces \cite[Theorem 5.3.1]{SuG}, and it will be useful for both parts of Theorem~\ref{thm:nwg}.
	
	\begin{prop}[generalization of a theorem of Mycielski]
		Suppose $E$ is a meager equivalence relation on a locally compact, Hausdorff space $X$. Then $\lvert X/E\rvert \geq 2^{\aleph_0}$.
	\end{prop}
	\begin{proof}
		The proof mimics that of the classical theorem for Polish spaces (for example see \cite[Theorem 5.3.1]{SuG}), only we replace the notion of diameter by compactness.
		
		Firstly, we can assume without loss of generality that $X$ is compact. This is because we can restrict our attention to the closure $\overline U$ of a small open set $U$: $E$ restricted to $\overline U$ is still meager, and if we show that $\overline U/E$ has the cardinality of at least the continuum, clearly the same will hold for $X/E$.

		Suppose $E\subseteq \bigcup_{n\in \omega} F_n$ with $F_n\subseteq X^2$ closed, nowhere dense. We can assume
		that the sets $F_n$ form an increasing sequence. We will define a family of nonempty open sets $U_s$ with $s\in 2^{<\omega}$, recursively with respect to the length of $s$, such that:
		\begin{itemize}
			\item
			$\overline{U_{s0}},\overline{U_{s1}}\subseteq U_s$,
			\item
			if $s\neq t$ and $s,t\in 2^{n+1}$, then $(U_s\times U_t)\cap F_n=\emptyset$.
		\end{itemize}
		
		Then, by compactness, for each $\eta\in 2^\omega$ we will find a point $x_\eta\in \bigcap_n U_{\eta\restr n}$. It is easy to see that this will yield a map from $2^\omega$ into $X$ such that any two distinct points are mapped to $E$-unrelated points.
		
		The construction can be performed as follows:
		\begin{enumerate}
			\item
			For $s=\emptyset$, we put $U_\emptyset=X$.
			\item
			Suppose we already have $U_s$ for all $\lvert s\rvert \leq n$, satisfying the assumptions.
			\item
			By compactness (more precisely, regularity), for each $s\in 2^n$ and $i\in\{0,1\}$ we can find a nonempty open set $U_{si}'$ such that $\overline{U_{si}'}\subseteq U_s$.
			\item
			For each (ordered) pair of distinct $\sigma,\tau\in 2^{n+1}$, the set $(U'_\sigma\times U'_\tau)\setminus F_n$ is a nonempty open set (because $F_n$ is closed, nowhere dense), so in particular, $U'_\sigma\times U'_\tau$ contains a smaller (nonempty, open) rectangle $U''_\sigma\times U''_\tau$ which is disjoint from $F_n$.
			\item
			Repeating the procedure from the previous point recursively, for each ordered pair $(\sigma,\tau)$, we obtain for each $\sigma\in 2^{n+1}$ a nonempty open set $U_\sigma\subseteq U_\sigma'$ such that for $\sigma\neq \tau$ we have $(U_\sigma\times U_\tau)\cap F_n=\emptyset$. It is easy to see that the sets $U_\sigma$ satisfy the inductive step for $n+1$.\qedhere
		\end{enumerate}		
	\end{proof}
	
	\begin{cor}\label{cor: from Mycielski}
		Suppose $G$ is a locally compact, Hausdorff group and $H$ is a subgroup which has the Baire property, but is not open. Then $\lvert G/H\rvert\geq 2^{\aleph_0}$.
	\end{cor}
	\begin{proof}
		It follows from Pettis theorem (see Fact~\ref{fct: Pettis}) that a non-meager Baire subgroup of a topological group is open, so, in our case, $H$ is meager. It is readily derived that the orbit equivalence relation of $H$ acting by left translations on $G$ is meager (because the map $(x,y)\mapsto xy^{-1}$ is continuous and open, so preimages of meager sets are meager), so we obtain the corollary immediately from the preceding proposition.
	\end{proof}
	
	\section{Application to bounded, invariant equivalence relations in a countable language}\label{section: smoothness}
	
	The main result of this section is Theorem~\ref{thm:main_Borel}. Before the proof, we formulate several immediate corollaries which in particular answer open questions mentioned in Subsection~\ref{subsection: bounded relations}.
	
	In this section, we have a blanket assumption that the theory is countable, and that the types we consider are in countably many variables.
	
	\begin{thm}\label{thm:main_Borel}
		We are working in a monster model $\C$ of a complete, countable theory. Suppose we have:
		\begin{itemize}[nosep]
			\item
			a $\emptyset$-type-definable, countably supported set $X$,
			\item
			a bounded, invariant equivalence relation $E$ on $X$,
			\item
			a pseudo-closed (i.e.\ type-definable) and $E$-saturated set $Y\subseteq X$.
		\end{itemize}
		Then, for every type $p \in X_{\emptyset}$ consistent with $Y$, either $E\restr_{Y\cap p(\C)}$ is type-definable, or $E\restr_{Y\cap p(\C)}$ is not smooth.
	\end{thm}
	
%	To be precise, in the above formulation, the empty relation on the empty domain is assumed to be smooth and (rather unusually) type-definable by definition.
	
	Recall that Fact~\ref{fct: type-definability imply smoothness} tells us that type-definability of an equivalence relation implies its smoothness. Hence, in all our dichotomies (formulated in this section) between type-definability and non-smoothness of a given relation, the fact that at most one of these options holds is always clear. Note also that these dichotomies can be formulated in the equivalent form saying that type-definability of the appropriate relation is equivalent to its smoothness.
	
	By Remark~\ref{rem:type-definability_of_relations}, in the context of Theorem~\ref{thm:main_Borel}, if $Y \cap p(\C) \ne \emptyset$, then the following conditions are equivalent: (1) $E\restr_{Y\cap p(\C)}$ is type-definable; (2) some [every] class of $E\restr_{Y\cap p(\C)}$ is type-definable; (3) $E\restr_{p(\C)}$ is type-definable. The condition that $E\restr_{Y\cap p(\C)}$ is not smooth implies that $E\restr_{p(\C)}$
	and $E\restr_Y$ are non-smooth. Recall that if $E$ is Borel, then non-smoothness of the above relations is equivalent to the fact that $\EZ$ Borel reduces to each them (Fact~\ref{fct:Harrington-Kechris-Louveau dichotomy}); in particular, for Borel relations non-smoothness implies having $2^{\aleph_0}$ classes.

	\begin{cor}\label{cor:Borel1}
		Take the assumptions of Theorem~\ref{thm:main_Borel}.
		Then either the restriction of $E$ to any complete type over $\emptyset$ consistent with $Y$ is type-definable,
		or $E\restr_Y$ is not smooth. If $E$ refines $\equiv$, then the first possibility is equivalent to the condition that every class of $E$ contained in $Y$ is type-definable.
	\end{cor}

	The most interesting instances of Theorem~\ref{thm:main_Borel} are when $Y$ is a complete type over $\emptyset$ or when it is one Kim-Pillay type. This is described in the next two corollaries.
	
	\begin{cor}\label{cor:mainA}
		If $E$ is a bounded, invariant equivalence relation defined on a single complete type $p \in S(\emptyset)$ (in countably many variables), then either $E$ is type-definable, or it is non-smooth.
	\end{cor}
	
	\begin{proof}
		Apply Theorem~\ref{thm:main_Borel} to $X=Y:=p(\C)$.
	\end{proof}

	The next corollary is a generalization of a theorem of Kaplan, Miller and Simon (see Fact~\ref{fct: KPS theorem}) from Lascar strong types to arbitrary bounded, invariant equivalence relations. Strictly speaking, it is a generalization of a key corollary of this theorem (namely, of Conjecture 1 from \cite{KPS}) concerning only the dichotomy between the condition that the restriction of the relation in question to a fixed Kim-Pillay type has one class and the condition that it is non-smooth (i.e.\ it does not use diameters in its formulation, as diameters are not available at such a level of generality).
	
	\begin{cor}\label{cor:mainB}
		Assume $E$ is a bounded, invariant equivalence relation (on some product $X$ of countably many sorts) refining $E_{KP}$. Then, for any $a \in X$, either $E$ restricted to $[a]_{E_{KP}}$ has only one class, or it is non-smooth. Thus, if $E \ne E_{KP}$, then $E$ is non-smooth.
	\end{cor}
	
	\begin{proof}
		Let $p=\tp(a)$. Apply Theorem~\ref{thm:main_Borel} to $Y:=[a]_{E_{KP}}$, and use \cite[Lemma 4.18]{LaPi} (which says that $E_{KP}$ restricted to $p(\C)$ is the finest bounded, $\emptyset$-type-definable equivalence relation on $p(\C)$)				 together with Remark~\ref{rem:type-definability_of_relations}.
	\end{proof}
	
	The next corollary answers Question 4.11 from \cite{KrRz} (i.e.\ Question~\ref{ques:quest1_from_KrRz} above).
	
	\begin{cor}\label{cor: answer to question from or paper}
		Suppose that $E$ is a bounded, invariant equivalence relation which is defined on a single complete type over $\emptyset$ or which refines $E_{KP}$.
		Then $E$ is smooth iff $E$ is type-definable.
		
		If we drop the assumption that $E$ is defined on a single complete type over $\emptyset$ or refines $E_{KP}$, then smoothness need not imply type-definability (while the converse is always true).
	\end{cor}
	
	\begin{proof}
		The first part of the corollary follows immediately from Corollaries~\ref{cor:mainA} and~\ref{cor:mainB}.
		The second part was demonstrated by Example 4.4 in \cite{KrRz}.
	\end{proof}
	
	Here, we should explain why in the discussion after Question 4.11 in \cite{KrRz} it is said that also some kind of ``definability'' assumption (like Borelness) on $E$ is needed in order to get that smoothness of $E$ implies type-definability, whereas in the above corollary we do not assume anything like that. Note that smoothness of $E$ (in the sense of this paper) implies that $E$ is Borel (see Subsection~\ref{subsection: descriptive set theory}). The point is that in \cite{KrRz} we did not specify what smoothness of non-Borel equivalence relations means and actually we thought about such a variant of this notion for which non-smoothness would imply possessing $2^{\aleph_0}$ classes. In this paper, we use the same definition of smoothness as for Borel relations, and for such a definition it may happen that a non-smooth relation has only two classes. For instance, see Example 4.7 from \cite{KrRz} (which originally comes from \cite{KaMi}) of an invariant equivalence relation $E$ on a single complete type over $\emptyset$ which is not type-definable, and so non-smooth by Theorem~\ref{thm:main_Borel}, and has only two classes; note that this relation $E$ is not Borel,
% because any invariant equivalence relation on a single complete type over $\emptyset$ which is not type-definable must be non-smooth by Theorem~\ref{thm:main_Borel}, and so, if it is Borel, it  has $2^{\aleph_0}$ classes (as $\EZ$ reduces into it).
	because otherwise, being non-smooth, it would have $2^{\aleph_0}$ classes, by the Silver dichotomy or the Harrington-Kechris-Louveau dichotomy (Facts~\ref{fct:silver},~\ref{fct:Harrington-Kechris-Louveau dichotomy}).
	
	The next corollary solves Problem 3.22 from \cite{KaMi} mentioned below Fact~\ref{fct:mainA}. To be precise, as was discussed below Fact~\ref{fct:mainA}, in \cite{KaMi} the formulation of the problem is slightly more general, but we find this generalization to be rather of a technical nature and we do not deal with it in this paper.
	
	\begin{cor}
		In Fact~\ref{fct:mainA}, one can replace the `orbitality on types' assumption by the assumption that $E$ refines $\equiv$. Without the assumption that $E$ refines $\equiv$, the conclusion of Fact~\ref{fct:mainA} may fail.
	\end{cor}

	\begin{proof}
		The fact that if we just drop the `orbitality on types' assumption (without assuming that $E$ refines $\equiv$ instead), then Fact~\ref{fct:mainA} is not any longer true is witnessed by the very simple Example 2.25 from \cite{KrRz}, namely: $T:=\Th(\R,+,\cdot,<)$, $X=Y:=\C$, and $E$ is the total (i.e.\ with only one class) relation on $\C$. It is explained in \cite{KrRz} how to find a normal form for $E$ with respect to which the only equivalence class of $E$ has infinite diameter, but clearly $E$ is smooth.
		
		Now, we check that if we replace the `orbitality on types' assumption by the assumption that $E$ refines $\equiv$, then Fact~\ref{fct:mainA} remains true. By Corollary 2.24 from \cite{KrRz} (which immediately follows from Newelski's theorem, i.e.\ from Fact~\ref{fct:twN}), the assumption that $C$ is of infinite diameter implies that it is not type-definable. Thus, by Corollary~\ref{cor:Borel1}, we get that $E\restr_Y$ is not smooth.
	\end{proof}

	The final corollary of Theorem~\ref{thm:main_Borel} concerns definable groups, and is a generalization of Fact~\ref{fct:mainG_KR}. Namely, we drop the assumptions that $H$ is normal and $F_\sigma$.
	\begin{cor}
		\label{cor:group_borel}
		Assume the language is countable.
		Suppose that $G$ is a $\emptyset$-definable group and $H\leq G$ is an invariant subgroup of bounded index. Suppose in addition that $K\geq H$ is a type-definable subgroup of $G$. Then $E_H\restr_{K}$ is smooth if and only if $H$ is type-definable (where $E_H$ is the relation of lying in the same right coset of $H$.)
	\end{cor}
	\begin{proof}
		We use the construction from Fact~\ref{fct:affine_sort}. We apply Theorem~\ref{thm:main_Borel} to $E:=E_{H,X}$ and $Y:=K\cdot x_0$, and then use Fact~\ref{fct:grres} and Remark~\ref{rem:grresplus} -- more precisely, the parts saying that:
		\begin{itemize}
			\item
			$K\cdot x_0$ is type-definable, 
			\item
			$H$ is type-definable if and only if $E_{H,X}\restr_{K\cdot x_0}$ is,
			\item
			$E_H \restr_K$ is smooth if and only if $E_{H,X}\restr_{K\cdot x_0}$ is (which is witnessed by the homeomorphism from Fact~\ref{fct:grres}).\qedhere
		\end{itemize}
	\end{proof}
	
	Now, we turn to the proof of Theorem~\ref{thm:main_Borel}. In the course of this proof, we take the notation (in particular, the names for all the functions) from Sections~\ref{section: top dyn for aut(C)} and~\ref{section: topological lemmas}.
	
	\begin{proof}[Proof of Theorem~\ref{thm:main_Borel}.]
		It is clear that Theorem~\ref{thm:main_Borel} does not depend on the choice of the monster model in which we are working.	
		So, in this proof, we are working in the bigger monster model $\C' \succ \C$ which was used to define $\bar h_E$ in Section~\ref{section: top dyn for aut(C)}.
		
		First of all, without loss of generality, we can assume that $X=p(\C')$ is the set of realizations of a single complete type $p$ over $\emptyset$.
		
		Take any $\bar \alpha \in Y(\C)$. By Remark~\ref{rem:type-definability_of_relations}, we know that $E$ is type-definable iff $[\bar \alpha]_E$ is type-definable iff $E\restr_Y$ is type-definable.
		
		Assume that $E\restr_Y$ is smooth.
		To prove the theorem, \emph{ we need to show that ${[\bar \alpha]_E}$ is type-definable}.	
		
		Let $Y_{\bar \alpha}$ be the collection of $\bar \beta \in X$ such that $\bar \beta /E \in \cl (\bar \alpha/E)$. Then $Y_{\bar \alpha}$ is type-definable, $E$-saturated and contained in $Y$.
		Thus, $E\restr_{Y_{\bar \alpha}}$ is smooth, and we can assume that $Y=Y_{\bar \alpha}$.
		Then, for every $\sigma' \in \aut(\C')$, the condition $\sigma'(\bar \alpha/E)=\bar \alpha/E$ implies that $\sigma'[Y]=Y$
		(because automorphisms of $\C'$ induce homeomorphisms of $X/E$, so $\sigma'$ takes $\cl(\bar \alpha/E)$ to $\cl( \sigma'(\bar\alpha)/E)$).
		
		Put
		\[
			S:=\{\sigma'/\autf_L(\C') \in \gal_L(T)\sbmid \sigma'[Y]=Y\}.
		\]
		Then $S$ is closed by Lemma~\ref{lem:lem_closed}. Now, we define
		\[
			(u\M)_S:=f^{-1}[S],
		\]
		which is a $\tau$-closed subgroup of $u\M$ by the continuity of $f$ (see Theorem~\ref{thm:main theorem 1}). By the conclusion of the preceding paragraph, we get
		
		\begin{equation}
			\ker(h_E) \leq (u\M)_S \; \; \mbox{and} \;\; \ker(\bar h_E) \leq (u\M)_S/H(u\M).
		\end{equation}
		
		We aim to show that $\ker(\bar h_E)$ is closed in the compact, Hausdorff group $u\M/H(u\M)$ which will allow us to finish the proof quickly, using Theorem~\ref{thm:main theorem 3}. To show that $\ker(\bar h_E)$ is closed, it is enough to check that $\ker(\bar h_E)$ (treated as a subgroup of the compact, Hausdorff group $(u\M)_S/H(u\M)$) satisfies the assumptions of Fact~\ref{fct:Miller}, which is done in Claims 2(ii) and 3 below.
		
		Choose any countable $M\prec \C$. Recall that after Question~\ref{question: Is zeta continuous} we defined the natural restriction function
		\[
			r \fcolon u\M \to p(\C')_M.
		\]
		From the definition of $S$ and $r$, it is clear that
		$g_E[S] \subseteq Y/E$ and $r[(u\M)_S] \subseteq Y_M$.
		
		Thus, we have the following commutative diagram of functions considered below. More precisely, these are restrictions to some smaller domains of the functions considered in Sections~\ref{section: top dyn for aut(C)},~\ref{section: topological lemmas} and also in the above observations, but we do not introduce new names for these restrictions. The natural function from $Y_M$ to $Y/E$ taking a type in $Y_M$ to the $E$-class of its realization will be denoted by $\rho$.
		
		\begin{figure}[H]
			\label{fig:cd2}
			\centering
			\begin{tikzcd}
				(u\M)_S\arrow[r,"j"]\arrow[d,"r"]\arrow[rrd,"h_E"] & (u\M)_S/H(u\M) \arrow[r,"\bar f"]\arrow[dr,"\bar h_E"]& S \arrow[d,"g_E"] \\
				Y_M\arrow[rr,"\rho"] & & Y/E
			\end{tikzcd}
			\caption{Commutative diagram of maps considered below}
		\end{figure}
		
		Define a relation $E'$ on $(u\M)_S$ by
		\[
			x\Er 'y \iff r(x)\Er ^Mr(y),
		\]
		and $E''$ on $(u\M)_S/H(u\M)$ as lying in the same left coset modulo $\ker (\bar h_E)$.
		Of course, $j\fcolon(u\M)_S \to (u\M)_S/H(u\M)$ is the quotient map.\\
		
		\begin{clm}
			For any $x,y \in (u\M)_S$, $r(x) \Er ^M r(y)$ iff $x\Er 'y$ iff $j(x)\Er ''j(y)$.
		\end{clm}
		
		\begin{clmproof}[Proof of Claim 1]
			Only the second equivalence requires an explanation. 
%			There are $\sigma', \tau' \in \autf_L(\C')$ such that $\sigma'(\bar c) \models \pi_0(x)$ and $\tau'(\bar c) \models \pi_0(y)$.			Then $r(x)=\tp(\sigma'(\bar \alpha)/M)$ and $r(y)=\tp(\tau'(\bar \alpha)/M)$.Thus, $x\Er 'y$ iff $\sigma'(\bar \alpha) \Er \tau'(\bar \alpha)$ iff $\tau'^{-1}\sigma' (\bar \alpha/E)=\bar \alpha/E$ iff $j(y)^{-1}j(x) \in \ker (\bar h_E)$.
			Roughly speaking, it follows from the commutativity of the above diagram, but we give the detailed sequence of equivalent conditions: $x \Er' y \iff r(x) \Er^M r(y) \iff (\rho \circ r)(x) = (\rho \circ r)(y) \iff (g_E \circ \bar f \circ j)(x) =  (g_E \circ \bar f \circ j)(y) \iff (g_E \circ \bar f)(j(x)) =  (g_E \circ \bar f)(j(y)) \iff \bar h_E(j(x)) = \bar h_E(j(y)) \iff j(y)^{-1}j(x) \in \ker(\bar h_E) \iff j(x) \Er'' j(y)$.
		\end{clmproof}
		
		Recall that we have assumed that $E\restr_Y$ is smooth, which by definition means that $E^M\restr_{Y_M}$ is smooth. Then Fact~\ref{fct: separating family} gives us a countable family $\{B_i\sbmid i\in \omega\}$ of Borel ($E^M$-saturated) subsets of $Y_M$, separating classes of $E^M\restr_{Y_M}$.\\
		
		\begin{clm}$\,$
			\begin{enumerate}[nosep,label=\roman{*})]
				\item
				The family $\{r^{-1}[B_i]\sbmid i \in \omega\}$ separates classes of $E'$, and so consists of $E'$-saturated sets.
				\item
				The family $\{j[r^{-1}[B_i]]\sbmid i \in \omega\}$ separates classes of $E''$, and so consists of $E''$-saturated (i.e.\ right $\ker (\bar h_E)$-invariant) sets.
			\end{enumerate}
		\end{clm}
		
		\begin{clmproof}[Proof of Claim 2]	
			(i) By definition, we have $x\Er 'y \iff r(x) \Er ^M r(y)$. Thus, $r^{-1}[[r(x)]_{E^M}]=[x]_{E'}$ for $x \in (u\M)_S$. 
%Since the sets $B_i$ are $E^M$-saturated, this implies that the preimages $r^{-1}[B_i]$ are $E'$-saturated. 
			Since $\{B_i\sbmid i\in \omega\}$ separates classes of $E^M\restr_{Y_M}$, for any $x \in (u\M)_S$
			we have $[r(x)]_{E^M} = \bigcap \{ B_i\sbmid r(x) \in B_i\}$, so
			\[
				[x]_{E'}=r^{-1}[[r(x)]_{E^M}]=\bigcap \{r^{-1}[B_i]\sbmid x \in r^{-1}[B_i]\}.
			\]
			This means that $\{r^{-1}[B_i]\sbmid i \in \omega\}$ separates classes of $E'$.\\[1mm]
			(ii) By Claim 1, we have $x\Er 'y \iff j(x)\Er ''j(y)$. 
%We also know that $j$ is onto. 
%Thus, since by (i), the preimages $r^{-1}[B_i]$ are $E'$-saturated, their images $j[r^{-1}[B_i]]$ are $E''$-saturated. Moreover, 
			Hence, since by (i) the preimages $r^{-1}[B_i]$ are $E'$-saturated, we see that for every $i$, $j^{-1}[j[r^{-1}[B_i]]]=r^{-1}[B_i]$.
			%for every $i$, $j^{-1}[j[r^{-1}[B_i]]]=r^{-1}[B_i]$ (because $r^{-1}[B_i]$ is $E'$-saturated, so it is a union of left cosets of $H(u\M)$). 
			Therefore, for any $I \subseteq \omega$,
			\[
				j\left[\bigcap_{i \in I} r^{-1}[B_i]\right]=\bigcap_{i \in I}j[r^{-1}[B_i]].
			\]
			Since by (i) the family $\{r^{-1}[B_i]\sbmid i \in \omega\}$ separates classes of $E'$ and $j$ is onto, this implies that the family $\{j[r^{-1}[B_i]]\sbmid i \in \omega\}$ separates classes of $E''$. Indeed, for any $z \in (u\M)_S/H(u\M)$, the preimage $j^{-1}(z)$ is nonempty, and choosing $x \in j^{-1}(z)$, we have that $j^{-1}(z) \subseteq [x]_{E'}$ and
			\[
				\bigcap \{j[r^{-1}[B_i]]: z \in j[r^{-1}[B_i]]\} = j\left[\bigcap \{r^{-1}[B_i]: j^{-1}(z) \subseteq r^{-1}[B_i]\}\right] = j[[x]_{E'}] = [z]_{E''}.
			\]
		\end{clmproof}
		
		\begin{clm}	
			For every $i \in \omega$, the set $j[r^{-1}[B_i]]$ is an analytic subset of $(u\M)_S/H(u\M)$, i.e. it belongs to $\Souslin(\CLO((u\M)_S/H(u\M)))$, which implies that it is strictly Baire in $(u\M)_S/H(u\M)$.
		\end{clm}
		
		\begin{clmproof}[Proof of Claim 3]
			Fix any $i \in \omega$. Since $B_i$ is a Borel (and hence analytic) subset of the Polish space $Y_M$, we have that $B_i \in \Souslin(\CLO(Y_M))$. Since $Y_M$ is closed in $p(\C')_M$, we get that $B_i \in \Souslin(\CLO(p(\C')_M))$. By Lemma~\ref{lem: main technical lemma}, for $r$ considered on its whole original domain $u\M$ (and not like in the current proof only on $(u\M)_S$), we get that $j[r^{-1}[B_i]] \in \Souslin(\CLO(u\M/H(u\M)))$. By intersecting with $(u\M)_S/H(u\M)$, we easily conclude that for $r$ restricted to $(u\M)_S$, one has $j[r^{-1}[B_i]] \in \Souslin(\CLO((u\M)_S/H(u\M)))$. By the basic facts recalled in the second paragraph below Fact~\ref{fct:Miller}, we get that $j[r^{-1}[B_i]]$ is a strictly Baire subset of $(u\M)_S/H(u\M)$.
		\end{clmproof}
		
		The group $(u\M)_S/H(u\M)$ is a compact, Hausdorff group. Hence, by Fact~\ref{fct:Miller} and Claims 2(ii) and 3, we conclude that the kernel $\ker (\bar h_E)$ is a $\tau$-closed subgroup of $(u\M)_S/H(u\M)$, and so of $u\M/H(u\M)$ as well. Hence, by Theorem~\ref{thm:main theorem 3}, we get
		\[
			\{\bar\alpha /E\}=\bar h_E[\ker (\bar h_E)]=\bar h_E[\cl_\tau(\ker (\bar h_E))]=\cl(\bar \alpha/E).
		\]
		This means that $\{\bar \alpha/E\}$ is closed, so the class $[\bar \alpha]_E$ is type-definable, and the proof is complete.
	\end{proof}

	\section{Application to bounded, invariant equivalence relations in an arbitrary language}\label{section: arbitrary language}

	\subsection{Main theorem for arbitrary language}
	
	In the countable language case, Theorem~\ref{thm:main_Borel} implies that whenever the relation in question is
	also Borel, then its restriction to $Y \cap p(\C)$ is type-definable, or it has $2^{\aleph_0}$ classes. The main goal of this section is to prove Theorem~\ref{thm:nwg} which is a generalization of this statement to the case of an arbitrary language; in particular, it generalizes
	key corollaries of Newelski's theorem (namely, Corollaries~\ref{cor:nwcr} and~\ref{cor: Newelski's theorem}). Moreover, in Theorem~\ref{thm:nwg}, we obtain new information concerning relative definability. Theorems~\ref{thm:main_Borel} and~\ref{thm:nwg} will easily give us the trichotomy theorem in the Section~\ref{section: trichotomy} which explains very well relationships between smoothness, type-definability, relative definability and the number of classes of Borel, bounded equivalence relations
	(on the set of realizations of a complete type over $\emptyset$).
	
	\begin{thm}
		\label{thm:nwg}
		Suppose that $E$ is a bounded, invariant equivalence relation
		on an invariant set $X\supseteq p(\C)$ for a complete type $p$ over $\emptyset$. Assume that $E$ can be obtained from type-definable sets by the Souslin operation
		(i.e.\ $E$ is in $\Souslin(\textrm{type-definable})$), while $Y\subseteq p(\C)$ is type-definable (with parameters) and $E$-saturated. Then:
		\begin{enumerate}[label=(\Roman*)]
			\item
			\label{it:nwg_1}
			$E\restr_{p(\C)}$ is type-definable, or $E\restr_Y$ has at least $2^{\aleph_0}$ classes,
			\item
			\label{it:nwg_2}
			in addition, if $\aut(\C/\{Y\})$ acts transitively on $Y/E$ (e.g.\ $Y=p(\C)$ or $Y$ is a KP strong type), then either $E\restr_Y$ is relatively definable (so, by compactness, it has finitely many classes), or $E\restr_Y$ has at least $2^{\aleph_0}$ classes.
		\end{enumerate}
	\end{thm}
	
	Applying this theorem in the case when $X=Y=p(\C)$, we get
	
	\begin{cor}
		Let $E$ be a bounded, invariant equivalence relation on a single complete type $p$ over $\emptyset$, and assume that $E$ is in $\Souslin(\textrm{type-definable})$. Then either $E$ is relatively definable (and so it has finitely many classes), or it has at least $2^{\aleph_0}$ classes.
	\end{cor}
	
	Applying the theorem in the case when $Y$ is a KP type, we get

	\begin{cor}\label{cor:mainBu}
		Let $E$ be a bounded, invariant equivalence relation (on some $X$) refining $E_{KP}$, and assume that $E$ is in $\Souslin(\textrm{type-definable})$. Then, for any $a \in X$, either $E \restr_{[a]_{KP}}$ has only one class, or it has at least $2^{\aleph_0}$ classes.
	\end{cor}
	
	Similarly to Corollary~\ref{cor:group_borel} in the countable case, we obtain the following corollary (where the extra assumption of~\ref{it:nwg_2} is ``automatically'' satisfied). This is a significant strengthening of
	\cite[Corollary 3.37]{KaMi} -- we weaken the assumption that $H$ is $F_\sigma$, relativize to a subgroup, and rule out infinite
	indices below $2^{\aleph_0}$. 
	\begin{cor}
		\label{cor:group_nwg}
		Suppose that $H$ is a bounded index, invariant subgroup of a $\emptyset$-definable group $G$. Assume that $H$ is in $\Souslin(\textrm{type-definable})$, while $K\geq H$ is a type-definable subgroup of $G$. Then either $H$ is relatively definable in $K$ (in which case $[K:H]$ is finite), or $[K:H]\geq 2^{\aleph_0}$.
	\end{cor}
	\begin{proof}
		Similarly to Corollary~\ref{cor:group_borel}, we use the construction from Fact~\ref{fct:affine_sort}. We want to apply Theorem~\ref{thm:nwg} for $E:=E_{H,X}$ and $Y:=K\cdot x_0$. 

		By Remark~\ref{rem:grresplus}, $Y$ is type-definable (and clearly $E$-saturated). By Proposition~\ref{prop: Souslin and parameters} and Remark~\ref{remark: analyticity for groups}, we get that $E$ is in $\Souslin(\textrm{type-definable})$. Moreover, from the description of the automorphism groups, we see that the extra assumption of Theorem~\ref{thm:nwg}~\ref{it:nwg_2} is also satisfied. Therefore we can apply this theorem in our case. Then we only need to notice that the number of classes of $E_{H,X}\restr_{K\cdot x_0}$ is just $[K:H]$ and apply Remark~\ref{rem:grresplus}.
	\end{proof}

	The following lemma will be used in the proof of both parts of Theorem~\ref{thm:nwg}.
	\begin{lem}
		\label{lem:analytic_subgroup}
		Let $E$ be a bounded, invariant equivalence relation on the set of realizations of $p$ (where $p$ is a complete type over $\emptyset$), and let $\bar \alpha\in p(\C)$.
		
		Consider the group $H:=\ker \bar h_E\leq u\M/H(u\M)$, where $\bar h_E$ is defined 
		following Remark~\ref{rem: g_E is quotient}.
		
		If $E$ is in $\Souslin(\textrm{type-definable})$, then $H$ is in $\Souslin(\CLO(u\M/H(u\M)))$, and as such, it is strictly Baire in $u\M/H(u\M)$.
	\end{lem}
	\begin{proof}
		Put $X=p(\C')$, where $\C' \succ \C$ is the bigger monster model which was used to define $\bar h_E$ in Section~\ref{section: top dyn for aut(C)}.
		
		The assumption that $E$ is in $\Souslin(\textrm{type-definable})$ clearly implies that so is $[\bar \alpha]_E$. By Proposition~\ref{prop: Souslin and parameters}, this implies that for a small model $M$, $[\tp(\bar \alpha/M)]_{E^M}$ is in $\Souslin(\CLO(S(M)))$, and so it is in $\Souslin(\CLO(X_M))$.
		Therefore, by Lemma~\ref{lem: main technical lemma}, the set $j[r^{-1}[[\tp(\bar \alpha/M)]_{E^M}]]$ is in $\Souslin(\CLO(u\M/H(u\M)))$.
		
		This finishes the proof, once we notice that $j[r^{-1}[[\tp(\bar \alpha/M)]_{E^M}]]=H$. The last equality follows from the commutativity of the following diagram and surjectivity of $j$. Namely, by the commutativity of the diagram, $r^{-1}[[\tp(\bar \alpha/M)]_{E^M}] = j^{-1}[\ker(\bar h_E)] = j^{-1}[H]$, so by the surjectivity of $j$, $j[r^{-1}[[\tp(\bar \alpha/M)]_{E^M}]]=H$.
		\begin{figure}[H]
			\centering
			\begin{tikzcd}
				u\M\arrow[r,"j"]\arrow[d,"r"]\arrow[rrd,"h_E"] & (u\M)/H(u\M) \arrow[r,"\bar f"]\arrow[dr,"\bar h_E"]& \gal_L(T) \arrow[d,"g_E"] \\
				X_M\arrow[rr] & & X/E
			\end{tikzcd}
			
		\end{figure}
	\end{proof}
	
	Now, we can proceed with the proof of Theorem~\ref{thm:nwg}.
	
	\begin{proof}[Proof of Theorem~\ref{thm:nwg}]
		\label{pf:nwg}
		Let $\bar \alpha \in Y(\C)$. Let $\C' \succ \C$ be the monster model using which $\bar h_E$ is defined. We can assume that $X=p(\C')$. We take the notation from Section~\ref{section: top dyn for aut(C)}.
		
		\ref{it:nwg_1} The argument
		is reminiscent of the proof of Theorem~\ref{thm:main_Borel}, but instead of Fact~\ref{fct:Miller}, we intend to apply Corollary~\ref{cor: from Mycielski} with $H:=\ker (\bar h_E)$ and $G:=\cl_\tau(H)$. Note that $G$ is a closed subgroup of $u\M/H(u\M)$, so it is a compact, Hausdorff group.
		
		Suppose $E\;(=E\restr_{p(\C')})$ is not type-definable. Then $\{\bar{\alpha}/E\}$ is not closed by Remark~\ref{rem:type-definability_of_relations}, which implies, by Theorem~\ref{thm:main theorem 3}, that $H$ is not $\tau$-closed (and thus not $\tau$-open).
		By Lemma~\ref{lem:analytic_subgroup}, $H$ has the Baire property, so Corollary~\ref{cor: from Mycielski} gives us that $|G/H|\geq 2^{\aleph_0}$. By Theorem~\ref{thm:main theorem 3}, this completes the proof, as $\bar h_E$ induces a bijection between $G/H$ and $\cl(\bar\alpha/E)\subseteq Y/E$.
		
		\ref{it:nwg_2}
		Let us assume that $E\restr_Y$ has less than $2^{\aleph_0}$ classes. Then, by part~\ref{it:nwg_1}, it is type-definable. We need to prove that it is relatively definable.
		
		By Lemma~\ref{lem:analytic_subgroup}, $H:=\ker(\bar h_E)$ has the strict Baire property as a subgroup of $u\M /H(u\M)$. Now, put
		\[
			G=(u\M)_S/H(u\M) := \bar{f}^{-1}[\aut(\C'/\{Y\})/\autf_L(\C')]
		\]
		and $H'=G\cap H$.
		
		By Lemma~\ref{lem:lem_closed} and continuity of $\bar f$, $G$ is a closed subgroup of $u\M /H(u\M)$, and therefore a compact, Hausdorff group, while $H'$ is (strictly) Baire in $G$.
		
		Notice that by the assumption of (II) concerning transitivity of the action together with the surjectivity of $\bar f$, $\bar h_E$ induces a bijection between $G/H'$ and $Y/E$.
		Since we have assumed that $\lvert Y/E\rvert<2^{\aleph_0}$, we deduce from Corollary~\ref{cor: from Mycielski} that $H'$ is open in $G$. 
		Since $G$ is compact, this implies that, in fact, $[G:H']$ is finite, and so is $Y/E$. Using this and the fact that $E$ is type-definable, compactness gives us that all classes of $E\restr_Y$ are relatively definable, and so is $E\restr_Y$. 
\end{proof}

	Note that if $Y$ is relatively definable in $p(\C)$, then in~\ref{it:nwg_2} of the last theorem, if $E\restr_Y$ has less than $2^{\aleph_0}$ classes, we have that in fact $E\restr_{p(\C)}$ is relatively definable (because it has a relatively definable class and we can use Remark~\ref{rem:type-definability_of_relations}), but otherwise this need not be true (e.g.\ if $Y$ is a single $E_{KP}$-class and $E={E_{KP}}$, then trivially $Y/E$ is a singleton, but $E_{KP}$ need not be relatively definable on a single type).
	
	Notice also that the assumption that $\aut(\C/\{Y\})$ acts transitively on $Y/E$ is essential in~\ref{it:nwg_2}, which can be seen in the following example.
	
	\begin{ex}
		Consider $T=\Th(2^\omega,E_n)_{n\in \omega}$, where $E_n$ is equality on the $n$-th coordinate. In the monster model, we consider the relation $E$ which is the intersection of all the relations $E_n$.
		
		Then there is only one type in $S_1(\emptyset)$, and $\C/E$ is naturally homeomorphic to $2^\omega$. We can find an $E$-saturated set $Y$ such that $Y/E$ corresponds to
		a subset of $2^\omega$ consisting of a convergent sequence of pairwise distinct elements along with its limit, which is the only limit point of the set of elements of this sequence. Then $Y$ is type-definable and $Y/E$ is of cardinality $\aleph_0$, which clearly implies that $E\restr_Y$ is not relatively definable. Hence, $Y$ does not satisfy the conclusion of Theorem~\ref{thm:nwg}~\ref{it:nwg_2}.
	\end{ex}
	
	Notice that Theorem~\ref{thm:nwg} gives us an alternative proof of Corollary~\ref{cor:nwcr}.
	
	\begin{proof}[Proof of Corollary~\ref{cor:nwcr}]
		If $E$ is $F_\sigma$, it can be obtained from type-definable sets by the Souslin operation (trivially, with $K_\eta=K_{\eta(0)}$ depending only on the first term of any given $\eta \in \omega^{<\omega}$), so Theorem~\ref{thm:nwg} applies.
	\end{proof}

	\subsection{Further considerations}
	
	In Theorem~\ref{thm:nwg}, we generalized Corollary~\ref{cor:nwcr}.
	To obtain an alternative proof or a generalization of Newelski's theorem (i.e.\ Fact~\ref{fct:twN}), we would need to somehow recover the notion of diameter, which is lost in the present generality. Hence the question is:
	
	\begin{ques}
		\label{ques:diam}
		Can we somehow extend the notion of diameter of a class of an $F_\sigma$ equivalence relation in a way that would allow us to generalize Fact~\ref{fct:twN}, or can we at least use our techniques to give an alternative proof of Fact~\ref{fct:twN}?
	\end{ques}
	
	In a different direction, recall the notion of the sub-Vietoris topology introduced in \cite{KrRz}.
	
	\begin{dfn}
		Suppose $X$ is a topological space. Then by the {\em sub-Vietoris topology} we mean the topology on $\powerset(X)$ (i.e.\ on the family of all subsets of $X$), or on any subfamily of $\powerset(X)$, generated by subbasis of open sets of the form $\{A\subseteq X\sbmid A\cap F=\emptyset\}$ for $F\subseteq X$ closed.
	\end{dfn}
	
	This allows us to state the following conjecture. The motivation is similar to \cite{KMS} and \cite{KrRz}. Namely, we would like to ``generalize'' Theorem~\ref{thm:main_Borel} to arbitrary (possibly uncountable) languages, not only in the form of a dichotomy between type-definability of the relation in question and a big cardinality of the set of its classes (as was done in Theorem~\ref{thm:nwg}\ref{it:nwg_1}), but we would like to show that if the relation is not type-definable, then, in some sense, $\EZ$ reduces to it. The main difference between the conjecture and Theorem~\ref{thm:main_Borel} is that we choose sets instead of points.
	
	\begin{conj}
		\label{conj:nwg2}
		Suppose we have $E,p,X,Y$ as in Theorem~\ref{thm:nwg}.
		
		Then whenever $E\restr_{p(\C)}$ is not type-definable, we have that for some small model $M$ there is a homeomorphic embedding $\psi\fcolon 2^{\omega}\to \powerset(Y_M)$ (where $\powerset(Y_M)$ is equipped with the sub-Vietoris topology) such that for any $\eta,\eta'\in 2^{\omega}$:
		\begin{enumerate}
			\item
			$\psi(\eta)$ is a nonempty closed set,
			\item
			if $\eta,\eta'$ are $\EZ$-related, then $[\psi(\eta)]_{E^M}=[\psi(\eta')]_{E^M}$,
			\item
			if $\eta,\eta'$ are distinct, then $\psi(\eta)\cap\psi(\eta')=\emptyset$,
			\item
			if $\eta,\eta'$ are not $\EZ$-related, then $(\psi(\eta)\times \psi(\eta'))\cap E^M=\emptyset$.
		\end{enumerate}
	\end{conj}
	
	It should be noted that the conjecture, if true, immediately implies Theorem~\ref{thm:nwg}\ref{it:nwg_1}. Furthermore, it would be (essentially) a generalization of \cite[Theorem 3.18]{KrRz} -- which, in turn, is a generalization of \cite[Theorem 5.1]{KMS} (see also \cite[Theorems 2.19, 3.19]{KaMi}).

	As we will see in Proposition~\ref{prop:conj_converse}, the conclusion of the conjecture implies that the relation $E$ is \emph{not} type-definable, so if Conjecture~\ref{conj:nwg2} holds, it actually gives us an equivalent condition for type-definability, similarly to Theorem~\ref{thm:main_Borel}, but without any countability assumptions.
	
	To show this, we first prove the following topological lemma.
	
	\begin{lem}
		\label{lem:subv_closed_relation}
		Let $X$ be a compact, Hausdorff space. Suppose $E$ is a binary relation on $X$. Write $\overline E$ for the relation on $2^X$ (closed subsets of $X$) defined by
		\[
			K_1 \mathrel{\overline E} K_2 \iff \exists k_1\in K_1\exists k_2\in K_2\quad k_1\Er k_2
		\]
		Then, if $E$ is a closed relation, so is $\mathrel{\overline E}$ (on $2^X$ with the sub-Vietoris topology).
	\end{lem}
	\begin{proof}
		Choose an arbitrary net $(K_i,K'_i)_{i\in I}$ in $\mathrel{\overline E}$ converging to some $(K,K')$ in $2^X$. We need to show that $(K,K')\in \overline E$.
		
		Let $k_i\in K_i, k_i'\in K_i'$ be such that $k_i \Er k_i'$. By compactness, we can assume without loss of generality that $(k_i,k_i')$ converges to some $(k,k') \in E$ (as $E$ is closed). If $k\in K$ and $k'\in K'$, we are done.
		
		Let us assume towards contradiction that $k\notin K$. Then, since $K$ is closed, and $X$ is compact, Hausdorff (and thus regular), we can find disjoint open sets $U,V$ such that $K\subseteq U$ and $k\in V$. Then we can assume without loss of generality that all $k_i$ are in $V$ (passing to a subnet if necessary). We see that $F:=X\setminus U$ is a closed set such that $F\cap K=\emptyset$. But for all $i$ we have $k_i\in F\cap K_i$, which gives us a (sub-Vietoris) basic open set separating $K$ from all $K_i$, a contradiction; therefore, we must have $k\in K$.
		
		Similarly, it cannot be that $k'\notin K'$, which completes the proof.
	\end{proof}
	(In fact, the converse is also true, because the map $x\mapsto \{x\}$ is a homeomorphic embedding of $X$ into $2^X$ with the sub-Vietoris topology.)

	Without further ado, we can prove the aforementioned proposition.
	\begin{prop}
		\label{prop:conj_converse}
		The converse of Conjecture~\ref{conj:nwg2} holds. More precisely, if $E$ is a bounded, $\emptyset$-type-definable equivalence relation on an invariant set $X$, while $Y$ is a type-definable, $E$-saturated subset of $X$, then there is no function $\psi$ as in the conclusion of Conjecture~\ref{conj:nwg2}.
	\end{prop}
	\begin{proof}
		Suppose towards contradiction that we have such a function $\psi\fcolon 2^\omega \to \powerset(Y_M)$. Denote by $\mathcal F$ the image of $\psi$.
		
		Since $E$ is type-definable, $E^M$ is closed, and since $\mathcal F$ consists of closed sets, by Lemma~\ref{lem:subv_closed_relation}, the restriction $\overline{E^M}\restr_{\mathcal F}$ is a closed relation. On the other hand, by the properties of $\psi$,  for any $\eta_1,\eta_2 \in2^{\omega}$, $\eta_1 \EZ \eta_2 \iff \psi(\eta_1) \overline{\Er^M} \psi(\eta_2)$. Since $\psi$ is a homeomorphism from $2^{\omega}$ to ${\mathcal F}$, we conclude that $\EZ$ is a closed relation which is obviously not true.		 
%		However, by the properties of $\psi$, $\overline{E^M}\restr_{\mathcal F}$ is an equivalence relation on a Polish space to which $\EZ$ reduces (by $\psi$ itself). This is a contradiction, because that would give us a closed, non-smooth equivalence relation on a Polish space, which is impossible.
	\end{proof}
	
	The proposition below is a weak variant of Conjecture~\ref{conj:nwg2}: namely, we assume that $E$ is $F_\sigma$ and, in the conclusion, we replace ``for some model'' with ``for any model'', and in return, we drop the property that $\psi$ takes distinct points to disjoint sets (which would imply that it is a homeomorphism, by Fact~\ref{fct:subVt} below).
	
	It should be noted that a variant of Conjecture~\ref{conj:nwg2} with the same conclusion, but
	with the assumption strengthened to $E$ being an \emph{orbital} $F_\sigma$ equivalence relation, is more or less a restatement of \cite[Theorem 3.18]{KrRz}, so the main strength of the next proposition lies in that we drop the ``orbital'' part of the assumption. Moreover, perhaps the proof could shed some light on how to prove the full conjecture.
	\begin{prop}
		\label{prop:conj_weak}
		Suppose we have $E,p,X,Y$ as in Theorem~\ref{thm:nwg}, and suppose moreover that $E$ is $F_\sigma$.
		
		Then whenever $E\restr_{p(\C)}$ is not type-definable, we have that for any model $M$, there is a continuous function $\psi\fcolon 2^{\omega}\to \powerset(Y_M)$ (where $\powerset(Y_M)$ is equipped with the sub-Vietoris topology) such that for any $\eta,\eta'\in 2^{\omega}$:
		\begin{itemize}
			\item
			$\psi(\eta)$ is a nonempty closed set,
			\item
			if $\eta,\eta'$ are $\EZ$-related, then $[\psi(\eta)]_{E^M}=[\psi(\eta')]_{E^M}$,
			\item
			if $\eta,\eta'$ are not $\EZ$-related, then $(\psi(\eta)\times \psi(\eta'))\cap E^M=\emptyset$.
		\end{itemize}
	\end{prop}
	
	Before the proof we need to recall a few facts and make some observations. The descriptive set theoretic tools which we use to prove the proposition are similar to those from \cite{KMS} and \cite{KrRz}.
	
	Recall that the {\em strong Choquet game} on a topological space $X$ is the following two-player game in $\omega$-rounds. In round $n$, player A chooses an open set
	$U_n \subseteq V_{n-1}$ and $x_n \in U_n$, and player B responds by choosing an open set $V_n \subseteq U_n$
	containing $x_n$. Player B wins when the intersection $\bigcap \{V_n \sbmid n < \omega\}$ is nonempty. A topological space X is a {\em strong Choquet space} if player B has a winning strategy in the strong Choquet game on $X$. For more details see Sections 8.C and 8.D of \cite[Chapter I]{Ke}. It is easy to see that each nonempty, compact, Hausdorff space is strong Choquet. Given a subset $C$ of $X$, we say that $X$ is {\em strong Choquet over $C$} to mean that the points that player A chooses are taken from $C$ (and player $B$ has a winning strategy in the modified game). Clearly, a strong Choquet space is also strong Choquet over each of its subsets.
	
	Given $X$, $R \subseteq X \times X$, and $x \in X$, define $R_x := \{y \in X \sbmid x \mathrel{R} y \}$.
	
	The next fact is Theorem 2.5 from \cite{KMS} with a slightly extended conclusion (which is a part of the proof there). It was stated in this extended form in \cite[Theorem 3.14]{KrRz}.
	
	\begin{fct}
		\label{fct:dtmtoolu}
		Suppose that $X$ is a regular topological space, $\langle R_n\mid n\in \omega\rangle$ is a sequence of $F_\sigma$ subsets of $X^2$, $\Sigma$ is a group of	homeomorphisms of $X$, and $\mathcal O\subseteq X$ is an orbit of $\Sigma$ with the property that for all $n\in \omega$ and open sets $U\subseteq X$ intersecting $\mathcal O$, there are distinct $x,y\in\mathcal O\cap U$ with $\mathcal O\cap (R_n)_x\cap (R_n)_y=\emptyset$ .
%(where $(R_n)_x,(R_n)_y$ are sections of $R_n$ at $x,y$, respectively).
		If $X$ is strong Choquet over $\mathcal O$, then there is a function $\tilde\phi\fcolon 2^{<\omega}\to \powerset(X)$ such that for any $\eta\in 2^\omega$ and any $n\in \omega$:
		\begin{itemize}
			\item
			$\tilde\phi(\eta\restr n)$ is a nonempty open set,
			\item
			$\overline{\tilde\phi(\eta\restr{(n+1)})}\subseteq \tilde\phi(\eta\restr{n})$
		\end{itemize}
		Moreover, $\phi(\eta)=\bigcap_n \tilde\phi(\eta\restr n)=\bigcap_n \overline{\tilde\phi(\eta\restr n)}$ is a nonempty closed $G_\delta$ set such that for any $\eta,\eta'\in 2^\omega$ and $n\in\omega$:
		\begin{itemize}
			\item
			if $\eta \EZ \eta'$, then there is some $\sigma\in\Sigma$ such that $\sigma\cdot \phi(\eta)=\phi(\eta')$,
			\item
			if $\eta(n)\neq \eta'(n)$, then $(\phi(\eta)\times \phi(\eta'))\cap R_n=\emptyset$, and if $\eta,\eta'$ are not $\EZ$-related, then $(\phi(\eta)\times \phi(\eta'))\cap \bigcup R_n=\emptyset$.
		\end{itemize}
	\end{fct}

	\begin{fct}[{\cite[Proposition 3.16]{KrRz}}]
		\label{fct:subVt}
		Suppose $X$ is a normal topological space (e.g.\ a compact, Hausdorff space) and $\mathcal A$ is any family of pairwise disjoint, nonempty closed subsets of $X$. Then $\mathcal A$ is Hausdorff with the sub-Vietoris topology.
	\end{fct}
	
	Using the last two facts, we obtain a corollary reminiscent of \cite[Theorem 3.18]{KrRz} (albeit topological group theoretic, and not model theoretic in nature), which will be used in the proof of Proposition~\ref{prop:conj_weak}.
	
	\begin{cor}
		\label{cor:Vietoris_embed}
		Suppose $G$ is a compact, Hausdorff group, while $H\leq G$ is $F_\sigma$ and not closed. Then there is a homeomorphic embedding $\phi\fcolon 2^\omega\to \powerset(G)$ (with the sub-Vietoris topology) such that for any $\eta,\eta'\in 2^\omega$:
		\begin{itemize}
			\item
			$\phi(\eta)$ is a nonempty closed set,	
			\item
			if $\eta \EZ \eta'$, then there is some $h\in H$ such that $\phi(\eta)h=\phi(\eta')$,
			\item
			if $\eta\neq \eta'$, then $\phi(\eta)\cap \phi(\eta')=\emptyset$,
			\item
			if $\eta,\eta'$ are not $\EZ$-related, then $\phi(\eta)H\cap \phi(\eta')H=\emptyset$.
		\end{itemize}
		In particular, $[G:H]\geq 2^{\aleph_0}$.
	\end{cor}
	\begin{proof}
		We can assume without loss of generality that $H$ is dense in $G$ (by replacing $G$ with $\overline H$). Since $H$ has the Baire property (as an $F_\sigma$ subset of a compact space), by Pettis theorem (i.e.\ Fact~\ref{fct: Pettis}) it follows that $H$ is meager in $G$ (because $H$ is not closed, and so not open). Therefore, since $H$ is $F_\sigma$ and closed meager sets are nowhere dense, there are nonempty closed, nowhere dense sets $F_n\subseteq G$, $n \in \omega$, such that $H = \bigcup_n F_n$. We can assume without loss of generality that the $F_n$'s are symmetric (i.e.\ $F_n=F_n^{-1}$ and $e \in F_n$), increasing, and satisfy $F_nF_m\subseteq F_{n+m}$.
		
		$H$ acts by homeomorphisms on $G$ (by right translations by inverses). Let us denote by $R_n$ the preimage of $F_n$ by $(g_1,g_2)\mapsto g_1^{-1}g_2$. We intend to show that the assumptions of Fact~\ref{fct:dtmtoolu} are satisfied, with $X:=G$, $\mathcal O=\Sigma:=H$ and $R_n$ just defined.
		
		Since $G$ is compact Hausdorff, it is strong Choquet over $\mathcal O$ (even over itself) and regular. Fix any open set $U$ and any $n\in \omega$. Then pick any $h\in H\cap U$ (which exists by density). Then $h\in F_N$ for some $N\in \omega$.
		
		From the fact that $H$ is dense and the $F_m$'s are closed nowhere dense, it follows that for each $m$, $H\setminus F_m$ is dense, so we can find some $h'\in U\cap (H\setminus F_{2n+N})$. Since the $F_n$'s are increasing, we see that $h \ne h'$.
		Moreover, we have
		\[
		H\cap (R_n)_h\cap (R_n)_{h'}=H\cap hF_n \cap h'F_n \subseteq F_N F_n \cap h' F_n .
		\]
		But if this last set was nonempty, we would have $h'\in F_NF_nF_n^{-1}\subseteq F_{2n+N}$ -- which would contradict the choice of $h'$ -- so $H\cap (R_n)_h\cap(R_n)_{h'}=\emptyset$, and the assumptions of Fact~\ref{fct:dtmtoolu} are satisfied. This gives us the map $\phi$, which satisfies all the bullets, as well as the auxiliary map $\tilde{\phi}$. What is left is to show that $\phi$ is a homeomorphic embedding.
		
		$\phi$ is clearly injective by the third bullet, and by the preceding fact, the range of $\phi$ is a Hausdorff space, so we only need to show that it is continuous. To do that, consider a subbasic open set $U=\{F\sbmid F\cap K= \emptyset\}$, and notice that by compactness, $\phi(\eta)\in U$ iff $\overline{\tilde{\phi}(\eta\restr n)}\cap K= \emptyset$ for some $n$, which is an open condition about $\eta$.
	\end{proof}
	
	\begin{rem}
		\label{rem:subVcont}
		Consider a map $f\fcolon X\to Y$ between topological spaces and the induced image and preimage maps $\mathcal F\fcolon \powerset(X)\to \powerset(Y)$ and $\mathcal G\fcolon \powerset(Y)\to \powerset(X)$. Then:
		\begin{itemize}
			\item
			If $f$ is continuous, so is $\mathcal F$.
			\item
			If $f$ is closed, $\mathcal G$ is continuous.
		\end{itemize}
		In particular, if $f$ is continuous, $Y$ is Hausdorff and $X$ is compact, then both $\mathcal F$ and $\mathcal G$ are continuous.
	\end{rem}
	\begin{proof}
		For the first point, consider a subbasic open set $B=\{A \sbmid A\cap F=\emptyset \}\subseteq \powerset(Y)$. Then $\mathcal F^{-1}[B]=\{A \sbmid f[A]\cap F=\emptyset\}=\{A \sbmid A\cap f^{-1}[F]=\emptyset\}$ (this is because any $a\in A$ witnessing that $A$ is not in one of the sets will witness the same for the other). The third set is clearly open in $\powerset(X)$. The second point is analogous.
	\end{proof}

	\begin{proof}[Proof of Proposition~\ref{prop:conj_weak}]
		Choose $\bar \alpha \in Y(\C)$. As usual, we can assume that $X=p(\C')$. Put $H=\ker(\bar h_E)$ and $G=\cl_\tau(H)$, as in the proof of part (I) of Theorem~\ref{thm:nwg} (page \pageref{pf:nwg}). Since $E$ is $F_\sigma$, we see that $[\tp(\bar \alpha/M)]_{E^M}$ is $F_\sigma$ as well. Thus, by 
		%a similar argument to the proof of Lemma~\ref{lem: main technical lemma} and 
		Remark~\ref{remark: to use in 5.12} and the equality $H=j[r^{-1}[[\tp(\bar \alpha/M)]_{E^M}]$ (justified in the last paragraph of the proof of Lemma~\ref{lem:analytic_subgroup}), we obtain that $H$ is $F_\sigma$ in the compact, Hausdorff group $G$. Further, since $E\restr_{p(\C')}$ is not type-definable, by Theorem~\ref{thm:main theorem 3} and Remark~\ref{rem:type-definability_of_relations}, we get that $H$ is not closed, so Corollary~\ref{cor:Vietoris_embed} applies and gives us a function $\phi\fcolon 2^\omega\to \powerset(G)$ as there. 
		
		Now, we need to introduce some other functions which combined with $\phi$ will yield the desired function $\psi\fcolon 2^\omega \to \powerset(Y_M)$.
		
		By Lemma~\ref{lem: variant}, 
		we have the continuous surjection $j_{\cl}:=j \circ \zeta\fcolon \cl(u\M)\twoheadrightarrow u\M/H(u\M)$ (explicitly given by $ j_{\cl} (p)=up/H(u\M)$) and the function $r_{\cl}\fcolon \cl(u\M)\to X_M$ (which is the restriction of the function $\hat r \fcolon EL\to X_M$ to $\cl(u\M)$). 

		We know that the following diagram commutes.
		\begin{figure}[H]
			\centering
			\begin{tikzcd}
				\cl(u\M)\arrow[r,"\hat{f}"]\arrow[d,"r_{\cl}"]&\gal_L(T)\arrow[d,"g_E"] \\
				X_M\arrow[r] & X/E
			\end{tikzcd}
		\end{figure}
		Since for any $x \in \cl(u\M)$ we have $\hat{f}(x)=\hat{f}(u)\hat{f}(x)=\hat{f}(ux)$ and $(\bar h_E \circ j_{\cl})(x)=(g_E\circ \bar f\circ j_{\cl})(x)=(g_E\circ \bar f \circ j)(ux)= (g_E \circ\hat{f})(ux)$, we see that $(\bar h_E \circ j_{\cl})(x)=(g_E \circ\hat{f})(x)$. Hence, from the commutativity of the above diagram, we conclude that the following diagram also commutes.
		\begin{figure}[H]
			\centering
			\begin{tikzcd}
				\cl(u\M) \arrow[r,"j_{\cl}"]\arrow[d,"r_{\cl}"] & u\M/H(u\M)\arrow[d,"\bar h_E"] \\
				X_M \arrow[r,"\rho"]& X/E
			\end{tikzcd}
		\end{figure}
		Consider the function $\psi\fcolon 2^\omega\to \powerset(X_M)$ given by $\psi(\eta)=r_{\cl}[j_{\cl}^{-1}[\phi(\eta)]]$. By Remark~\ref{rem:subVcont}, it is continuous. Further, the image of $\psi$ consists of subsets of $Y_M$ -- it is a consequence of the commutativity of the diagram above, the equality $\bar{h}_E[G]=\cl(\bar \alpha/E)\subseteq Y/E$ (which is the content of Theorem~\ref{thm:main theorem 3}) and the fact that $Y$ is $E$-saturated. It remains to check that $\psi$ satisfies the three bullets. This is an easy exercise using commutativity of the above diagram, but we give the details.

		\begin{itemize}
			\item Each $\psi(\eta)$ is closed and nonempty because of the surjectivity of $j_{\cl}$, continuity of $j_{\cl}$ and $r_{\cl}$ and the fact that $\phi(\eta)$ is closed and nonempty. 
			\item Consider any $\eta \EZ \eta'$. Then $\phi(\eta)h=\phi(\eta')$ for some $h \in H$. Consider any $p \in \psi(\eta)$. The goal is to find $q \in \psi(\eta')$ which is $E^M$ related to $p$. There is $x \in \cl(u\M)$ such that $r_{\cl}(x)=p$ and $j_{\cl}(x) \in \phi(\eta)$. Then $j_{\cl}(x)h \in \phi(\eta')$. Since $j_{\cl}$ is surjective, there is $y \in \cl(u\M)$ such that $j_{\cl}(y) = j_{\cl}(x)h \in \phi(\eta')$. This implies that $q:=r_{\cl}(y) \in \psi(\eta')$. It remains to show that $p \mathrel{E^M} q$. Since $h \in H = \ker(\bar h_E)$, we get $\rho(q)=(\rho \circ r_{\cl})(y) = (\bar h_E \circ j_{\cl})(y) = \bar h_E ( j_{\cl}(x)h) = (\bar h_E \circ j_{\cl})(x) = (\rho \circ r_{\cl})(x) = \rho(p)$, which means that $p \mathrel{E^M} q$.
			\item Consider $\eta, \eta'$ which are not $\EZ$ related. Then $\phi(\eta)H \cap \phi(\eta')H = \emptyset$. Consider any $p \in \psi(\eta)$ and $q \in \psi(\eta')$. The goal is to show that $p$ is not $E^M$ related to $q$. Suppose for a contradiction that $p \mathrel {E^M} q$. There are $x, y \in \cl(u\M)$ such that $r_{\cl}(x) =p$, $r_{\cl}(y)=q$, $j_{\cl}(x) \in \phi(\eta)$, and $j_{\cl}(y) \in \phi(\eta')$. We conclude that $\bar h_E(j_{\cl}(x)) = (\rho \circ r_{\cl})(x) = \rho(p) = \rho(q) = ( \rho \circ r_{\cl} )(y) =  \bar h_E(j_{\cl}(y))$. This implies that  $j_{\cl}(y)^{-1} j_{\cl}(x) \in H$, and so $j_{\cl}(x) \in \phi(\eta) \cap \phi(\eta')H$, a contradiction. \qedhere
		\end{itemize}
	\end{proof}
	
	Notice that by the above proof and Lemma~\ref{lem:analytic_subgroup}, if in Corollary~\ref{cor:Vietoris_embed} we were able to weaken the assumption that $H$ is $F_\sigma$ to the one that it is only in $\Souslin(\CLO(G))$, then the same thing could be done in Proposition~\ref{prop:conj_weak} 
	(i.e.\ we could drop the assumption that $E$ is $F_\sigma$, leaving it simply in $\Souslin(\textrm{type-definable})$, as in Theorem~\ref{thm:nwg}).
	
	We reiterate that if we could weaken the assumption of Proposition~\ref{prop:conj_weak} as in the last paragraph and strengthen the conclusion to have that $\psi$ maps distinct points to disjoint sets (for some model $M$), then we would obtain Conjecture~\ref{conj:nwg2} -- the only part apparently missing would be the property that $\psi$ is a homeomorphic embedding, which would then be an easy consequence of Fact~\ref{fct:subVt}. It is possible that the property of mapping distinct points to disjoint sets could be attainable just by a careful choice of the model $M$.

	\section{Trichotomy theorem}\label{section: trichotomy}
	Here, we formulate our trichotomy mentioned in the abstract (but in a more general form), and we give a very short proof based on Theorems~\ref{thm:main_Borel} and~\ref{thm:nwg}. In order to get the trichotomy stated in the abstract, it is enough to apply Corollary~\ref{cor: elegant trichotomy} to $X=Y:=p(\C)$.

	Recall that a Borel (more generally, analytic) subset of any product of sorts is invariant by definition (see Definition~\ref{definition: Borel subsets of a model} and the paragraph after Fact~\ref{fct:eqcmp}).
	
	\begin{cor}\label{cor: elegant trichotomy}
		Assume that the language is countable. Let $E$ be a bounded, Borel (or, more generally, analytic) equivalence relation on an invariant set $X\supseteq p(\C)$ for a complete type $p$ over $\emptyset$, and let $Y\subseteq p(\C)$ be type-definable (with parameters) and $E$-saturated. Assume additionally that $\aut(\C/\{Y\})$ acts transitively on $Y/E$ (e.g.\ $Y=p(\C)$ or $Y$ is a KP strong type). Then exactly one of the following holds:
		\begin{itemize}[nosep]
			\item
			$E\restr_{Y}$ is relatively definable (on $Y$), smooth, and has finitely many classes,
			\item
			$E\restr_{Y}$ is not relatively definable, but it is type-definable, smooth, and has $2^{\aleph_0}$ classes,
			\item
			$E\restr_{Y}$ is not type definable and not smooth, and has $2^{\aleph_0}$ classes.
		\end{itemize}
	\end{cor}
	
	\begin{proof}
		This follows from Theorem~\ref{thm:main_Borel} and Theorem~\ref{thm:nwg}, after noting that since Borel sets in Polish spaces are analytic and so can be obtained by the Souslin operation applied to closed sets, the relation $E$ is in $\Souslin(\textrm{type-definable})$ (by Proposition~\ref{prop: Souslin and parameters}).
		
		Namely, if $E\restr_{Y}$ is relatively definable on $Y$, then it clearly has finitely many classes. Suppose $E\restr_{Y}$ is not relatively definable (on $Y$). Then, by Theorem~\ref{thm:nwg}(II) and the countability of the language, $E\restr_{Y}$ has exactly $2^{\aleph_0}$ classes. Finally, the equivalence of type-definability and smoothness of $E\restr_{Y}$ is provided by Theorem~\ref{thm:main_Borel}.
	\end{proof}
	
	This also immediately carries over to the definable group case.
	\begin{cor}
		\label{cor:group_trichotomy}
		Assume that the language is countable. Let $H$ be a bounded index, Borel  (or, more generally, analytic) subgroup of a $\emptyset$-definable group $G$, and suppose $K\geq H$ is a type-definable subgroup of $G$. Then exactly one of the following holds:
		\begin{itemize}[nosep]
			\item
			$H$ is relatively definable in $K$ and $[K:H]$ is finite,
			\item
			$H$ not relatively definable, but it is type-definable, $[K:H]=2^{\aleph_0}$ and $E_H$ is smooth,
			\item
			$H$ is not type-definable, $[K:H]=2^{\aleph_0}$ and $E_H$ is not smooth.
		\end{itemize}
	\end{cor}
	\begin{proof}
		It follows from Corollaries~\ref{cor:group_borel} and~\ref{cor:group_nwg}.
	\end{proof}

	It should be noted that if Conjecture~\ref{conj:nwg2} holds, the trichotomy can be extended to uncountable case as well (with the ``smoothness'' in the second branch of the trichotomy replaced by the nonexistence of a function such as the one in the conclusion of the conjecture, and with number of classes greater or equal to $2^{\aleph_0}$ in the second and third branch). This follows from Proposition~\ref{prop:conj_converse}.

%	\section*{Appendix: On the existence of a semigroup structure on the type space \texorpdfstring{${S_{\bar c}(\mathfrak{C})}$}{Sc(C)}}\label{section: semigroup operation}
\appendix
\section{On the existence of a semigroup structure on the type space \texorpdfstring{${S_{\bar c}(\mathfrak{C})}$}{Sc(C)}}\label{section: semigroup operation}

	%In Section~\ref{section: top dyn for aut(C)}, 
	When considering the dynamical system of $\aut(\C)$ acting on $S_{\bar c}(\C)$, we 
	heavily used the enveloping semigroup $EL=EL(S_{\bar c}(\mathfrak{C}))$. In this appendix, we show that the natural action of $\aut(\C)$ on $S_{\bar c}(\C)$ can be extended to a left continuous semigroup operation on $S_{\bar c}(\C)$ (which allows us to identify the semigroups $EL$ and $S_{\bar c}(\C)$) if and only if the underlying theory is stable.
%at least when $\C$ is a saturated model of a strongly inaccessible cardinality larger than $\lvert T\rvert$ (the actual condition we need for the argument is slightly weaker).

	The general idea is as follows. We establish the inclusion of $\aut(\C)$ in the type space $S_{\bar c}(\C)$ as a universal object in a certain category. This allows us to describe the existence of a semigroup operation on $S_{\bar c}(\C)$ in terms of a ``definability of types'' kind of statement, which in turn can be related to stability using a type counting argument.
	
	\begin{prop}\label{proposition: category C}
		Consider $\aut(\C)\subseteq S_{\bar c}(\C)$ given by $\sigma \mapsto \tp(\sigma(\bar c)/\C)$.
		Consider the category $\catg$ whose objects are maps $\aut(\C)\to K$ such that:
		\begin{itemize}
			\item
			$K$ is a compact, zero-dimensional, Hausdorff space,
			\item
			preimages of clopen sets in $K$ are relatively $\C$-definable in $\aut(\C)$, i.e.\ for each clopen $C$ there is a formula $\varphi(x,a)$ with $a$ from $\C$ such that $\sigma$ is in the preimage of $C$ if and only if $\models \varphi(\sigma(\bar c),a)$,
		\end{itemize}
		where morphisms are continuous maps between target spaces with the obvious commutativity property.
		Then the inclusion of $\aut(\C)$ into the space $S_{\bar c}(\C)$ is the initial object of $\catg$.
	\end{prop}
	\begin{proof}
		Firstly, $\aut(\C)$ is dense in $S_{\bar c}(\C)$, so the uniqueness part of the universal property is immediate. What is left to show is that for every $h\fcolon \aut(\C)\to K$, $h\in \catg$, we can find a continuous map $\bar h\fcolon S_{\bar c}(\C)\to K$ extending $h$. 
		
		Choose any $p\in S_{\bar c}(\C)$ and consider it as an ultrafilter 
		on relatively $\C$-definable subsets of $\aut(\C)$, and then consider $K_p:=\bigcap \left\{\overline{h[D]}\sbmid D\in p\right\}\subseteq K$.
		It is the intersection of a centered (i.e.\ with the finite intersection property) family of nonempty, closed subsets of $K$, so it is nonempty. In fact, it is a singleton.
		%otherwise, we could find a clopen set separating two points in the intersection, and its (relatively definable) preimage in $\aut(\C)$ would contradict the fact that $p$ is complete. 
		If not, there are two distinct elements $k_1,k_2 \in K_p$. Take a clopen neighborhood $U$ of $k_1$ such that $k_2 \notin U$. Since $h \in \catg$, $h^{-1}[U]= \{\sigma \in \aut(\C) \sbmid \;\models \varphi(\sigma(\bar c),a)\}$ for some formula $\varphi(\bar x, a)$. If $\varphi(\bar x, a) \in p$, then $K_p \subseteq \overline{h[h^{-1}[U]]} \subseteq \overline{U}=U$, a contradiction as $k_2 \notin U$. If  $\neg \varphi(\bar x, a) \in p$, then $K_p \subseteq \overline{h[\aut(\C) \setminus h^{-1}[U]]} \subseteq \overline{K \setminus U}=K \setminus U$, a contradiction as $k_1 \notin K \setminus U$.
		%Let the point in the intersection be the value of $\bar h$ at $p$.
		In conclusion, we can define $\bar h(p)$ to be the unique point in $K_p$.
		%		
		%		By the definition of $\bar h$, 
		
		We see that $\bar h$ extends $h$. Moreover, $\bar h$ is continuous, because the preimage of a clopen set $C\subseteq K$ is the basic open set in $S_{\bar c}(\C)$ corresponding to the relatively definable set $h^{-1}[C]$.
	\end{proof}
	
	In Corollary~\ref{cor: semigroup iff definability of types}, we will establish the aforementioned ``definability of types''-like condition from the existence of a semigroup operation. 
	For this we will need the following definition.

%\begin{dfn}
%Let $M$ be a model (e.g. $M=\C$). A type $q(x) \in S(M)$ is {\em piecewise weakly definable} if for any formula $\varphi(x,y)$ and type $p(y)\in S(\emptyset)$ there is a finitary type $\bar p(yz) \in S(\emptyset)$ extending $p(y)$ such that the set of $ab \in \bar p(M)$ for which $q\vdash \varphi(x,a)$ is relatively $M$-definable in $\bar p(M)$ (that is, there is a formula $\delta(yz,c)$ (with $c$ from $M$) such that for any $ab\in \bar p(M)$ we have $q\vdash \varphi(x,a)$ if and only if $\delta(ab,c)$).
%\end{dfn}

%The next remark is an obvious restatement of the above definition.

%\begin{rem}
%A type $q(x) \in S(M)$ is piecewise weakly definable if for every $\varphi(x,y)$ and every $a$ from $M$ there is some $b$ from $\C$ such that the set of all $a'b'$ form $M$ with $a'b'\equiv ab$ and $q\vdash \varphi(x,a')$ is relatively definable over $M$ (among all $a'b'$ from $M$ equivalent to $ab$).
%\end{rem}

	\begin{dfn}
		Let $M$ be a model (e.g. $M=\C$). A type $q(x) \in S(M)$ is {\em piecewise definable} if for every type $p(y)\in S(\emptyset)$ and formula $\varphi(x,y)$ the set of $a \in p(M)$ for which $q\vdash \varphi(x,a)$ is relatively $M$-definable in $p(M)$ (that is, there is a formula $\delta(y,c)$ (with $c$ from $M$) such that for any $a\in p(M)$ we have $q\vdash \varphi(x,a)$ if and only if $\delta(a,c)$).
		%A type $q(x) \in S(\C)$ is {\em piecewise weakly definable} if for every type $p(y)\in S(\emptyset)$ and formula $\varphi(x,y)$ there is a type $\bar p(yz) \in S(\emptyset)$ extending $p(y)$ such that the set of $ab \models \bar p$ for which $q\vdash \varphi(x,a)$ is relatively $\C$-definable in $\bar p(\C)$ (that is, there is a formula $\delta(yz,c)$ (with $c$ from $\C$) such that for any $ab \models \bar p$ we have $q\vdash \varphi(x,a)$ if and only if $\delta(ab,c)$).
	\end{dfn}

%The next remark is an obvious restatement of the second item of the above definition.
%\begin{rem}
%A type $q(x) \in S(\C)$ is piecewise weakly definable if and only if for every $\varphi(x,y)$ and every $a$ from $\C$ there is some $b$ from $\C$ such that the set of all $a'b'$ from $\C$ with $a'b'\equiv ab$ and $q\vdash \varphi(x,a')$ is relatively definable over $\C$ (among all $a'b'$ from $\C$ equivalent to $ab$).
%\end{rem}

	\begin{rem}\label{remark: piecewise weak definability equals piecewise definability}
		%A type $q(x) \in S(\C)$ is piecewise weakly definable  if and only if it is piecewise definable.
		Let  $q(x) \in S(\C)$. The following conditions are equivalent.
		\begin{enumerate}
			\item  $q(x)$ is piecewise definable.
			\item  For every type $p(y)\in S(\emptyset)$ and formula $\varphi(x,y)$ there is a type $\bar p(yz) \in S(\emptyset)$ extending $p(y)$ such that the set of $ab \models \bar p$ for which $q\vdash \varphi(x,a)$ is relatively $\C$-definable in $\bar p(\C)$ (that is, there is a formula $\delta(yz,c)$ (with $c$ from $\C$) such that for any $ab \models \bar p$ we have $q\vdash \varphi(x,a)$ if and only if $\delta(ab,c)$).
			\item For every $\varphi(x,y)$ and every $a$ from $\C$ there is some $b$ from $\C$ such that the set of all $a'b'$ from $\C$ with $a'b'\equiv ab$ and $q\vdash \varphi(x,a')$ is relatively definable over $\C$ (among all $a'b'$ from $\C$ equivalent to $ab$).
		\end{enumerate}
	\end{rem}

	\begin{proof}
		The equivalence $(2) \Leftrightarrow (3)$ is obvious. It is also clear that $(1) \Rightarrow (2)$, by taking $z$ and $b$ to be empty in (2). It remains to prove $(2) \Rightarrow (1)$.
		%The implication $(\Leftarrow)$ is clear by taking $z$ and $b$ empty in the definition of piecewise weak definability.
		
		Take any type $p(y) \in S(\emptyset)$ and formula $\varphi(x,y)$. By (2), there is a type $\bar p(yz) \in S(\emptyset)$ extending $p(y)$ and a formula $\delta(yz,c)$ (with $c$ from $\C$) such that for any $ab \models \bar p$ we have $q\vdash \varphi(x,a)$ if and only if $\delta(ab,c)$. This implies that for any $a,b,b'$ such that $ab \models \bar p$ and $ab' \models \bar p$ we have $\delta(ab,c) \leftrightarrow \delta(ab',c)$. By compactness, there is a formula $\psi(y,z) \in \bar p(yz)$ such that for any $a,b,b'$ with $ab \models \psi(y,z)$ and $ab' \models \psi(y,z)$ we have $\delta(ab,c) \leftrightarrow \delta(ab',c)$. Put 
		$$\delta'(y,c) : = (\exists z) (\psi(y,z) \wedge \delta(yz,c)).$$
		It remains to check that for any $a \models p$, $q\vdash \varphi(x,a)$ if and only if $\delta'(a,c)$. 
		
		First, assume $q\vdash \varphi(x,a)$ and $a \models p$. Take $b$ such that $ab \models \bar p$. Then $\psi(a,b) \wedge \delta(ab,c)$, and so  $\delta'(a,c)$. 
		
		Now, assume that  $\delta'(a,c)$ and $a \models p$. Then there is $b$ such that $\psi(a,b)$ and $\delta(ab,c)$. There is also $b'$ with $ab' \models \bar p$, and then $\psi(a,b')$. By the choice of $\psi$, we conclude that $\delta(ab',c)$. Hence,  $q\vdash \varphi(x,a)$.
	\end{proof}

%\begin{cor}\label{cor: semigroup iff definability of types}
%The natural action $\aut(\C)\times S_{\bar c}(\C)\to S_{\bar c}(\C)$ extends to a left-continuous semigroup operation on $S_{\bar c}(\C)$ if and only if each type over $\C$ is piecewise weakly definable, i.e.\ for any finitary type $p(y)\in S(\emptyset)$ and any $q(x)\in S(\C)$ and a formula $\varphi(x,y)$, there is a complete, finitary type $\bar p(yz)$ over $\emptyset$ extending $p$ such that the set of $ab\models \bar p$ for which $q\vdash \varphi(x,a)$ is relatively $\C$-definable in $\bar p(\C)$ (that is, there is a formula $\delta(yz,c)$ (with $c$ from $\C$) such that for any $ab \models \bar p(yz)$ we have $q\vdash \varphi(x,a)$ if and only if $\delta(ab,c)$).
		
%This last condition can be restated as follows: for any $q,\varphi(x,y)$ and any $a$ from $\C$ there is some $b$ from $\C$ such that the set of all $a'b'\equiv ab$ with $q\vdash \varphi(x,a')$ is relatively definable (among all $a'b'$ from $\C$ equivalent to $ab$).
%\end{cor}

	\begin{cor}\label{cor: semigroup iff definability of types}
		The natural action $\aut(\C)\times S_{\bar c}(\C)\to S_{\bar c}(\C)$ extends to a left-continuous semigroup operation on $S_{\bar c}(\C)$ if and only if each complete type over $\C$ is piecewise definable.
	\end{cor}

	\begin{proof}
		First, using Proposition~\ref{proposition: category C}, we will easily deduce 
		\begin{clm*}
			The action $\aut(\C)\times S_{\bar c}(\C)\to S_{\bar c}(\C)$ extends to a left-continuous semigroup operation on $S_{\bar c}(\C)$ if and only if for each $q\in S_{\bar c}(\C)$ the mapping $h_q \fcolon \aut(\C)\to S_{\bar c}(\C)$ given by $\sigma \mapsto \sigma(q)$ is in the category $\catg$ (i.e.\ the preimages of clopen sets are relatively $\C$-definable).
		\end{clm*}

		\begin{clmproof}[Proof of claim]
			$(\Rightarrow)$ Let $*$ be a left-continuous semigroup operation on $S_{\bar c}(\C)$ extending the action of $\aut(\C)$. Consider any $q \in S_{\bar c}(\C)$. Define $\bar h_q \fcolon S_{\bar c}(\C) \to S_{\bar c}(\C)$ by $\bar h_q(p) := p*q$. Then $\bar h_q$ is a continuous extension of $h_q$.  By continuity, the preimages of clopen sets by $\bar h_q$ are clopen, and therefore their intersections with $\aut(\C)$ (which are exactly the preimages of clopen sets by the original map $h_q$) are relatively $\C$-definable.
			
			$(\Leftarrow)$ By Proposition~\ref{proposition: category C}, for any $q \in S_{\bar c}(\C)$ there exists a continuous function $\bar h_q \fcolon S_{\bar c}(\C) \to S_{\bar c}(\C)$ which extends $h_q$. For $p,q \in S_{\bar c}(\C)$ define $p*q := \bar h_q(p)$. It is clear that $*$ (treated as a two-variable function) is left continuous and extends the action of $\aut(\C)$ on $S_{\bar c}(\C)$. We leave as a standard exercise on limits of nets to check that $*$ is also associative.
		\end{clmproof}
		
		By the claim and Remark~\ref{remark: piecewise weak definability equals piecewise definability}, the whole proof boils down to showing that for any type $q(x) \in S_{\bar c}(\C)$ we have the following equivalence: the preimage by $h_q$ of any clopen subset of $S_{\bar c}(\C)$ is relatively definable in $\aut(\C)$ if and only if $q(x)$ satisfies item (3) of Remark~\ref{remark: piecewise weak definability equals piecewise definability}.

		Let us fix an arbitrary $q(x) \in S_{\bar c}(\C)$, and any formula $\varphi(x,a)$ for some
		$a$ from $\C$. The preimage by $h_q$ of the clopen set $[\varphi(x,a)]$  equals
		\[
			\{\sigma \in \aut(\C)\sbmid \sigma(q)\vdash \varphi(x,a)\}=\{\sigma \in \aut(\C)\sbmid q\vdash \varphi(x,\sigma^{-1}(a))\}.
		\]
		
		Now, for $(\Leftarrow)$, suppose there is some $b$ from $\C$ and a formula $\delta(yz,c)$ (with $c$ from $\C$) such that for any $a'b'\equiv ab$ we have $q\vdash \varphi(x,a')$ if and only if $\models \delta(a'b',c)$. Then (taking $a'b'=\sigma^{-1}(ab)$) we have that
		\[
			q\vdash \varphi(x,\sigma^{-1}(a))\iff \models \delta(\sigma^{-1}(ab),c)\iff \models \delta(ab,\sigma(c)),
		\]
		and the last statement is clearly relatively $\C$-definable about $\sigma$.
	
		For $(\Rightarrow)$, suppose $\{\sigma\in \aut(\C)\sbmid q\vdash \varphi(x,\sigma^{-1}(a))\}$ is defined by some formula $\delta$, i.e.\ for some $d,c$ from $\C$, for any $\sigma \in \aut(\C)$ we have
		\[
			q\vdash \varphi(x,\sigma^{-1}(a))\iff \models\delta(d,\sigma(c)) \iff \models \delta(\sigma^{-1}(d),c).
		\]
		We can assume without loss of generality that $d=ab$ for some $b$ from $\C$ (adding dummy variables to $\delta$ if necessary). But then, for $a'b'\equiv ab$ there is some automorphism $\sigma$ such that $\sigma(a'b')=ab$, so we have
		\[
			q\vdash \varphi(x,a')\iff \models \delta(a'b',c).\qedhere
		\]
	\end{proof}
	
	This easily implies that stability is sufficient for the existence of a semigroup structure.
	
	\begin{cor}\label{corollary: stability implies semigroup}
		If $T$ is stable, then $S_{\bar c}(\C)$ has a left-continuous semigroup operation extending the action of $\aut(\C)$ on $S_{\bar c}(\C)$
	\end{cor}
	\begin{proof}
		If $T$ is stable, then every type over $\C$ is definable, so in particular it is piecewise definable, which by Corollary~\ref{cor: semigroup iff definability of types} implies that the semigroup structure exists.
	\end{proof}
	
	For the other direction, we will use Corollary~\ref{cor: semigroup iff definability of types} and an easy counting argument. But before that we need to establish a transfer property for piecewise definability.
	
	\begin{prop}\label{proposition: transfer of piecewise definability}
		Suppose that each complete type over $\C$ is piecewise definable. Then each complete type over any model $M$ of cardinality less then $\kappa$ (where $\kappa$ is the degree of saturation of $\C$) is piecewise definable. 
	\end{prop}
	
	\begin{proof}
		Take any $q(x) \in S(M)$. 
		%Suppose for a contradiction that it is not piecewise definable. 
		Consider any type $p(y) \in S(\emptyset)$ and formula $\varphi(x,y)$.
		Take a coheir extension $\bar q \in S(\C)$ of $q$. Then $q$ is invariant over $M$. By assumption, $\bar q$ is piecewise definable. So there is a formula $\delta(y,c)$ (with $c$ from $\C$) such that for any $a \in p(\C)$, $\bar q\vdash \varphi(x,a)$ if and only if $\delta(a,c)$. Denote by $A$ the set of all $a \in p(\C)$ satisfying these equivalent conditions; so $A$ is a relatively definable subset of $p(\C)$. By the invariance of $\bar q$ over $M$, we see that $A$ is invariant over $M$, and so, by $\kappa$-saturation and strong $\kappa$-homogeneity of $\C$, the subset $A$ of $p(\C)$ is relatively definable over $M$. In other words, there is a formula $\delta'(y,m)$ (with $m$ from $M$) such that for any  $a \in p(\C)$, $\bar q\vdash \varphi(x,a)$ if and only if $\delta'(a,m)$. Hence, for any  $a \in p(M)$, $q\vdash \varphi(x,a)$ if and only if $\delta'(a,m)$.
	\end{proof}
	
	\begin{cor}\label{cor: semigroup = stability}
	%Suppose $\C$ is saturated (in its own cardinality) and that $\lvert \C\rvert^{2^{\lvert T\rvert}}=\lvert \C\rvert$. Then 
	$S_{\bar c}(\C)$ has a left-continuous semigroup operation extending the action of $\aut(\C)$ on $S_{\bar c}(\C)$ if and only if $T$ is stable.
	\end{cor}
	\begin{proof}
		The ``if'' part is the content of Corollary~\ref{corollary: stability implies semigroup}.	
		
		$(\Rightarrow)$ Assume $S_{\bar c}(\C)$ has a left-continuous semigroup operation extending the action of $\aut(\C)$ on $S_{\bar c}(\C)$. Then, by Corollary~\ref{cor: semigroup iff definability of types}, all complete types over $\C$ are piecewise definable.  We will show that this implies that $T$ is $\beth_2(|T|)$-stable (where $\beth_2(|T|): = 2^{2^{|T|}}$). Consider any $M \models T$ of cardinality at most $\beth_2(|T|)$. We need to show that $|S_1(M)| \leq \beth_2(|T|)$. For this it is enough to prove that for any $\varphi(x,y)$ (where $x$ is a single variable) $|S_\varphi(M)| \leq  \beth_2(|T|)$. Without loss of generality $M \prec \C$.  
		By Proposition~\ref{proposition: transfer of piecewise definability}, each complete type over $M$ is piecewise definable. This implies that each type $q \in S_\varphi(M)$ is determined by a function $S_y(\emptyset) \to \lang(M)$ which takes $p(y)$ to $\delta(y,c)$ witnessing piecewise definability of $q$  (or, more precisely, of an arbitrarily chosen extension of $q$ to a type in $S_1(M)$) for the formula $\varphi(x,y)$. So $|S_\varphi(M)| \leq |\lang(M)|^{|S_y(\emptyset)|} \leq (\beth_2(|T|))^{2^{|T|}}=\beth_2(|T|)$.
	\end{proof}
	
	It is well known that if $T$ is stable, then it is $2^{|T|}$-stable. The reason why we worked with $\beth_2(|T|)$ in the above proof is that this is the ``degree'' of stability which we can deduce directly from piecewise definability. Then, knowing that $T$ is stable, we have the usual definability of types which implies $2^{|T|}$-stability.

	\section*{Acknowledgments}
	The first author would like to thank Maciej Malicki for
	drawing his attention to Fact~\ref{fct:Miller} which turned out to be useful in the proof of Theorem~\ref{thm:main_Borel}.
	
	The authors are grateful to the referee for very careful reading and all the comments and suggestions which helped us to improve presentation.
	
	\printbibliography

\end{document}